\documentclass[12pt,a4paper]{article}
\usepackage[utf8]{inputenc}
\usepackage[T1]{fontenc}
\usepackage{lmodern}
\usepackage{amsmath,amsthm,amsfonts,amssymb}
\usepackage{graphicx}
\usepackage{numprint}
\usepackage{enumerate}
\usepackage[margin=2.8cm]{geometry} 
\usepackage{comment} 
\usepackage{cancel}
\usepackage{tikz}
\usepackage{caption}

\newcommand{\todo}

\newtheorem{theorem}{Theorem}
\newtheorem*{maintheorem}{Main theorem}
\newtheorem{theorem-definition}[theorem]{Theorem-Definition}
\newtheorem{corollary}[theorem]{Corollary}
\newtheorem{proposition}[theorem]{Proposition}
\newtheorem{lemma}[theorem]{Lemma}
\newtheorem{definition}[theorem]{Definition}
\newtheorem{remark}[theorem]{Remark}
\newtheorem{fact}[theorem]{Fact}
\newtheorem{example}[theorem]{Example}

\begin{document}

\newcommand{\R}{\mathbb{R}}
\newcommand{\N}{\mathbb{N}}
\newcommand{\C}{\mathbb{C}}
\newcommand{\p}{\partial}
\newcommand{\euc}{\operatorname{euc}}
\newcommand{\cont}{\operatorname{cont}}
\newcommand{\inj}{\operatorname{inj}}
\newcommand{\vol}{\operatorname{vol}}
\newcommand{\diam}{\operatorname{diam}}
\newcommand{\K}{\operatorname{K}}
\newcommand{\dist}{\operatorname{d}}
\newcommand{\interior}{\operatorname{int}}
\newcommand{\module}{\operatorname{mod}}
\newcommand{\area}{\mathcal{A}rea}
\newcommand{\M}{\mathcal{M}}

\title{
    \textsc{A compactness theorem for surfaces with Bounded Integral Curvature}    
}
\newpage
\author{Clément Debin\thanks{This research is supported by the ERC Advanced Grant 320939, Geometry and Topology of Open Manifolds (GETOM)}}

\setcounter{page}{1}

\maketitle{}

\normalsize{\begin{abstract}
We prove a compactness theorem for metrics with Bounded Integral Curvature on a fixed closed surface $\Sigma$. As a corollary, we obtain a compactification of the space of Riemannian metrics with conical singularities, where an accumulation of singularities is allowed.
\end{abstract}}


\section*{Introduction and statement of the Main theorem}

The aim of this article is to compactify the space of metrics with conical singularities, on a fixed compact surface $\Sigma$. These metrics are Riemannian everywhere but at a finite number of points, where they look like an Euclidean cone; see section \ref{sectiondefinitionsingulariteconiques} for a precise definition. Since we allow cone points to accumulate, we need to define metrics with conical singularities ``along a curve'', or along a more complicated set, a Cantor set for example. In the case of flats metrics with conical singularities (we can think of polyhedra), the curvature is, in some sense to be made precise, concentrated at the cone points. Hence we need to understand what is a metric with curvature concentrated along a Cantor set.

In the late 1940's, A. Alexandrov and the school of Leningrad developed a very rich theory of singular surfaces. These are smooth surfaces, endowed with intrinsic metrics, for which there exists a natural notion of curvature, which is a Radon measure. They are called surfaces (respectively, metrics) with \emph{Bounded Integral Curvature}, denoted by ``B.I.C.`` in the sequel. The precise definition is given in section \ref{sectiondefinition}. For an exposition of the theory, see the book of A. Alexandrov and V. Zalgaller \cite{AZ}, the book of Y. Reshetnyak \cite{Reshetnyak_livre}, its article \cite{Reshetnyak_article} or the modern concise survey of M. Troyanov \cite{Troyanov_alexandrovsurfaces}.

The curvature measure is a fundamental object in this singular geometry. Its construction generalizes the Gauss-Bonnet formula: roughly speaking, the curvature of a geodesic triangle $ABC$ is $\overline{a}+\overline{b}+\overline{c}-\pi$, where $\overline{a},\overline{b}$ and $\overline{c}$ are the (upper) angles at $A,B$ and $C$ (see section \ref{sectiongeometricdefinition}). This theory includes smooth Riemannian metrics: in this case, the curvature measure is $\K_g d\mathcal{A}_g$, where $\K_g$ stands for the Gauss curvature, as well as metrics with conical singularities, where the curvature measure is $\K_g d\mathcal{A}_g$ + a sum of Dirac masses at the cone points ($\K_g$ is the Gauss curvature of the smooth part). Alexandrov surfaces of curvature bounded by above (the "CBA-spaces") or bounded by below (the "CBB-spaces") are also surfaces with B.I.C. 
The next example shows how a sequence of metrics with conical singularities can converge to a surface with B.I.C.:
\begin{center}
\begin{tikzpicture}[scale=1]
\draw ({-3+cos(180+10)},{0.5*sin(180+10)}) -- ({-3+cos(120+10)},{0.5*sin(120+10)});
\draw ({-3+cos(120+10)},{0.5*sin(120+10)}) -- ({-3+cos(60+10)},{0.5*sin(60+10)});
\draw ({-3+cos(60+10)},{0.5*sin(60+10)}) -- ({-3+cos(0+10)},{0.5*sin(0+10)});
\draw ({-3+cos(-180+10)},{0.5*sin(-180+10)}) -- ({-3+cos(-120+10)},{0.5*sin(-120+10)});
\draw ({-3+cos(-120+10)},{0.5*sin(-120+10)}) -- ({-3+cos(-60+10)},{0.5*sin(-60+10)});
\draw ({-3+cos(-60+10)},{0.5*sin(-60+10)}) -- ({-3+cos(-0+10)},{0.5*sin(-0+10)});
\draw [dashed] ({-3+cos(180+10)},{-2+0.5*sin(180+10)}) -- ({-3+cos(120+10)},{-2+0.5*sin(120+10)});
\draw [dashed] ({-3+cos(120+10)},{-2+0.5*sin(120+10)}) -- ({-3+cos(60+10)},{-2+0.5*sin(60+10)});
\draw [dashed] ({-3+cos(60+10)},{-2+0.5*sin(60+10)}) -- ({-3+cos(0+10)},{-2+0.5*sin(0+10)});
\draw ({-3+cos(-180+10)},{-2+0.5*sin(-180+10)}) -- ({-3+cos(-120+10)},{-2+0.5*sin(-120+10)});
\draw ({-3+cos(-120+10)},{-2+0.5*sin(-120+10)}) -- ({-3+cos(-60+10)},{-2+0.5*sin(-60+10)});
\draw ({-3+cos(-60+10)},{-2+0.5*sin(-60+10)}) -- ({-3+cos(-0+10)},{-2+0.5*sin(-0+10)});
\draw ({-3+cos(180+10)},{0.5*sin(180+10)}) -- ({-3+cos(180+10)},{-2+0.5*sin(180+10)});
\draw [dashed] ({-3+cos(120+10)},{0.5*sin(120+10)}) -- ({-3+cos(120+10)},{-2+0.5*sin(120+10)});
\draw [dashed] ({-3+cos(60+10)},{0.5*sin(60+10)}) -- ({-3+cos(60+10)},{-2+0.5*sin(60+10)});
\draw ({-3+cos(0+10)},{0.5*sin(0+10)}) -- ({-3+cos(0+10)},{-2+0.5*sin(0+10)});
\draw ({-3+cos(-60+10)},{0.5*sin(-60+10)}) -- ({-3+cos(-60+10)},{-2+0.5*sin(-60+10)});
\draw ({-3+cos(-120+10)},{0.5*sin(-120+10)}) -- ({-3+cos(-120+10)},{-2+0.5*sin(-120+10)});

\draw ({cos(180+10)},{0.5*sin(180+10)}) -- ({cos(135+10)},{0.5*sin(135+10)});
\draw ({cos(135+10)},{0.5*sin(135+10)}) -- ({cos(90+10)},{0.5*sin(90+10)});
\draw ({cos(90+10)},{0.5*sin(90+10)}) -- ({cos(45+10)},{0.5*sin(45+10)});
\draw ({cos(45+10)},{0.5*sin(45+10)}) -- ({cos(0+10)},{0.5*sin(0+10)});
\draw ({cos(0+10)},{0.5*sin(0+10)}) -- ({cos(-45+10)},{0.5*sin(-45+10)});
\draw ({cos(-45+10)},{0.5*sin(-45+10)}) -- ({cos(-90+10)},{0.5*sin(-90+10)});
\draw ({cos(-90+10)},{0.5*sin(-90+10)}) -- ({cos(-135+10)},{0.5*sin(-135+10)});
\draw ({cos(-135+10)},{0.5*sin(-135+10)}) -- ({cos(-180+10)},{0.5*sin(-180+10)});
\draw [dashed] ({cos(180+10)},{-2+0.5*sin(180+10)}) -- ({cos(135+10)},{-2+0.5*sin(135+10)});
\draw [dashed] ({cos(135+10)},{-2+0.5*sin(135+10)}) -- ({cos(90+10)},{-2+0.5*sin(90+10)});
\draw [dashed] ({cos(90+10)},{-2+0.5*sin(90+10)}) -- ({cos(45+10)},{-2+0.5*sin(45+10)});
\draw [dashed] ({cos(45+10)},{-2+0.5*sin(45+10)}) -- ({cos(0+10)},{-2+0.5*sin(0+10)});
\draw ({cos(0+10)},{-2+0.5*sin(0+10)}) -- ({cos(-45+10)},{-2+0.5*sin(-45+10)});
\draw ({cos(-45+10)},{-2+0.5*sin(-45+10)}) -- ({cos(-90+10)},{-2+0.5*sin(-90+10)});
\draw ({cos(-90+10)},{0-2+.5*sin(-90+10)}) -- ({cos(-135+10)},{-2+0.5*sin(-135+10)});
\draw ({cos(-135+10)},{-2+0.5*sin(-135+10)}) -- ({cos(-180+10)},{-2+0.5*sin(-180+10)});
\draw ({cos(180+10)},{0.5*sin(180+10)}) -- ({cos(180+10)},{-2+0.5*sin(180+10)});
\draw [dashed] ({cos(135+10)},{0.5*sin(135+10)}) -- ({cos(135+10)},{-2+0.5*sin(135+10)});
\draw [dashed] ({cos(90+10)},{0.5*sin(90+10)}) -- ({cos(90+10)},{-2+0.5*sin(90+10)});
\draw [dashed] ({cos(45+10)},{0.5*sin(45+10)}) -- ({cos(45+10)},{-2+0.5*sin(45+10)});
\draw ({cos(0+10)},{0.5*sin(0+10)}) -- ({cos(0+10)},{-2+0.5*sin(0+10)});
\draw ({cos(-45+10)},{0.5*sin(-45+10)}) -- ({cos(-45+10)},{-2+0.5*sin(-45+10)});
\draw ({cos(-90+10)},{0.5*sin(-90+10)}) -- ({cos(-90+10)},{-2+0.5*sin(-90+10)});
\draw ({cos(-135+10)},{0.5*sin(-135+10)}) -- ({cos(-135+10)},{-2+0.5*sin(-135+10)});

\draw ({3+cos(180+10)},{0.5*sin(180+10)}) -- ({3+cos(144+10)},{0.5*sin(144+10)});
\draw ({3+cos(144+10)},{0.5*sin(144+10)}) -- ({3+cos(108+10)},{0.5*sin(108+10)});
\draw ({3+cos(108+10)},{0.5*sin(108+10)}) -- ({3+cos(72+10)},{0.5*sin(72+10)});
\draw ({3+cos(72+10)},{0.5*sin(72+10)}) -- ({3+cos(36+10)},{0.5*sin(36+10)});
\draw ({3+cos(36+10)},{0.5*sin(36+10)}) -- ({3+cos(0+10)},{0.5*sin(0+10)});
\draw ({3+cos(0+10)},{0.5*sin(0+10)}) -- ({3+cos(-36+10)},{0.5*sin(-36+10)});
\draw ({3+cos(-36+10)},{0.5*sin(-36+10)}) -- ({3+cos(-72+10)},{0.5*sin(-72+10)});
\draw ({3+cos(-72+10)},{0.5*sin(-72+10)}) -- ({3+cos(-108+10)},{0.5*sin(-108+10)});
\draw ({3+cos(-108+10)},{0.5*sin(-108+10)}) -- ({3+cos(-144+10)},{0.5*sin(-144+10)});
\draw ({3+cos(-144+10)},{0.5*sin(-144+10)}) -- ({3+cos(-180+10)},{0.5*sin(-180+10)});
\draw [dashed] ({3+cos(180+10)},{-2+0.5*sin(180+10)}) -- ({3+cos(144+10)},{-2+0.5*sin(144+10)});
\draw [dashed] ({3+cos(144+10)},{-2+0.5*sin(144+10)}) -- ({3+cos(108+10)},{-2+0.5*sin(108+10)});
\draw [dashed] ({3+cos(108+10)},{-2+0.5*sin(108+10)}) -- ({3+cos(72+10)},{-2+0.5*sin(72+10)});
\draw [dashed] ({3+cos(72+10)},{-2+0.5*sin(72+10)}) -- ({3+cos(36+10)},{-2+0.5*sin(36+10)});
\draw [dashed] ({3+cos(36+10)},{-2+0.5*sin(36+10)}) -- ({3+cos(0+10)},{-2+0.5*sin(0+10)});
\draw ({3+cos(0+10)},{-2+0.5*sin(0+10)}) -- ({3+cos(-36+10)},{-2+0.5*sin(-36+10)});
\draw ({3+cos(-36+10)},{-2+0.5*sin(-36+10)}) -- ({3+cos(-72+10)},{-2+0.5*sin(-72+10)});
\draw ({3+cos(-72+10)},{-2+0.5*sin(-72+10)}) -- ({3+cos(-108+10)},{-2+0.5*sin(-108+10)});
\draw ({3+cos(-108+10)},{-2+0.5*sin(-108+10)}) -- ({3+cos(-144+10)},{-2+0.5*sin(-144+10)});
\draw ({3+cos(-144+10)},{-2+0.5*sin(-144+10)}) -- ({3+cos(-180+10)},{-2+0.5*sin(-180+10)});
\draw ({3+cos(180+10)},{0.5*sin(180+10)}) -- ({3+cos(180+10)},{-2+0.5*sin(180+10)});
\draw [dashed] ({3+cos(144+10)},{0.5*sin(144+10)}) -- ({3+cos(144+10)},{-2+0.5*sin(144+10)});
\draw [dashed] ({3+cos(108+10)},{0.5*sin(108+10)}) -- ({3+cos(108+10)},{-2+0.5*sin(108+10)});
\draw [dashed] ({3+cos(72+10)},{0.5*sin(72+10)}) -- ({3+cos(72+10)},{-2+0.5*sin(72+10)});
\draw [dashed] ({3+cos(36+10)},{0.5*sin(36+10)}) -- ({3+cos(36+10)},{-2+0.5*sin(36+10)});
\draw ({3+cos(0+10)},{0.5*sin(0+10)}) -- ({3+cos(0+10)},{-2+0.5*sin(0+10)});
\draw ({3+cos(-144+10)},{0.5*sin(-144+10)}) -- ({3+cos(-144+10)},{-2+0.5*sin(-144+10)});
\draw ({3+cos(-108+10)},{0.5*sin(-108+10)}) -- ({3+cos(-108+10)},{-2+0.5*sin(-108+10)});
\draw ({3+cos(-72+10)},{0.5*sin(-72+10)}) -- ({3+cos(-72+10)},{-2+0.5*sin(-72+10)});
\draw ({3+cos(-36+10)},{0.5*sin(-36+10)}) -- ({3+cos(-36+10)},{-2+0.5*sin(-36+10)});

\draw[->,>=latex] (4.5,-1) -- (6.5,-1);
\draw [domain=-1:1] plot (8+\x,{0.5*sqrt(1-\x*\x)});
\draw [domain=-1:1] plot (8+\x,{-0.5*sqrt(1-\x*\x)});
\draw [dashed] [domain=-1:1] plot (8+\x,{-2+0.5*sqrt(1-\x*\x)});
\draw [domain=-1:1] plot (8+\x,{-2-0.5*sqrt(1-\x*\x)});
\draw (7,0) -- (7,-2);
\draw (9,0) -- (9,-2);
\end{tikzpicture}

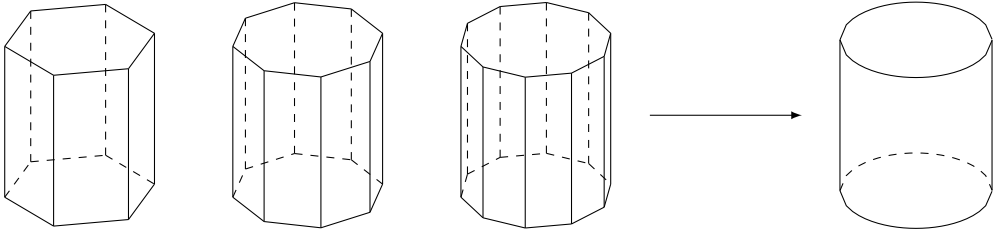
\captionof{figure}{Accumulation of singularities}
\label{figureintro}
\end{center}

\noindent The limit space is a cylinder (or a can: the top and the bottom belong to the surface), and the curvature measure of this singular surface is the usual angle measure on the two circles, at the top and at the bottom of the cylinder. We can imagine more complicated examples, for example the cone points may accumulate along a Cantor set in $\mathbb{S}^1$; the limit metric would then have conical singularities along a Cantor set, and the curvature measure would be the Hausdorff measure of the Cantor set.

Since these singular surfaces may be defined by approximation by smooth Riemannian surfaces (see definition \ref{definitionsurfacesBICbyapproximation}), most of the properties of smooth surfaces extend to this setting:  there is always a definite angle between any two geodesics, we have the existence of local conformal coordinates (see theorem-definition \ref{theoremedefinitionconformalcharts})... This last property is crucial in our article: the metric is locally induced by a (singular) Riemannian metric  $g_{\omega,h} = e^{2V[\omega](z) + 2h(z)} |dz|^2$, where $V[\omega]$ is the potential of the curvature measure $\omega$, and $h$ is a harmonic function. Hence, if we know the curvature measure, then we know the local expression the metric, up to a harmonic function. One of the key steps in this article is to obtain a control on this harmonic term (see theorem \ref{theoremetermeharmoniqueborne}). When we forget it (that is, we put $h=0$), then we have the following local convergence theorem, due to Y. Reshetnyak (see section \ref{sectionconformalcharts} for the definition of $\overline{d}_{\omega_m,0}$ and $\overline{d}_{\omega,0}$):
\begin{theorem} [Y. Reshetnyak, see \cite{Reshetnyak_livre}, theorem 7.3.1] \label{theoremeconvergencedistancereshetnjack}
Let $\omega_m^+$ and $\omega_m^-$ be a sequence of non-negative Radon measures with support in $\overline{D}(1/2)$, weakly converging to measures $\omega^+$ and $\omega^-$. Let $\omega_m := \omega_m^+ - \omega_m^-$ and $\omega := \omega^+ - \omega^-$. Then
\[
\overline{d}_{\omega_m,0} \underset{m \rightarrow \infty} \longrightarrow \overline{d}_{\omega,0},
\]
uniformly on any closed set $A \subset \overline{D}(1/2)$ such that $\omega^+(\{z\})<2\pi$ for every $z \in A$. That is, if $z_m \rightarrow z$ and $z'_m \rightarrow z'$, with $\omega^+(\{z\})<2\pi$ and $\omega^+(\{z'\})<2\pi$, then $\overline{d}_{\omega_m,0}(z_m,z'_m) \rightarrow \overline{d}_{\omega,0}(z,z')$.
\end{theorem}

In this article, we use this local theorem to prove a \emph{global} convergence theorem for surfaces with B.I.C., and as a corollary we obtain a compactification of the space of Riemannian metrics with conical singularities. In the (classical) smooth setting, there are very well-known compactness results. Let $\Lambda,i,V$ be positive constants, and $\mathcal{M}_n(\Lambda,i,V)$ be the set of compact Riemannian $n$-manifolds with
\begin{enumerate}
\item |sectional curvature| $\leq \Lambda$,
\item injectivity radius $\geq i$,
\item volume $\leq V$.
\end{enumerate}
In the early 1980's, M. Gromov, in \cite{GLP}, stated the precompactness of the set $\mathcal{M}_n(\Lambda,i,V)$, in the Lipschitz topology: for every sequence $(X_m,g_m) \in \mathcal{M}_n(\Lambda,i,V)$, there exists a Riemannian $n-$manifold $X$, a Riemannian metric $g$ and diffeomorphisms $\varphi_m : X \rightarrow X_m$ such that, after passing to a subsequence, $(X,(\varphi_m)^* d_{g_m}) \rightarrow (X,d_g)$ in the Lipschitz topology ($d_{g_m}$ and $d_g$ are the length distance associated to the Riemannian metrics $g_m$ and $g$). This so-called Cheeger-Gromov convergence theorem was already implicit in the thesis of J. Cheeger in 1970. Since then, many articles were published on the subject, and the initial statement of M. Gromov was improved in two different ways: one only needs a bound on the Ricci curvature, and the convergence is much stronger than in the Lipschitz topology (see \cite{An}, \cite{AC}, \cite{GW}, \cite{K} and \cite{P}). We need to use harmonic coordinates in order to obtain the optimal regularity in the convergence (see \cite{DTK} and \cite{JK}).

For surfaces with B.I.C., the only convergence theorem known to the author deals with a sequence of metrics in a fixed conformal class (see the theorem 6.2 in \cite{Troyanov_alexandrovsurfaces}): it is a direct consequence of the local convergence theorem (theorem \ref{theoremeconvergencedistancereshetnjack}). When we look for a convergence theorem for a sequence of metrics $d_m$ on a surface $\Sigma$, at some point one needs to construct the diffeomorphisms $\varphi_m : \Sigma \rightarrow \Sigma$. It always involves serious work, for example by embedding the manifolds in some bigger space (see \cite{GLP}, \cite{HH} or the present article). Some of the consequences of a uniform convergence $d_m \rightarrow d$, up to diffeomorphisms (that is, there exists diffeomorphisms $\varphi_m : \Sigma \rightarrow \Sigma$ such that $(\varphi_m)^* d_m \rightarrow d$ uniformly on $\Sigma$) are described in \cite{Reshetnyak_livre} and \cite{AZ}: it deals with the length of converging curves, convergence of polygons, weak convergence of curvature measures, convergence of angles...

We want to adapt the three hypothesis of the compactness theorem for smooth Riemannian metrics to our singular setting. The hypothesis 1. deals with the sectional (Gauss) curvature, which does not exist everywhere in the singular setting, hence we ask for a bound on the curvature measure instead. In order to avoid a cusp, that is, a point $x \in \Sigma$ where the non-negative part of the curvature measure is $\omega^+(\{x\})=2\pi$ (such a point may be at infinite distance to any other point of the surface, see remark \ref{remarkcusp}), we ask for the inequality $\omega^+(B(x,\varepsilon)) \leq 2\pi-\delta$ for every $x \in \Sigma$ ($\varepsilon$ and $\delta$ are positive constants). The hypothesis 3., which deals with the volume, already makes sense - there is a natural notion of area on a surface with B.I.C.

So let us look at the hypothesis 2. In the smooth setting, a lower bound on the injectivity radius avoids a pinching of the manifold (as may happen, for example, when one factor of a torus $\mathbb{S}^1 \times \mathbb{S}^1$ shrinks to a point). But for surfaces with B.I.C., the injectivity radius does not make sense, and even for a surface with conical singularities, the injectivity radius of the (open) smooth part is zero (if $x$ is at a distance $r$ of a cone point, then $\inj(x) < r$). Hence we need to define some similar quantity, which makes sense for non-Riemannian metric spaces. We introduce the new notion of \emph{contractibility radius} (see section \ref{sectioncontractibilityradius}), which is the biggest $r$ such that all the closed balls of radius $s<r$ are homeomorphic to a closed disc (hence they are contractible). The important point is that a lower bound on the contractibility radius avoids a pinching of the surface. This notion is very natural: in the classical Cheeger-Gromov convergence theorem, one can replace a lower bound on the injectivity radius by a lower bound on the contractibility radius (see proposition \ref{propositionequivalencerayoncontractibiliterayoninjectivite}).

From now on, we fix $\Sigma$, a \emph{closed} surface: that is, a connected compact smooth surface, without boundary. Let $A,c,\varepsilon,\delta$ be some positive constants. Let $\M_\Sigma (A,c,\varepsilon,\delta)$ be the class of metrics $d$ with B.I.C. on $\Sigma$ such that:
\begin{enumerate}
\item For every $x\in \Sigma$ we have
\[
\omega^+ (B(x,\varepsilon)) \leq 2\pi-\delta ;
\]
\item the contractibility radius of $(\Sigma, d)$ verifies
\[
\cont({\Sigma, d}) \geq c ;
\]
\item the area of $(\Sigma,d)$ verifies
\[
 \area(\Sigma,d) \leq A.
\]
\end{enumerate}


\begin{remark}
From now on, when considering a set $\M_\Sigma (A,c,\varepsilon,\delta)$, we always assume $\varepsilon < c$ (hence there exists conformal charts on balls of radius $\varepsilon$, see theorem-definition \ref{theoremedefinitionconformalcharts} and property \ref{propositioncontboulesouvertes}).
\end{remark}

The main result of the article is the following
\begin{maintheorem} \label{theoremeprincipal}
The space $\M_\Sigma(A,c,\varepsilon,\delta)$ is compact, in the uniform metric sense. That is for every sequence $d_m \in \M_\Sigma(A,c,\varepsilon,\delta)$, there exists a metric $d$ with B.I.C. such that, after passing to a subsequence, there are diffeomorphisms $\varphi_m : \Sigma \rightarrow \Sigma$ with
\[
(\varphi_m)^* d_m \underset{m \rightarrow \infty} \longrightarrow d \mbox{ uniformly on } \Sigma.                                                                                                                                                
\]
\end{maintheorem}

In the classical Cheeger-Gromov theorem, an easy packing argument shows that one can replace an upper bound on the area by an upper bound on the diameter. In our setting, if we want to do so, we also need to ask for an upper bound on the total measure curvature $|\omega|:=\omega^+ + \omega^-$ (see proposition \ref{propositionequivalencedefinition}):

\begin{corollary} \label{corollairetheoremeprincipal}
The space of metrics with B.I.C. verifying the conditions 1. and 2. above, and with diameter $\diam(\Sigma,d) \leq D$ and total measure curvature $|\omega|(\Sigma,d) \leq \Omega$ is compact, in the uniform metric sense. That is, for every sequence of metrics $d_m$ verifying the conditions above, there exists a metric $d$ with B.I.C. such that, after passing to a subsequence, there are diffeomorphisms $\varphi_m : \Sigma \rightarrow \Sigma$ with $(\varphi_m)^* d_m \rightarrow d$ uniformly on $\Sigma$.
\end{corollary}

A fortiori, we have compactness in the Gromov-Hausdorff sense of the sequence of metric spaces $(\Sigma,d_m)$. This property is true under much weaker assumptions: the set of surfaces with B.I.C. with diameter $\leq D$ and total measure curvature $\leq \Omega$ is precompact in the Gromov-Hausdorff topology, see \cite{Shioya}. Of course, in this case, the limit metric space may not be a surface (a sphere can for example shrink to a point).

For metrics with conical singularities, we obtain the following

\begin{corollary} \label{corollairesingconiques}
Consider on $\Sigma$ a sequence $(g_m)$ of Riemannian metrics with conical singularities at points $(p_i^m)_{i \in I_m}$, with angles $\theta_i^m$. Suppose that
\begin{enumerate}
\item for every $x \in \Sigma$, we have
\[
\int_{B_m(x,\varepsilon)} \K_m^+ d\mathcal{A}_m + \sum_{i} (2\pi - \theta_i^m)^+ \leq 2\pi -\delta,
\]
where the sum is taken over $i \in I_m$ such that $p_i^m \in B_m(x,\varepsilon)$;
\item the contractibility radius of $(\Sigma, d_m)$ verifies
\[
\cont({\Sigma, d_m}) \geq c ;
\]
\item the area of $(\Sigma,d_m)$ verifies
\[
 \area(\Sigma,d_m) \leq A.
\]
\end{enumerate}
Then, there exists a metric $d$ with B.I.C. such that, after passing to a subsequence, there are diffeomorphisms $\varphi_m : \Sigma \rightarrow \Sigma$ with
\[
(\varphi_m)^* d_m \underset{m \rightarrow \infty} \longrightarrow d \mbox{ uniformly on } \Sigma.                                                                                                                                                
\]
\end{corollary}

We also obtain some interesting corollaries by considering smooth Riemannian metrics. In the case of non-positive Gauss curvature, the first condition is automatically satisfied, and the injectivity radius is half of the length of the smallest closed geodesic. Hence we obtain the
\begin{corollary} \label{corollairesingconiques}
Consider on $\Sigma$ a sequence $(g_m)$ of smooth Riemannian metrics, with non-positive sectional curvature, such that the length of the smallest closed geodesic is bounded by below, and the area is bounded by above.

Then, there exists a metric $d$ with B.I.C. such that, after passing to a subsequence, there are diffeomorphisms $\varphi_m : \Sigma \rightarrow \Sigma$ with
\[
(\varphi_m)^* d_{g_m} \underset{m \rightarrow \infty} \longrightarrow d \mbox{ uniformly on } \Sigma.                                                                                                                                                
\]
\end{corollary}
\noindent The article is organized as follows:

In section 1, we define metrics with B.I.C., as well as metrics with conical singularities. We also state the existence of local conformal charts.

In section 2, we define the new notion of contractibility radius, we give some properties and we look at some examples. We also explain the link with the injectivity radius in the case of smooth Riemannian metrics with bounded sectional curvature.

In section 3, we prove two properties for surfaces with B.I.C.: one result concerns the volume of balls (by analogy with the case of smooth Riemannian metrics, when one has a control on the sectional curvature); another one is on the length of a line segment, for a singular Riemannian metric which has no harmonic term.

The heart of the article is the section 4: we prove preliminary properties for the set $\M_\Sigma(A,c,\varepsilon,\delta)$. Let $d \in \M_\Sigma(A,c,\varepsilon,\delta)$, and let $H : B(x,\varepsilon) \rightarrow D(1/2)$ be a conformal chart, with $H(x)=0$. First, we prove that the harmonic term for the metric is bounded on every compact set of $D(1/2)$ (this is theorem \ref{theoremetermeharmoniqueborne}). Then, we prove the fundamental theorem \ref{theoremerecouvrement}. Roughly speaking, we have a control on the images by $H$ of balls of "big" radii $B(x,\varepsilon/2)$ and $B(x,\varepsilon/4)$, and balls of "small" radii $B(x,\kappa\varepsilon)$ (for some small constant $\kappa>0$). This control has to be uniform, that is independent of the metric $d\in \M_\Sigma(A,c,\varepsilon,\delta)$ we consider.

In section 5, we prove the Main theorem. We present a detailed sketch of the proof at the beginning of the section: this is an adaptation of the proof of Cheeger-Gromov's compactness theorem presented in \cite{HH}.

In the appendix, we state some results of conformal geometry of annuli, needed in section 4. This standard material can be found in the book of L. Ahlfors \cite{Ahlfors}; we recall it to fix the notations.

\textbf{Notations:} the usual non-negative and non-positive parts of a real number $x$ are $x^+ := \max(x,0)$ and $x^- := \max(-x,0)$. If $f$ is a function, its non-negative and non-positive parts are $f^+(x) := (f(x))^+$ and $f^-(x) := (f(x))^-$, and if $\nu$ is a Radon measure, we define its non-negative and non-positive parts by
\[
 \nu^+(X) := \sup_{A \subset X} \nu(A) \mbox{ and } \nu^-(X) := \sup_{A \subset X} -\nu(A).
\]
$\nu^+$ and $\nu^-$ are two non-negative measures; we have $\nu = \nu^+ - \nu^-$, and we set $|\nu| := \nu^+ + \nu^-$.

\section{Surfaces with Bounded Integral Curvature}

We give two definitions of a surface with B.I.C. : the first one is geometric, and the second one is by approximation by smooth Riemannian surfaces. Then we state the fundamental property that they are locally isometric to a (singular) Riemannian metric, conformal to the euclidean metric $|dz|^2$. Finally we give the definition of a metric with conical singularities. All the notions presented here are taken from \cite{Reshetnyak_livre}, \cite{Reshetnyak_article} and \cite{Troyanov_alexandrovsurfaces}.

Let $\Sigma$ be a closed surface. Recall that a metric $d$ on $\Sigma$ is \emph{intrinsic} if for every $x,y \in \Sigma$ we have
\[
d(x,y) = \inf L(\gamma),
\]
where the infimum is taken over all continuous curves $\gamma : [0,1] \rightarrow \Sigma$, with $\gamma(0)=x$ and $\gamma(1)=y$, and where the length of $\gamma$ is defined by
\[
L(\gamma) := \sup_{0=t_0 \leq ... \leq t_n=1} ~ \left(\sum_{i=0}^{n-1} d(\gamma(t_i),\gamma(t_{i+1}))\right).
\]
In our setting, $\Sigma$ is compact, so if $d$ is an intrinsic metric on $\Sigma$ compatible with the topology, there always exists a minimizing geodesic between two points. That is for every $x,y\in \Sigma$, there exists a continuous curve $\gamma : [0,1] \rightarrow \Sigma$, with $\gamma(0)=x$ and $\gamma(1)=y$, such that $d(x,y)=L(\gamma)$.

\subsection{Definition} \label{sectiondefinition}

\subsubsection{A geometric definition} \label{sectiongeometricdefinition}
For this definition, see \cite{Reshetnyak_article}. Roughly speaking, metrics with B.I.C. are intrinsic metrics, for which a curvature measure is well-defined; we first define the curvature measure of a geodesic triangle (by analogy with the smooth case, where this is equal to the sum of the angles minus $\pi$), then we extend it to any Borel set.

Recall that if $OXY$ is a triangle in the Euclidean space, then if $x=|OX|$, $y=|OY|$ and $z=|XY|$, then the angle at $O$ is
\[
 \arccos\frac{x^2 + y^2 - z^2}{2xy}.
\]

Now, let $d$ be an intrinsic metric on $\Sigma$ (compatible with the topology), and let $\gamma_1,\gamma_2 : [0,\varepsilon) \rightarrow \Sigma$ be two continuous curves with $\gamma_1(0)=\gamma_2(0)=O$. Then we can define the upper angle at $O$ between $\gamma_1$ and $\gamma_2$ by
\[
\overline{a} = \limsup_{X \rightarrow O, Y \rightarrow O} \left(\arccos \frac{d(O,X)^2 + d(O,Y)^2 - d(X,Y)^2}{2d(O,X) d(O,Y)}\right)
\]
(the $\limsup$ is taken over points $X \in \gamma_1$ and $Y \in \gamma_2$). If we replace the $\limsup$ by a $\liminf$, we obtain the lower angle $\underline{a}$ at $O$. If $\overline{a}=\underline{a}$, then we say that the angle at $O$ between $\gamma_1$ and $\gamma_2$ exists, and we set $a := \overline{a}=\underline{a}$.

If $T=ABC$ is a triangle (that is, $A,B,C \in \Sigma$, and we specify some geodesics $\gamma_1 = [AB]$, $\gamma_2 =[BC]$ and $\gamma_3 =[CA]$), then we can define the upper excess of the triangle $ABC$ by
\[
 \overline{\delta}(T) := \overline{a} + \overline{b} + \overline{c}-\pi,
\]
where $\overline{a}$ is the upper angle at $A$ between $\gamma_1$ and $\gamma_3$, $\overline{b}$ is the upper angle at $B$ between $\gamma_1$ and $\gamma_2$, and $\overline{c}$ is the upper angle at $C$ between $\gamma_2$ and $\gamma_3$.

Remark that if the metric is Riemannian, then by the Gauss-Bonnet formula this quantity is $\overline{\delta}(T) = \int_T \K_g d\mathcal{A}_g$.
\begin{definition}
The metric $d$ is with B.I.C. on $\Sigma$ (and we say that $(\Sigma,d)$ is a surface with B.I.C.) if: 
\begin{enumerate}
 \item $d$ is an intrinsic metric on $\Sigma$;
 \item $d$ is compatible with the topology of $\Sigma$;
 \item for every $x \in \Sigma$, there exists a neighborhood $U$, homeomorphic to an open disc, and a constant $M(U)<\infty$ such that for any system $T_1,...,T_n$ of pairwise non-overlapping \emph{simple} (this technical condition is explained in detail in \cite{Reshetnyak_livre}) triangles contained in $U$ we have the inequality
\[
 \sum_{i=1}^n \overline{\delta}(T_i) \leq M(U).
\]
\end{enumerate}
\end{definition}

If $(\Sigma,d)$ is a surface with B.I.C., then we know that the angle between two geodesics always exists, and for a geodesic triangle $T$ with angles $a, b,c$ we set $\delta(T) := a + b + c - \pi$. We can then define the curvature measure: we set $\omega = \omega^+ - \omega^-$, where $\omega^+$ and $\omega^-$ are two non-negative Radon measures, which are defined as follows. If $U \subset \Sigma$ is an open set, we set
\[
 \omega^+ (U) = \sup \left(\sum_{i=1}^n \delta(T_i)^+\right) \mbox{ and } \omega^- (A) = \sup \left(\sum_{i=1}^n \delta(T_i)^-\right),
\]
where the supremum is taken over all system $(T_1,...,T_n)$ of pairwise non-overlapping simple triangles contained in $U$, and if $A \subset \Sigma$ is Borel set we set
\[
 \omega^+(A) := \inf_{\mbox{open sets }U \supset A} \omega^+(U) \mbox{ and } \omega^-(A) := \inf_{\mbox{open sets }U \supset A} \omega^-(U)
\]
(see \cite{AZ}, chapter 5). We will see another definition of the curvature measure $\omega$ by approximation with smooth Riemannian metrics (see the next section).

\begin{remark}
Rather surprisingly, for a geodesic triangle $T$ with angles $a,b,c$, we do not always have the equality $\omega(T)=a+b+c-\pi$; see \cite{Reshetnyak_livre} and \cite{AZ} for more details.
\end{remark}

If $d$ is a Riemannian metric, then the curvature measure is $\omega=\K_g d\mathcal{A}_g$, and we have $\omega^+ = \K_g^+ d\mathcal{A}_g$ and $\omega^- = \K_g^- d\mathcal{A}_g$.

\begin{example}
Let us see on an example how we can compute the excess of a triangle. Look at a geodesic triangle $T= ABC$ on a cylinder:
\begin{center}
\begin{tikzpicture}
\draw [domain=-1:1][samples=200] plot (2*\x,{0.6*sqrt(1-\x*\x)});
\draw [domain=-1:1] [samples=200] plot (2*\x,{-0.6*sqrt(1-\x*\x)});
\draw (-2,0) -- (-2,-2);
\draw (2,0.04) -- (2,-2);
\draw (0,0) node {$\bullet$};
\draw (0,0) node [above] {A};
\draw (2*-0.5,{-0.6*sqrt(1-0.5*0.5)}) node {$\bullet$};
\draw (2*-0.5,{-0.6*sqrt(1-0.5*0.5)}) node [above] {B};
\draw (1,{-2*sqrt(6-1)+2*sqrt(6-2*2)-0.6*sqrt(1-0.5*0.5)}) node {$\bullet$};
\draw (1,{-2*sqrt(6-1)+2*sqrt(6-2*2)-0.6*sqrt(1-0.5*0.5)}) node [below] {C};
\draw (2*0.5,{-0.6*sqrt(1-0.5*0.5)}) node {$\bullet$};
\draw (2*0.5,{-0.6*sqrt(1-0.5*0.5)}) node [above] {P};
\draw [domain=-2:-1] plot (3+2*\x,{-2*sqrt(6-\x*\x)+2*sqrt(6-2*2)-0.6*sqrt(1-0.5*0.5)});
\draw (0,0) -- (2*-0.5,{-0.6*sqrt(1-0.5*0.5)});
\draw (2*0.5,{-0.6*sqrt(1-0.5*0.5)}) -- (2*0.5, {-2*sqrt(6-1)+2*sqrt(6-2*2)-0.6*sqrt(1-0.5*0.5)});
\draw (2*0.5,{-0.6*sqrt(1-0.5*0.5)}) -- (0,0);
\end{tikzpicture}
\captionof{figure}{}
\end{center}
We assume that $A$ is the center of the circle at the top of the cylinder, $B$ is on this circle, and $C$ is not on the top. Let $P$ be the intersection point between the circle and the geodesic between $A$ and $C$. Call $\theta$ the angle at $A$; in the geodesic triangle $ABC$, we are looking for the angles at $B$ and $C$. If we cut the cylinder along the circle, we obtain two parts:
\begin{center}
\begin{tikzpicture}
\draw (-2,0) circle (2);
\draw (-2,0) node {$\bullet$};
\draw [domain=240:300] plot ({-2+0.3*cos(\x)},{0.3*sin(\x)});
\draw (-2,0) node [above] {A};
\draw (-2,-0.75) node [above] {$\theta$};
\draw ({-2+2*cos(240)},{2*sin(240)}) node {$\bullet$};
\draw ({-2+2*cos(240)},{2*sin(240)}) node [below] {B};
\draw ({-2+2*cos(300)},{2*sin(300)}) node {$\bullet$};
\draw ({-2+2*cos(300)},{2*sin(300)}) node [below] {P};
\draw (-2,0) -- ({-2+2*cos(240)},{2*sin(240)});
\draw (-2,0) -- ({-2+2*cos(300)},{2*sin(300)});
\draw [dashed] ({-2+2*cos(240)-sin(240)},{2*sin(240)+cos(240)}) -- ({-2+2*cos(240)+sin(240)},{2*sin(240)-cos(240)});
\draw [dashed] ({-2+2*cos(300)-sin(300)},{2*sin(300)+cos(300)}) -- ({-2+2*cos(300)+sin(300)},{2*sin(300)-cos(300)});
\draw (1,0) node {and};
\draw (2.5,0) -- (5,0);
\draw (5,0) -- (5,-2);
\draw (2.5,0) -- (5,-2);
\draw (4.8,0) -- (4.8,-0.2);
\draw (5,-0.2) -- (4.8,-0.2);
\draw [domain=-38.7:0] plot ({2.5+0.3*cos(\x)},{0.3*sin(\x)});
\draw (3.2,-0.25) node {$b$};
\draw [domain=90:141.3] plot ({5+0.3*cos(\x)},{-2+0.3*sin(\x)});
\draw (4.8,-1.5) node {$c$};
\draw (2.5,0) node [left] {B};
\draw (5,0) node [right] {P};
\draw (5,-2) node [right] {C};
\end{tikzpicture}
\captionof{figure}{}
\end{center}
Call $b$ and $c$ the angles at $B$ and $C$ in the right (Euclidean) triangle $BPC$: we have $b+c=\pi/2$. Then in the geodesic triangle $ABC$, the angle at $B$ is $b+\pi/2$, and the angle at $C$ is $c$, hence the excess of the triangle is
\[
\delta(T) = (\theta + b + \pi/2 + c) - \pi = \theta.
\]
\end{example}

\subsubsection{A definition by approximation}

Metrics with B.I.C. can be uniformly approximated by Riemannian metrics. Indeed, we have the following alternative definition (see \cite{Troyanov_alexandrovsurfaces}):

\begin{definition} \label{definitionsurfacesBICbyapproximation}
 $d$ is a metric with B.I.C. on $\Sigma$ if:
\begin{enumerate}
\item $d$ is an intrinsic distance on $\Sigma$;
\item $d$ is compatible with the topology of $\Sigma$;
\item there exists a sequence $g_m$ of Riemannian metrics on $\Sigma$, with $(\int_\Sigma |\K_{g_m}| d\mathcal{A}_m)_{m \in \N}$ bounded, such that $d$ is the uniform limit of the metrics $d_{g_m}$ on $\Sigma$.
\end{enumerate}
\end{definition}

The third condition explains the terms ``Bounded Integral Curvature''. We can then define $\omega$ as the weak limit of $\K_{g_m} d\mathcal{A}_m$, and the area measure $d\mathcal{A}$ as the weak limit of $d\mathcal{A}_m$ (because of the equivalence of the two definitions, these measures do not depend of the choice of the sequence $(g_m)$). Note that the area measure coincides with the two-dimensional Hausdorff measure of the metric space $(\Sigma,d)$.

\subsection{Conformal charts} \label{sectionconformalcharts}

In the sequel, for $r>0$, we set
\[
 D(r) := \{ z \in \C, |z|<r \},
\]
and
\[
\overline{D}(r) := \{ z \in \C, |z|\leq r \}.
\]

Let $\omega$ be a Radon measure with support in $\overline{D}(1/2)$, and $h$ a harmonic function on $D(1/2)$. Consider the following (singular) Riemannian metric:
\[
 g_{\omega,h} = e^{2V[\omega](z) + 2h(z)} |dz|^2.
\]
The map $V[\omega]$ is the potential of the measure $\omega$, and is defined by
\begin{equation*} 
 V[\omega] (z) := \iint_{\C} \left(\frac{-1}{2\pi}\right) \ln|z-\xi| d\omega(\xi) : 
\end{equation*}
it is defined for almost every $z \in \C$, and $V[\omega] \in L^1_{loc}(\C)$. It verifies $\Delta V[\omega] = \omega$ in the weak sense, where the sign convention for the Laplace operator on $\C$ is $\Delta = -\frac{\partial^2}{\partial x^2} - \frac{\partial^2}{\partial y^2}$.

Since
\[
V[\omega] (z) = \iint_{D(1/2)} \left(\frac{-1}{2\pi}\right) \ln|z-\xi| d\omega(\xi),
\]
and $\omega = \omega^+ - \omega^-$, we can write $V[\omega] = V[\omega^+]-V[\omega^-]$. Moreover, for every $z,\xi \in D(1/2)$ we have $\ln|z-\xi| \leq 0$, so for almost every $z \in D(1/2)$ we have $V[\omega^+](z) \geq 0$ and $V[\omega^-](z) \geq 0$, hence
\[
 -V[\omega^-](z) \leq V[\omega](z) \leq V[\omega^+](z).
\]
These inequalities will be used many times in the sequel. We would not have such inequalities if $D(1/2)$ had been replaced by another set (the unit disc $D(1)$ for example); this explains why in theorem-definition \ref{theoremedefinitionconformalcharts}, the conformal charts are defined on $D(1/2)$.

Consider $\gamma : [0,1] \rightarrow D(1/2)$ a continuous simple curve (that is, $\gamma$ is injective), parametrized with constant speed $s$ (that is, the Euclidean length of the curve $\gamma_{|[t_1,t_2]}$ is $s\cdot (t_2-t_1)$ for every $t_1 \leq t_2$). We define the length of $\gamma$ for the singular Riemannian metric $g_{\omega,h}$ by
\[
L_{\omega,h}(\gamma) := \int_0^1 e^{V[\omega](\gamma(t)) + h(\gamma(t))} s  \cdot dt
\]
(we use a curve with constant speed because we only make a continuity assumption, so the quantity $|\gamma'(t)|$ may not exist). This integral makes sense, that is $V[\omega](\gamma(t))$ is well defined for almost every $t \in [0,1]$ (this is because $V[\omega]$ is the difference of two subharmonic functions; see \cite{Reshetnyak_livre}, p.99).

Then we set
\[
 d_{\omega,h}(z,z') := \inf\big\{ L_{\omega,h} (\gamma) \big\} \in [0,+\infty],
\]
where the infimum is taken over all continuous simple curves $\gamma : [0,1] \rightarrow D(1/2)$, parametrized with constant speed, with $\gamma(0)=z$ and $\gamma(1)=z'$. It is clear that $d_{\omega,h}$ verify all the properties to be a distance, except that we may have $d_{\omega,h}(z,z')=\infty$. A sufficient condition for $d_{\omega,h}$ to be a distance is the following (see \cite{Troyanov_alexandrovsurfaces}, proposition 5.3):
\begin{proposition} \label{propositioninegaliteomegaplus}
If for every $z \in D(1/2)$ we have $\omega^+(\{z\})<2\pi$, then $d_{\omega,h}$ is a distance on $D(1/2)$. 
\end{proposition}

\begin{remark} \label{remarkcusp}
 If $\omega^+(\{z_0\})=2\pi$ for some $z_0 \in D(1/2)$ (we say that $z_0$ is a \emph{cusp}), then $z_0$ may be at infinite distance to any other point $z \in D(1/2)$. For example, if we set $\omega = 2\pi \delta_0$ ($\delta_0$ is the Dirac mass at $0 \in \C$) and $h=0$, then $g_{\omega,h} = |z|^{-2} |dz|^2$, and we easily see that $d_{\omega,h}(0,z)=\infty$ for any $z \neq 0$.
\end{remark}

If the condition of proposition \ref{propositioninegaliteomegaplus} is satisfied, then we say that the metric has no cusp. $d_{\omega,h}$ is then compatible with the topology of $D(1/2)$ (as a subset of $\C$), and $(D(1/2),d_{\omega,h})$ is a surface with B.I.C. In the sequel, we will always assume that the metrics have no cusp. By the hypothesis 1., this is true for every $d \in \M_\Sigma (A,c,\varepsilon,\delta)$.

To state the local convergence theorem (theorem \ref{theoremeconvergencedistancereshetnjack}), we also need the following definition: if $z,z' \in D(1/2)$, we set 
\begin{equation} \label{equationdistancetechnique}
\overline{d}_{\omega,0}(z,z'):= \inf \big\{ L_{\omega,0}(\gamma) \big\}
\end{equation}
where the infimum is taken over all continuous simple curves $\gamma : [0,1] \rightarrow \overline{D}(1/2)$, parametrized with constant speed, with $\gamma(0)=z$ and $\gamma(1)=z'$. The difference with $d_{\omega,0}(z,z')$ is that the curves we are considering here can meet $\partial D(1/2)$. This technical detail is only needed in the proof of corollary \ref{corollaireconvergencedistance}. At every other place in the article, we will use $d_{\omega,h}$ or $d_{\omega,0}$.

\begin{example}
Let $g$ be a smooth Riemannian metric on $\Sigma$, and let $x \in \Sigma$. Around $x$, we can find local coordinates $z \in D(1/2)$ such that the metric is conformal to the Euclidean metric, that is
\[
 g = e^{2u(z)} |dz|^2.
\]
For the Gauss curvature and the area, we have the following formulas:
\[
 \K_g = (\Delta u) e^{-2u} \mbox{ and } d\mathcal{A}_g = e^{2u} d\lambda(z)
\]
(in all the sequel, $d\lambda$ is the Lebesgue measure on $\mathbb{C}$), so the curvature measure is $\omega = \K_g d\mathcal{A}_g = \Delta u ~ d\lambda(z)$. Let $h := u - V[\omega]$: by definition of the potential $V[\omega]$, $h$ is harmonic, so the metric has the following form:
\[
 g = e^{2V[\omega] + 2h} |dz|^2 = g_{\omega,h}.
\]
\end{example}

Next property is fundamental. Like Riemannian metrics, metrics with B.I.C. are locally conformal to the Euclidean metric: this is Theorem 4 in \cite{Reshetnyak_article}.

\begin{theorem-definition} \label{theoremedefinitionconformalcharts}
Let $(\Sigma,d)$ be a surface with B.I.C., with no cusp, and let $U$ be an open set, homeomorphic to an open disc, such that $\overline{U}$ is homeomorphic to a closed disc. Then there exists a map $H$, a measure $\omega_H$, defined in $D(1/2)$, and a harmonic function $h$ on $D(1/2)$ such that
\[
H : (U,d_{|U}) \rightarrow (D(1/2),d_{\omega_H,h})
\]
is an isometry. Such a map $H$ is called a \emph{conformal chart}.

We denote by $d_{|U}$ the intrinsic distance induced by $d$ on $U$: that is, $d_{|U}(x,y)$ is the infimum of the $d-$length of curves joigning $x$ and $y$ in $U$.

The measure $\omega_H$ is defined by $\omega_H = H_\# \omega$ ($\omega$ is the standard curvature measure associated to every surface with B.I.C.); that is, $\omega_H (A) = \omega(H^{-1}(A))$ for every Borel set $A \subset D(1/2)$.

Like for surfaces with Riemannian metrics, the area of any Borel set $A \subset U$ is
\[
 \area(A) = \int_{H(A)} e^{2V[\omega_H](z)+2h(z)} d\lambda(z).
\]

Moreover, this theorem shows that the surface $\Sigma$ has a natural structure of a Riemann surface (see \cite{Reshetnyak_livre}).

\end{theorem-definition}

\subsection{Surfaces with conical singularities} \label{sectiondefinitionsingulariteconiques}

\begin{definition}
A metric with conical singularities is a metric $d$ with B.I.C., with no cusp such that, if the Radon-Nikodym decomposition of the curvature measure $\omega$ with respect to the area measure $d\mathcal{A}$ reads
\[
\omega = \mu + \K d\mathcal{A}
\]
for some function $\K \in L^1_{loc}(d\mathcal{A})$, then the singular measure $\mu$ is a finite sum of Dirac masses.
\end{definition}

If $\mu = \sum_{i \in I} k_i \delta_{p_i}$ (where $\delta_{p_i}$ is the Dirac mass at $p_i$, and $k_i < 2\pi$), then in a neighborhood of any $p_i$ there are complex coordinates $z \in D(1/2)$ such that the singular metric reads
\[
g = |z|^{2\beta_i} e^{2u_i(z)}|dz|^2,
\]
with $\beta_i:=-k_i/2\pi>-1$, and $u_i \in L^1_{loc}(D(1/2))$, with $\Delta u_i \in L^1_{loc}(D(1/2))$ in the weak sense.

The metric looks like an Euclidean cone: the plane, endowed with the metric $g = |z|^{2\beta} |dz|^2$, is isometric an Euclidean cone of angle
\[
\theta = 2\pi (\beta+1) = 2\pi - k.
\]
For $\theta \in (0,2\pi)$, this can be obtained by gluing an angular sector of the plane:
\begin{center}
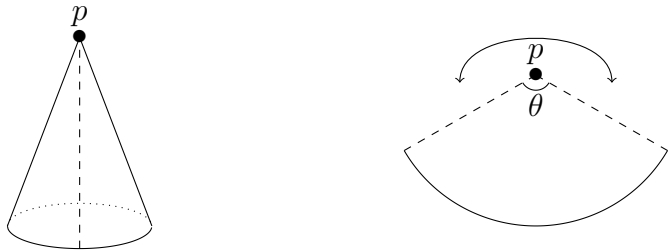

\begin{tikzpicture}
\draw [dotted] [domain = -6.95:-5.05] [samples = 200] plot (\x, {-2+0.3*sqrt(1-((\x+6)^2)/0.95^2))});
\draw [domain = -6.95:-5.05] [samples = 200] plot (\x, {-2-0.3*sqrt(1-((\x+6)^2)/0.95^2))});
\draw (-6.95,-2) -- (-6,0.5);
\draw (-5.05,-2) -- (-6,0.5);
\draw [dashed] (-6,-2.3) -- (-6,0.5);
\draw (-6,0.5) node {$\bullet$};
\draw (-6,0.5) node [above] {$p$};
\draw [domain = -30:-150] [samples = 100] plot (\x : 2) ;
\draw [dashed] (0,0) -- ({2*cos(30)},{-2*sin(30)});
\draw [dashed] (0,0) -- ({2*cos(150)},{-2*sin(150)});
\draw[<->] (-1,-0.1) to[out=90,in=90] (1,-0.1);
\draw [-] [domain = -30:-150] [samples = 100] plot (\x : 0.2);
\draw (0,-0.4) node {$\theta$};
\draw (0,0) node {$\bullet$};
\draw (0,0) node [above] {$p$};
\end{tikzpicture}
\captionof{figure}{If we cut the cone, we obtain an angular sector of the plane.}
\end{center}

\section{Contractibility radius} \label{sectioncontractibilityradius}

We have already mentioned in the introduction that the contractibility radius is, in some sense, a generalization of the injectivity radius to non-Riemannian metric spaces: the important point is that a lower bound on the contractibility radius prevents a pinching of the surface. We first prove a proposition on the topology of closed balls, needed for the definition; then we give a criterion which ensures the positivity of the contractibility radius of some surface $(\Sigma,d)$.  We also look at some examples (the Euclidean cones), and we end this section by proving that, in Cheeger-Gromov's convergence theorem, one can replace a lower bound on the injectivity radius by a lower bound on the contractibility radius.

\subsection{Definition}

Let $(\Sigma,d)$ be a closed surface with B.I.C. If $x \in \Sigma$, we denote by $\overline{B}(x,r)$ the closed ball centered in $x$ and with radius $r$ (that is, the set of $y\in \Sigma$ with $d(x,y)\leq r$). Since the metric is intrinsic, it is the closure of $B(x,r)$. To define the contractibility radius, we need the following

\begin{proposition} \label{propositionavantdefcont}
For every $x \in \Sigma$, there exists some $r>0$ such that for every $s<r$, $\overline{B}(x,s)$ is homeomorphic to a closed disc.
\end{proposition}

To prove this proposition, we need a lemma, which is a direct consequence of a result due to Y. D. Burago and M. B. Stratilatova, see \cite{Reshetnyak_livre}, theorem 9.1. Let $S(x,r)$ be the sphere with center $x$ and radius $r$ (that is, the set of $y \in \Sigma$ with $d(x,y)=r$). In the general case, the set $S(x,r)$ may be arranged in a rather complicated way. 

\begin{theorem}[Y. D. Burago and M. B. Stratilatova] \label{theoremcriterepratique}
Let $U$ be a set homeomorphic to an open disc, with $x \in U$ and $\omega^+(U-\{x\})<\pi$. If $S(x,r) \subset U$, then $S(x,r)$ is a Jordan curve.
\end{theorem}

\begin{lemma} \label{lemmecriterepratique}
Let $U$ be a set homeomorphic to an open disc, with $x \in U$ and $\omega^+(U-\{x\})<\pi$.
If $\overline{B}(x,r) \subset U$, then for every $s \leq r$, $\overline{B}(x,s)$ is homeomorphic to a closed disc.
\end{lemma}
\begin{proof}[Proof of lemma \ref{lemmecriterepratique}]
Let $h : U \rightarrow \C$ be a homeomorphism, and let $s \leq r$. Since $S(x,s) \subset \overline{B}(x,r) \subset U$, we can apply theorem \ref{theoremcriterepratique}: $h(S(x,s))$ is a Jordan curve $\Gamma$. $\C-\Gamma$ has two connected components: call the bounded component the "interior" of $\Gamma$, and the unbounded component the "exterior" of $\Gamma$. Since $B(x,s)$ is open and closed in $U-S(x,s)$, $h(B(x,s))$ is a connected component of $h(U-S(x,s))=\C-\Gamma$, hence we have either
\[
h(B(x,s))=\mbox{interior of } \Gamma, \mbox{ or } h(B(x,s))= \mbox{ exterior of } \Gamma.
\]
The second case is impossible, since the closure of $h(B(x,s))$ is $h(\overline{B}(x,s))$, which is compact, and the closure of the exterior of $\Gamma$ is non-compact.

Hence $h(B(x,s))$ is the interior of $\Gamma$, and $h(\overline{B}(x,s))$ is the closure of the interior of $\Gamma$. By the Jordan-Schoenflies' theorem, we know that these sets are (respectively) homeomorphic to an open (respectively closed) disc on the plane, and this ends the proof of lemma \ref{lemmecriterepratique}.
\end{proof}

\begin{proof}[Proof of proposition \ref{propositionavantdefcont}]
By the structure of smooth surface of $\Sigma$, we know that we can construct a decreasing sequence of open sets $(U_i)$, such that every $U_i$ is homeomorphic to an open disc, with
\[
\{x\} = \bigcap_{i\in\N} U_i.
\]
We have
\[
0=\omega^+ \big(\bigcap_{i\in\N} (U_i-\{x\}) \big) = \lim_{i \rightarrow \infty} \omega^+(U_i-\{x\}),
\]
hence there exists some $i_0 \in \N$ with $\omega^+(U_{i_0}-\{x\})<\pi$. Consider some $r>0$ such that $\overline{B}(x,r) \subset U_{i_0}$: we can apply lemma \ref{lemmecriterepratique}, and for every $s\leq r$, $\overline{B}(x,s)$ is homeomorphic to a closed disc. This ends the proof of proposition \ref{propositionavantdefcont}.
\end{proof}

We can then define the following \emph{contractibility radius}:
\[
\cont(\Sigma,d,x) := \sup \big\{r>0 \text{ } \big| \text{ for every } s<r,  \overline{B}(x,s) \mbox{ is homeomorphic to a closed disc} \big\}
\]
(by definition, $\cont(\Sigma,d,x)>0$) and
\[
\cont(\Sigma,d) := \inf_{x \in \Sigma} \cont(\Sigma,d,x).
\]
Since $\overline{B}(x,\diam \Sigma) = \Sigma$ is not homeomorphic to a closed disc, we have the inequalities 
\[
\cont(\Sigma,d,x) \leq \diam \Sigma \mbox{ and } \cont(\Sigma,d) \leq \diam \Sigma.
\]
The following proposition gives a criteria which ensures that $\cont(\Sigma,d)>0$. This will not be used in the sequel.
\begin{proposition}
If the non-negative part of the curvature measure of $(\Sigma,d)$ verifies $\omega^+(\{x\})<\pi$ for every $x \in \Sigma$, then $\cont(\Sigma,d)>0$.
\end{proposition}
Conversely, in the next section, we show that the contractibility radius of an Euclidean cone with curvature at the vertex $p$ greater than $\pi$ (that is with $\omega^+(\{p\})>\pi$) is zero.

\begin{proof}
Let $(x_m)$ be a sequence in $\Sigma$ such that
\[
\cont(\Sigma,d,x_m) \underset{m \rightarrow \infty} \longrightarrow \cont(\Sigma,d).
\]
By compactness, we may assume $x_m \rightarrow x \in \Sigma$. As in the proof of proposition \ref{propositionavantdefcont}, consider a decreasing sequence of open sets $(U_i)$, such that every $U_i$ is homeomorphic to an open disc, with
\[
\{x\}= \bigcap_{i \in \N} U_i.
\]
Since $\omega^+(\{x\}) = \lim_{i \rightarrow \infty} \omega^+(U_i) < \pi$, there exists some $i_0 \in \N$ such that $\omega^+(U_{i_0})<\pi$. Consider some $r>0$ such that $\overline{B}(x,r) \subset U_{i_0}$. Then for $m$ large enough, we have $\overline{B}(x_m,r/2) \subset \overline{B}(x,r) \subset U_{i_0}$: since $\omega^+(U_{i_0}-\{x_m\}) \leq \omega^+(U_{i_0})<\pi$, we can apply lemma \ref{lemmecriterepratique} to obtain $\cont(\Sigma,d,x_m) \geq r/2$. We $m$ tends to infinity, we obtain $\cont(\Sigma,d) \geq r/2$.
\end{proof}

A lower bound for $\cont(\Sigma,d)$ avoids a pinching of the surface at the point $x$. In the following situation, $\cont(\Sigma,d,x)$ is very small:
\begin{center}
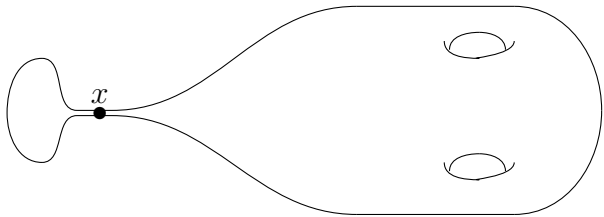

\begin{tikzpicture}[scale=0.23]
\draw (0,0)--(2,0);
\draw (0,-0.3)--(2,-0.3);
\draw[-] (0,0) to[out=180,in=0] (-2,3);
\draw[-] (-2,3) to[out=180,in=90] (-4,-0.15);
\draw[-] (-2,-3) to[out=180,in=270] (-4,-0.15);
\draw[-] (-2,-3) to[out=0,in=180] (0,-0.3);
\draw[-] (2,0) to[in=180,out=0] (16,6);
\draw (16,6)--(25,6);
\draw[-] (2,-0.3) to[in=180,out=0] (16,-6);
\draw (16,-6)--(25,-6);
\draw[-] (25,6) to[in=90,out=0] (30,0);
\draw[-] (25,-6) to[in=270,out=0] (30,0);
\draw[-] (21,4) to[in=180,out=270] (23,3);
\draw[-] (23,3) to[in=270,out=180] (25,4);
\draw[-] (21.3,3.5) to[in=180,out=90] (23,4.5);
\draw[-] (23,4.5) to[in=90,out=0] (24.5,3.5);
\draw[-] (21,-3) to[in=180,out=270] (23,-4);
\draw[-] (23,-4) to[in=270,out=180] (25,-3);
\draw[-] (21.3,-3.5) to[in=180,out=90] (23,-2.5);
\draw[-] (23,-2.5) to[in=90,out=0] (24.5,-3.5);
\draw (1.3,-0.2) node {$\bullet$};
\draw (1.3,-0.2) node [above] {$x$};
\end{tikzpicture}
\captionof{figure}{The surface is pinched at the point $x$.}
\end{center}

\begin{remark}
To avoid a pinching of the surface, we could have defined the following natural quantity:
\[
\sup \big\{r>0 \text{ } \big| \overline{B}(x,r) \mbox{ is homeomorphic to a closed disc} \big\},
\]
but it is not relevant, since it is not small in the example given above.
\end{remark}

The definition of the contractibility radius deals with \emph{closed} balls. To be able to apply theorem-definition \ref{theoremedefinitionconformalcharts} with the \emph{open} balls $B(x,r)$, we need the following

\begin{proposition} \label{propositioncontboulesouvertes}
For every $s<\cont(\Sigma,d,x)$, $B(x,s)$ is homeomorphic to an open disc.
\end{proposition}

\begin{proof}
Let $r \in (s,\cont(\Sigma,d,x))$. Let $H$ a homeomorphism between $\overline{B}(x,r)$ and the closed unit disc $\overline{D}(1)$. Since $H(B(x,s))$ is an open set of the plane, by the uniformization theorem, we only need to show that $H(B(x,s))$ is simply connected.

Let $\gamma:\mathbb{S}^1 \rightarrow H(B(x,s))$ be a continuous simple curve. By compactness, the curve $H^{-1}(\gamma)$ in $\Sigma$ is included in some ball $\overline{B}(x,s-\iota)$ for some $\iota>0$. Then $\gamma$ is in $H(\overline{B}(x,s-\iota))$, which is homeomorphic to a closed disc, hence simply connected. $\gamma$ is homotopic to zero in $H(\overline{B}(x,s-\iota))$, hence is homotopic to zero in $H(B(x,s))$ and this ends the proof.
\end{proof}

\subsection{An example: the case of an Euclidean cone}

Let us look at the case of a Euclidean cone of cone angle $\theta\in (0,2\pi)$. Recall that the curvature at the vertex $p$ of the cone is $k=2\pi - \theta$.
\begin{itemize}
 \item If $\theta>\pi$, then the contractibility radius at every point is infinite: every closed ball with center $x$ and radius $r$ is homeomorphic to a closed disc. Indeed, if $p$ is not in this ball, then it is a flat Euclidean ball; and if the vertex $p$ is in this ball, then we can cut the cone along the line passing by $p$ and $x$ to obtain the following picture (where we have drawn $S(x,r)$) :
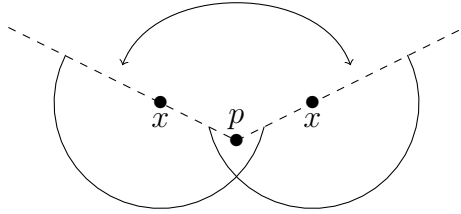
\begin{center}
\begin{tikzpicture}[scale=2]
\draw [dashed] (-1.5,0.75)--(0,0);
\draw [dashed] (0,0)--(1.5,0.75);
\draw (-0.5,0.25) node {$\bullet$};
\draw (-0.5,0.25) node [below] {$x$};
\draw (0.5,0.25) node {$\bullet$};
\draw (0.5,0.25) node [below] {$x$};
\draw (0,0) node {$\bullet$};
\draw (0,0) node [above] {$p$};
\draw [domain=153:347] plot ({-0.5+0.7*cos(\x)},{0.25+0.7*sin(\x)});
\draw [domain=-167:27] plot ({0.5+0.7*cos(\x)},{0.25+0.7*sin(\x)});
\draw[<->] (-0.75,0.5) to[out=70,in=110] (0.75,0.5) ;
\end{tikzpicture}
\end{center}

Gluing the cone back we see that $\overline{B}(x,r)$ is homeomorphic to a closed disc.

\begin{remark}
We already knew this result, by use of the lemma \ref{lemmecriterepratique}.
\end{remark}

\item But if $\theta<\pi$, the contractibility radius at $x$ goes to zero when $x$ tends to the vertex $p$ of the cone: if $x \neq p$, we have $\cont(\Sigma,d,x) < d(x,p)$.  Indeed, cut the cone along the line passing by $p$ and $x$. Then consider some $r$ such that the following situation occurs. The two balls, centered in $x$ and with radius $r$, have a non-empty intersection, and do not contain $p$ (this is possible because the angle at the vertex of the cone is less than $\pi$):
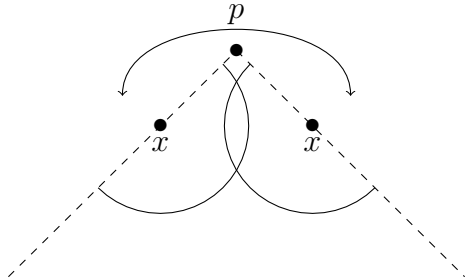
\begin{center}
\begin{tikzpicture}[scale=2]
\draw [dashed] (-1.5,-1.5)--(0,0);
\draw [dashed] (0,0)--(1.5,-1.5);
\draw (-0.5,-0.5) node {$\bullet$};
\draw (-0.5,-0.5) node [below] {$x$};
\draw (0.5,-0.5) node {$\bullet$};
\draw (0.5,-0.5) node [below] {$x$};
\draw (0,0) node {$\bullet$};
\draw (0,0.1) node [above] {$p$};
\draw [domain=225:405] plot ({-0.5+0.58*cos(\x)},{-0.5+0.58*sin(\x)});
\draw [domain=135:315] plot ({0.5+0.58*cos(\x)},{-0.5+0.58*sin(\x)});
\draw[<->] (-0.75,-0.3) to[out=90,in=90] (0.75,-0.3);
\end{tikzpicture}
\end{center}
Gluing the cone back, we see that $\overline{B}(x,r)$ is not homeomorphic to a closed disc (it does not contain the vertex $p$).

This example shows that a surface with a conical singularity of angle $<\pi$ has a contractibility radius equal to zero (so in corollary \ref{corollairesingconiques}, conical singularities of angles $<\pi$ are not allowed). This does not avoid an accumulation of singularities (as in Figure \ref{figureintro} in the introduction): in such a situation, the pointwise curvature of the singularities has to go to zero, hence the cone angles converge to $2\pi$. In the example described in figure \ref{figureintro}, we do have a lower bound for the contractibility radius.
\end{itemize}

The last example shows that in the Main Theorem, we can not have a conical singularity at $p$ of angle $< \pi$ (that is, we have $\omega^+(\{p\}) \leq \pi$). We can then wonder if hypothesis 1. is needed in the definition of $\M_\Sigma(A,c,\varepsilon,\delta)$; in other words, if $\cont(\Sigma,d) \geq c$ implies the existence of some $\varepsilon>0$ and $\delta>0$ such that for every $x \in \Sigma$ we have $\omega^+(B(x,\varepsilon)) \leq 2\pi-\delta$. This is not true, as the next example shows:

\begin{center}
\begin{tikzpicture}[scale=2]
\draw (0,0) -- (1,1.5);
\draw (0,0) -- (4,0);
\draw (1.5,0.5) -- (2.5,0.5);
\draw (1.5,0.7) -- (2.5,0.7);
\draw (1.7,1.1) -- (2.7,1.1);
\draw (1.5,0.5) -- (1.5,0.7);
\draw (2.5,0.5) -- (2.5,0.7);
\draw (1.5,0.7) -- (1.7,1.1);
\draw (2.5,0.7) -- (2.7,1.1);
\draw (2.7,1.1) -- (2.7,0.9);
\draw (2.7,0.9) -- (2.5,0.5);
\draw [dotted] (1.7,1.1) -- (1.7,0.9);
\draw [dotted] (1.5,0.5) -- (1.7,0.9);
\draw [dotted] (1.7,0.9) -- (2.7,0.9);
\end{tikzpicture}
\captionof{figure}{}
\end{center}
Consider the plane outside the box, plus the four sides of the box, plus the top of the box. This flat surface is homeomorphic to the plane, and has 8 conical singularities: at the top of the box, 4 singularities of angle $3\pi/2$ (that is, of curvature $\pi/2$), and at the bottom of the box, 4 singularities of angle $5\pi/2$ (that is, of curvature $-\pi/2$). If the height of the box is small enough, then the contractibility radius is infinite (all the balls $\overline{B}(x,r)$ are homeomorphic to closed discs). If we multiply the metric by a scale factor so that the box shrinks to a point $x$, the contractibility radius is still infinite, but $\omega^+$ converge weakly to a Dirac mass at $x$, with mass $2\pi$.

\subsection{Equivalence of a lower bound on the injectivity radius and on the contractibility radius} \label{sectionequivalencerayoninjectivitécontractibilite}

If $d$ is a Riemannian metric and $x\in \Sigma$, for every $r<\inj(x)$, the exponential map at the point $x$ is a homeomorphism between a closed disc on the plane and the closed ball with center $x$ and radius $r$. Then we clearly have the inequality
\[
\cont(\Sigma,d) \geq \inj(\Sigma,g).
\]
Conversly, the next proposition shows that in the classical Cheeger-Gromov compactness theorem, one can replace a lower bound on the injectivity radius by a lower bound on the contractibility radius. This result will not be used in the sequel.

\begin{proposition} \label{propositionequivalencerayoncontractibiliterayoninjectivite}
 Let $\Lambda>0$ and $c>0$. There exists a constant $i=i(\Lambda,c)$ verifying the following property. For every closed Riemannian surface $(\Sigma,g)$ with $|K_g| \leq \Lambda$ and $\cont(\Sigma,d_g) \geq c$, we have
\[
 \inj(\Sigma,g) \geq i.
\]
\end{proposition}
\begin{proof}
By a well-known result of W. Klingenberg, we know that we can find a lower bound for the injectivity radius in the following way:
\[
\inj(\Sigma,g) \geq \min\big(\pi/\sqrt{\Lambda} , \frac{1}{2} \cdot \mbox{length of the smallest closed geodesic} \big).
\]
Let $\gamma$ be a closed geodesic, of length $l$. We may assume $l < c$, and we look for a lower bound for $l$. $\gamma$ is included is some open ball $B$ with radius $l < \cont(\Sigma,d)$, which is homeomorphic to an open disc (see proposition \ref{propositioncontboulesouvertes}). Hence $\gamma$ bounds a domain $U \subset B$ homeomorphic to an open disc. Since the geodesic curvature of the boundary $\partial U$ of $U$ is identically zero, the Gauss-Bonnet's formula writes
\[
 2\pi = \int_U \K_g d\mathcal{A}_g,
\]
hence
\[
 2 \pi \leq \Lambda \cdot \area_g(U).
\]
But $U \subset B$, and since $\K_g \geq -\Lambda$, the area of the ball $B$ is at most the area of a ball of radius $l$ in a simply connected surface of constant curvature $-\Lambda<0$. This area is less than $\sqrt{\Lambda} \cdot \exp(l/\sqrt{\Lambda})$, so we get
\[
2\pi \leq \Lambda^{3/2} \cdot \exp(l/\sqrt{\Lambda}).
\]
Thus we have obtained
\[
 \inj(\Sigma,g) \geq \min\big( \frac{\pi}{\sqrt{\Lambda}} , \frac{c}{2}, \frac{\sqrt{\Lambda} \ln(2\pi/\Lambda^{3/2})}{2}\big),
\]
and this ends the proof.
\end{proof}

\section{Some results on surfaces with B.I.C.}

We prove two results for surfaces with B.I.C., which will be needed in the proof of the Main theorem.

\subsection{On the volume of balls}

We need to find an upper bound (respectively, a lower bound) for the volume of balls of radius $r$ in surfaces with B.I.C. In Riemannian geometry, this is a well-known fact that a lower bound (respectively, an upper bound) on the sectional curvature is sufficient. To generalize such results for surfaces with B.I.C., we need to have a property in Riemannian geometry which depends only on the curvature measure $\omega$, and not on the pointwise (Gauss) curvature. In \cite{Shioya}, T. Shioya proves the following:

\begin{theorem}\label{theoremevolumeboulesriemannien}
Let $(\Sigma,g)$ be a closed Riemannian surface, and let $x \in M$.
\begin{enumerate}
\item Let $r>0$. We have
\[
\area(B(x,r)) \leq \big( 2\pi + \omega^-(B(x,r)) \big) r^2/2,
\]
where
\[
\omega^-(B(x,r)) = \int_{B(x,r)} \K_g^- d\mathcal{A}_g
\]
is the non-positive part of the curvature of $B(x,r)$.
\item Let $r > 0$ such that $\cont(\Sigma,d_g,x) \geq r$ (that is for every $s<r$, $\overline{B}(x,s)$ is homeomorphic to a closed disc). Then
\[
\area(B(x,r)) \geq \big(2\pi - \omega^+(B(x,r)) \big) r^2/2,
\]
where
\[
\omega^+(B(x,r)) = \int_{B(x,r)} \K_g^+ d\mathcal{A}_g
\]
is the non-negative part of the curvature of $B(x,r)$.
\end{enumerate}
\end{theorem}

\begin{remark}
The hypothesis needed to get a lower bound for the area is necessary: an Euclidean cylinder is flat (so $\omega^+=0$), and when the radius of the cylinder goes to zero, the area of a ball of fixed radius goes to zero.
\end{remark}

As a direct consequence, we obtain the

\begin{corollary} \label{corollairevolumeboulesalexandrov}
Let $(\Sigma,d)$ be a surface with B.I.C., and let $x \in \Sigma$.
\begin{enumerate}
\item Let $r>0$. We have
\[
\area(B(x,r)) \leq \big( 2\pi + \omega^-(B(x,r)) \big) r^2/2.
\]
\item Let $r> 0$ such that $\cont(\Sigma,d,x) \geq r$. Then
\[
\area(B(x,r)) \geq \left( 2\pi - \omega^+(B(x,r)) \right) r^2/32.
\]
\end{enumerate}
\end{corollary}

The second inequality is not optimal, since $r^2/2$ is replaced by $r^2/32$; see the proof below.

\begin{remark}
Consider an Euclidean cone, with vertex $x$ and cone angle $\theta \in (0,\infty)$. The curvature measure is $\omega = (2\pi-\theta)\delta_x$ (where $\delta_x$ is the Dirac mass at $x$). We have
\[
2\pi + \omega^-(B(x,r)) = 2\pi + \max(\theta-2\pi,0) = \max(\theta,2\pi)
\]
and
\[
2\pi - \omega^+(B(x,r)) = 2\pi - \max(2\pi-\theta,0) = 2\pi + \min(\theta-2\pi,0) = \min(\theta,2\pi).
\]
Hence we have
\[
2\pi - \omega^+(B(x,r)) \leq \theta \leq 2\pi + \omega^-(B(x,r)),
\]
and since the area of $B(x,r)$ is $\theta r^2/2$ this gives
\[
\left(2\pi - \omega^+(B(x,r))\right)r^2/2 \leq \area(B(x,r)) \leq \left(2\pi + \omega^-(B(x,r))\right)r^2/2.
\]
\end{remark}

\begin{proof}[Proof of corollary \ref{corollairevolumeboulesalexandrov}.]
Let $g_m$ be a sequence of Riemannian metrics on $\Sigma$, such that $d_m := d_{g_m} \rightarrow d$ uniformly on $\Sigma$ when $m$ goes to infinity. Let $d\mathcal{A}_m$ (resp. $d\mathcal{A}$) be the area measure for $(\Sigma,d_m)$ (resp. $(\Sigma,d)$), and let $\omega_m=\omega_m^+ - \omega_m^-$ (resp. $\omega = \omega^+ - \omega^-$) be the curvature measure of $(\Sigma,d_m)$ (resp. $(\Sigma,d)$), with its non-negative and non-positive parts. Then we know (see \cite{Reshetnyak_livre}, theorems 8.1.9 and 8.4.3) that $d\mathcal{A}_m \rightarrow d\mathcal{A}$ and $\omega_m \rightarrow \omega$, weakly on $\Sigma$. We do not necessarily have $\omega_m^+ \rightarrow \omega^+$ and $\omega_m^- \rightarrow \omega^-$ (weakly on $\Sigma$), but we can choose a sequence of metrics $(g_m)$ verifying these properties; see \cite{AZ} and \cite{Reshetnyak_livre}.

In the sequel we will use the following classical property of converging measures: if $U$ is an open set and $K \subset U$ is a compact set, and if we have a (weakly) convergence of non-negative measures $\mu_m \rightarrow \mu$, then for every $\varepsilon>0$, for $m$ large enough we have $\mu(K) < \mu_m(U) + \varepsilon$ and $\mu_m(K) < \mu(U) + \varepsilon$.

\noindent\underline{Proof of 1.}
Remark that $\cup_{\varepsilon>0} \overline{B}(x,r-\varepsilon) = B(x,r)$, so
\[
 \area(B(x,r)) = \lim_{\varepsilon\rightarrow 0} \area(\overline{B}(x,r-\varepsilon)).
\]
Now, let $\varepsilon>0$. For $m$ large enough we have
\[
  \area(\overline{B}(x,r-\varepsilon)) < \area_m(B(x,r-3\varepsilon/4)) + \varepsilon,
\]
and with $B(x,r-3\varepsilon/4) \subset B_m(x,r-\varepsilon/2)$ for $m$ large enough we get
\[
  \area(\overline{B}(x,r-\varepsilon))  <  \area_m(B_m(x,r-\varepsilon/2)) + \varepsilon.
\]
Using theorem \ref{theoremevolumeboulesriemannien} with $r - \varepsilon/2$ we obtain
\[
\area(\overline{B}(x,r-\varepsilon)) \leq \big(2\pi + \omega_m^-(B_m(x,r-\varepsilon/2)) \big) (r-\varepsilon/2)^2/2 + \varepsilon,
\]
and with $B_m(x,r-\varepsilon/2) \subset B(x,r-\varepsilon/4)$ for $m$ large enough we get
\begin{eqnarray*}
\area(\overline{B}(x,r-\varepsilon)) & \leq & \big(2\pi + \omega_m^-(B(x,r-\varepsilon/4)) \big) (r-\varepsilon/2)^2/2 + \varepsilon \\
& \leq & \big(2\pi + \omega_m^-(B(x,r-\varepsilon/4)) \big) r^2/2 + \varepsilon.
\end{eqnarray*}
For $m$ large enough we have
\[
\omega_m^-(B(x,r-\varepsilon/4)) \leq \omega^-(B(x,r)) + \varepsilon,
\]
hence we obtain
\[
 \area(\overline{B}(x,r-\varepsilon)) \leq \big(2\pi + \omega^-(B(x,r)) + \varepsilon \big) r^2/2 + \varepsilon,
\]
and letting $\varepsilon \rightarrow 0$ this ends the proof.

\noindent\underline{Proof of 2.}
The assertion is trivial if $\omega^+(B(x,r)) \geq 2\pi$, so we may assume $\omega^+(B(x,r)) < 2\pi$. We can not directly apply theorem \ref{theoremevolumeboulesriemannien} for the Riemannian metrics $d_m$, because we may not have $\cont(\Sigma,d_m,x) \geq r$ (this is the reason why $r^2/2$ is replaced by $r^2/32$). Let $y \in B(x,r)$ be some point with $d(x,y) = r/2$. Then we have
\[
 B(x,r/4) \bigcap B(y,r/4) = \emptyset,
\]
and
\[
 B(x,r/4) \bigcup B(y,r/4) \subset B(x,r).
\]
Since $\omega^+(B(x,r)) < 2\pi$, this shows that we have $\omega^+(B(x,r/4)) <\pi$, or $\omega^+(B(y,r/4)) <\pi$. Let $z$ ($z = x$ or $y$) be a point with
\begin{equation} \label{equationomegaplus}
\omega^+(B(z,r/4)) < \pi. 
\end{equation}
Let $\varepsilon>0$. After passing to a subsequence, we may assume
\begin{equation} \label{equationmesuresconvergentes1}
\overline{B}_m(z,r/4-\varepsilon) \subset B(z,r/4-\varepsilon/2)
\end{equation}
and with equation (\ref{equationomegaplus}) we may also assume
\begin{equation} \label{equationmesuresconvergentes2}
 \omega_m^+(B(z,r/4-\varepsilon/2)) < \pi. 
\end{equation}
Let $U := B(z,r/4-\varepsilon/2)$. $U$ is homeomorphic to an open disc, and we have $\overline{B}_m(z,r/4-\varepsilon) \subset U$ and $\omega_m^+(U-\{z\})\leq \omega_m^+(U)<\pi$. We can then apply lemma \ref{lemmecriterepratique} to obtain
\[
\cont(\Sigma,d_m,z) \geq r/4-\varepsilon.
\]
We can then apply theorem \ref{theoremevolumeboulesriemannien}, with the metric $d_m$, the point $z$ and the radius $r/4-\varepsilon$: for $m$ large enough we have
\begin{equation} \label{equationapplicationtheoremdm}
\area_m(B_m(z,r/4-\varepsilon)) \geq  \big( 2\pi - \omega_m^+(B_m(z,r/4-\varepsilon)) \big) (r/4-\varepsilon)^2/2. 
\end{equation}
For $m$ large enough we also have
\[
 B_m(z,r/4-\varepsilon) \subset B(x,r-\varepsilon),
\]
so for $m$ large enough we get
\[
\left\{
\begin{array}{l}
\area(B(x,r)) \geq \area_m(B(x,r-\varepsilon)) - \varepsilon \\
\omega^+(B(x,r)) \geq \omega_m^+(B(x,r-\varepsilon)) - \varepsilon,
\end{array}
\right.
\]
which gives
\[
\left\{
\begin{array}{l}
\area(B(x,r)) \geq \area_m(B_m(z,r/4-\varepsilon)) - \varepsilon \\
\omega^+(B(x,r)) \geq \omega_m^+(B_m(z,r/4-\varepsilon)) - \varepsilon.
\end{array}
\right.
\]
With equation (\ref{equationapplicationtheoremdm}) we get
\[
\area(B(x,r)) \geq  \big( 2\pi - \omega^+(B(x,r)) - \varepsilon \big) (r/4-\varepsilon)^2/2 - \varepsilon,
\]
and letting $\varepsilon \rightarrow 0$ this ends the proof.
\end{proof}

\subsection{An upper bound for the length of a line segment}

We want to find an upper bound for the length of a line segment, for a singular Riemannian metric which has "no harmonic term", that is, when $g = e^{2V[\omega](z)} |dz|^2$ for some Radon measure $\omega$.

First, consider a Riemannian metric on $D(1/2)$ with a conical singularity at 0: $g = |z|^{2\beta} |dz|^2$ (for some $\beta>-1$), and let $\gamma$ be the line segment joigning 0 and a point $z \in D(1/2)$. The length of $\gamma$ is
\[
L(\gamma) = \int_0^1 |tz|^\beta |z| dt = \frac{1}{1+\beta}|z|^{1+\beta}.
\]
Moreover, the curvature measure is $\omega = -2\pi \beta \cdot \delta_0$ (where $\delta_0$ is the Dirac mass at 0), so the non-negative part of the curvature measure is $\omega^+(D(1/2)) = \max(0,-2\pi \beta)=2\pi \beta^-$ and we get
\[
1+\beta \geq 1 - \beta^- =1-\omega^+(D(1/2))/2\pi,
\]
hence
\[
L(\gamma) \leq \frac{1}{1-\omega^+(D(1/2))/2\pi}|z|^{1-\omega^+(D(1/2))/2\pi}.
\]
The next proposition shows how we can extend this result to arbitrary curvature measures.

\begin{proposition} \label{propositionmajorationdistance}
Let $\omega$ be a Radon measure defined in $D(1/2)$, with $\omega^+(D(1/2)) <2\pi$. Let $z,z' \in D(1/2)$, and let $\gamma(t) := (1-t)z + tz'$ be the line segment $[zz']$. Let $L(\gamma)$ be the length of this line segment for the singular metric $g = e^{2V[\omega](z)} |dz|^2$, then
\begin{equation} \label{inegalitedistance}
L(\gamma) \left( = \int_0^1 e^{V[\omega](\gamma(t))} |z-z'| dt) \right) \leq \frac{2}{1-\omega^+(D(1/2))/2\pi} |z-z'|^{1-\omega^+(D(1/2))/2\pi}.
\end{equation}
\end{proposition}

\begin{proof}~
\textbf{First step.}
We first show that this is sufficient to prove the proposition, in the case where $\omega$ is a sum of Dirac masses. If so, let $\omega$ be a Radon measure with $\omega^+(D(1/2))<2\pi$, and write $\omega$ as $\omega = \omega^+ - \omega^-$, where $\omega^+$ and $\omega^-$ are non-negative Radon measures. Let $\omega_m^+$ and $\omega_m^-$ be a sequence of sums of Dirac masses such that $\omega_m^+ \rightarrow \omega^+$ and $\omega_m^- \rightarrow \omega^-$ weakly, and let $\omega_m := \omega_m^+ - \omega_m^-$.

Let $L_m(\gamma)$ be the length of the line segment $\gamma$ for the singular metric $g_m = e^{2V[\omega_m](z)} |dz|^2$. For almost every $t \in [0,1]$ we have
\begin{multline*}
 V[\omega_m](\gamma(t)) = \iint_{D(1/2)} \left(\frac{-1} {2\pi}\right) \ln|\gamma(t)-\xi| d\omega_m(\xi) \underset{m \rightarrow \infty} \longrightarrow \\
 \iint_{D(1/2)} \left(\frac{-1}{2\pi}\right) \ln|\gamma(t)-\xi| d\omega(\xi) = V[\omega](\gamma(t)),
\end{multline*}
hence by Fatou's lemma we get
\begin{eqnarray*}
 L(\gamma) & = & \int_0^1 e^{V[\omega](\gamma(t))} |z-z'| dt \\
 & = & \int_0^1 \liminf_{m \rightarrow \infty} \big( e^{V[\omega_m](\gamma(t))} |z-z'| \big) dt \\
 & \leq & \liminf_{m \rightarrow \infty} \big( \int_0^1 e^{V[\omega_m](\gamma(t))} |z-z'| dt \big) \\
& = & \liminf_{m \rightarrow \infty} L_m(\gamma).
\end{eqnarray*}
If we apply the inequality (\ref{inegalitedistance}) with the measures $\omega_m$ (which are a sum of Dirac masses, with $\omega_m^+(D(1/2))<2\pi$ for $m$ is large enough), we get
\[
 L_m(\gamma) \leq \frac{2}{1-\omega_m^+(D(1/2))/2\pi}|z-z'|^{1-{\omega_m^+(D(1/2))}/2\pi},
\]
so
\begin{eqnarray*}
L(\gamma) & \leq & \liminf_{m\rightarrow\infty} \left(\frac{2}{1-\omega_m^+(D(1/2))/2\pi}|z-z'|^{1-{\omega_m^+(D(1/2))}/2\pi}\right) \\
& = & \frac{2}{1-\omega^+(D(1/2))/2\pi} |z-z'|^{1-\omega^+(D(1/2))/2\pi}.
\end{eqnarray*}

\textbf{Second step.}
We may now assume that $\omega$ is a sum of Dirac masses: there exists $p_1,...,p_n \in D(1/2)$ and $k_1,...,k_n \in \R$ such that
\[
\omega = \sum_{s=1}^n k_s \delta_{p_s}
\]
($\delta_{p_s}$ is the Dirac mass at $p_s$), and
\[
\omega^+(D(1/2)) = \sum_{s=1}^n k_s^+ < 2\pi.
\]
For almost every $z \in D(1/2)$ we have
\[
 V[\omega](z) = \iint_D \left(\frac{-1}{2\pi}\right) \ln|z-\xi| d\omega(\xi) = \sum_{s=1}^n \big( \frac{-k_s}{2\pi} \big) \ln|z-p_s| \leq \sum_{s=1}^n \big( \frac{-k_s^+}{2\pi} \big) \ln|z-p_s|,
\]
so if we set $\beta_s := -k_s/2\pi$, then we have $\beta_s^-= \max(0,-\beta_s) = \max(0,k_s/2\pi)=k_s^+/2\pi$, hence
\[
 e^{V[\omega](z)} \leq \prod_{s=1}^n |z -p_s|^{-\beta_s^-}.
\]
We then have
\[
L(\gamma) = \int_0^1 e^{V[\omega](\gamma(t))} |z-z'| dt \leq \int_0^1 \big(\prod_{s=1}^n |\gamma(t) -p_s|^{-\beta_s^-} \big) |z-z'| dt.
\]
Let $S := \{s \mbox{ such that } \beta_s^- >0 \}$.

If $S = \emptyset$, then $\beta_s^-=0$ for every $s$: we have $\omega^+(D(1/2))=0$, and the last inequality shows that $L(\gamma) \leq |z-z'|$, so the inequality (\ref{inegalitedistance}) is true (the factor 2 will be needed for the case $S \neq \emptyset$).

Then we may assume $S \neq \emptyset$. Let  $M := -\sum_{s \in S} \beta_s^-$: by hypothesis we have $-1<M<0$. For $s \in S$, let $q_s := -M/\beta_s^-$: we have
\[
q_s \geq 1 \mbox{ and } \sum_{s \in S} \frac{1}{q_s}= 1.
\]
Since $|\gamma(t)-p_s|^{-\beta_s^-} \leq 1$ if $\beta_s^- \leq 0$, we can apply Hölder's inequality as follows:
\begin{eqnarray}
\int_0^1 \big(\prod_{s=1}^n |\gamma(t) -p_s|^{-\beta_s^-} \big) dt & \leq & \int_0^1 \big(\prod_{s \in S} |\gamma(t) -p_s|^{-\beta_s^-} \big) dt \\
 & \leq & \prod_{s \in S} \big(\int_0^1 |\gamma(t) -p_s|^{-q_s \beta_s^-} dt \big)^{1/q_s} \\ \label{inegalitelemme}
 & = & \prod_{s \in S} \big(\int_0^1 |\gamma(t) -p_s|^M dt \big)^{1/q_s}.
\end{eqnarray}
Now, fix some $s\in S$ and consider $\int_0^1 |\gamma(t) -p_s|^M dt$. Let $p'_s$ be the projection of the point $p_s$ on the line $(zz')$:
\begin{center}
\begin{tikzpicture}
\draw (-4,0) -- (2,0);
\draw (-1,0) node {$\times$};
\draw (-1,-0.1) node [below] {$z$};
\draw (1,0) node {$\times$};
\draw (1,0) node [below] {$z'$};
\draw (0.2,0) node {$\bullet$};
\draw (0.2,0) node [below] {$\gamma(t)$};
\draw (-3,1) node {$\bullet$};
\draw (-3,1) node [left] {$p_s$};
\draw (-3,0) node {$\bullet$};
\draw (-3,0) node [below] {$p'_s$};
\draw (-3,1) -- (-3,0);
\draw (-3,0.2) -- (-2.8,0.2);
\draw (-2.8,0.2) -- (-2.8,0);
\end{tikzpicture}
\captionof{figure}{}
\end{center}
We have $|\gamma(t)-p'_s| \leq |\gamma(t)-p_s|$, so with $M<0$ we get
\[
\int_0^1 |\gamma(t) -p_s|^M dt \leq \int_0^1 |\gamma(t) -p'_s|^M dt.
\]
Since $p'_s$ belongs to the line $(zz')$, we can write $p'_s = (1-\lambda_s) z+ \lambda_s z'$ for some $\lambda_s \in \R$, hence
\[
|\gamma(t) -p'_s| = |(1-t)z + tz' - (1-\lambda_s) z - \lambda_s z'| = |t-\lambda_s| |z-z'|.
\]
Moreover, we easily see that the map $\lambda \mapsto \int_0^1 |t-\lambda|^M dt$ admits its maximum for $\lambda = 1/2$ (recall that $M<0$), hence
\[
 \int_0^1 |t-\lambda_s|^M dt \leq \int_0^1 |t-\frac{1}{2}|^M dt = \frac{1}{(M+1)2^{M+1}},
\]
so we obtain
\[
\int_0^1 |\gamma(t) -p_s|^M dt \leq \frac{|z-z'|^M}{(M+1)2^M}.
\]
If we put this in the inequality (\ref{inegalitelemme}), with $\sum_{s \in S} 1/q_s =1$, we get
\[
\int_0^1 \big(\prod_{s \in S} |\gamma(t) -p_s|^{-\beta_s^-} \big) dt \leq \frac{|z-z'|^M}{(M+1)2^M}.
\]
Since $M>-1$ we have $2^M \geq 1/2$, and with the equality $M = -\omega^+(D(1/2))/2\pi$ we obtain
\[
L(\gamma) \leq \frac{2}{1-\omega^+(D(1/2))/2\pi}|z-z'|^{1-\omega^+(D(1/2))/2\pi}.
\]
\end{proof}

\section{Preliminary properties of $\M_\Sigma(A,c,\varepsilon,\delta)$}

Before starting the proof of the Main theorem, we prove some important preliminary properties for the set $\M_\Sigma(A,c,\varepsilon,\delta)$.

\subsection{Another definition of $\M_\Sigma(A,c,\varepsilon,\delta)$}

By analogy with Cheeger-Gromov's convergence theorem, where we can replace a bound on the volume by a bound on the diameter, the following proposition shows that in the Main theorem, we can replace a bound on the area by a bound on the diameter and on the total curvature (this is corollary \ref{corollairetheoremeprincipal}). Recall that $|\omega|(\Sigma,d) = \omega^+(\Sigma,d) + \omega^-(\Sigma,d)$.
\begin{proposition} \label{propositionequivalencedefinition}
Let $\Sigma$ be a closed surface and let $c,\varepsilon,\delta>0$. Let $d$ be a metric with B.I.C. on $\Sigma$, verifying the properties 1. and 2. in the definition of $\M_\Sigma$, that is
\begin{equation*}
\begin{array}{l}
1. \mbox{ for every } x \in \Sigma, ~ \omega^+ (B(x,\varepsilon)) \leq 2\pi -\delta \\
2. ~\cont({\Sigma, d}) \geq c.
\end{array}
\end{equation*}
Then:
\begin{itemize}
 \item for every $A>0$, there exists some positive constants $D$ and $\Omega$ such that
\begin{equation*}
\area(\Sigma,d) \leq A  ~ (\Longleftrightarrow d \in \M_\Sigma(A,c,\varepsilon,\delta)) \implies
\left\lbrace
\begin{array}{rll}
\diam(\Sigma,d) \leq D\\
|\omega|(\Sigma,d) \leq \Omega.
\end{array}
\right.
\end{equation*}

\item Conversely, for every $\Omega,D>0$, there exists a positive constant $A$ such that
\begin{equation*}
\left.
\begin{array}{ccc}
\diam(\Sigma,d) \leq D\\
|\omega|(\Sigma,d) \leq \Omega
\end{array}
\right\rbrace \implies
\area(\Sigma,d) \leq A  ~ (\Longleftrightarrow d \in \M_\Sigma(A,c,\varepsilon,\delta)).
\end{equation*}
\end{itemize}
\end{proposition}

\begin{proof}
To prove the first property, consider some $d \in \M_\Sigma(A,c,\varepsilon,\delta)$. Let $B(x_i,\varepsilon/2)$, for $i \in \{1,...,N\}$, be a maximal number of disjoint balls of radius $\varepsilon/2$ in $\Sigma$. By property \ref{corollairevolumeboulesalexandrov}, all these balls have an area bounded by below; since the area of $(\Sigma,d)$ is bounded by above, the integer $N$ is also bounded. If the diameter was arbitrarily large, then we could find an arbitrarily large number of disjoint balls: this shows that there exists some $D>0$ such that $\diam(\Sigma,d) \leq D$. And by an elementary covering argument, the $N$ balls $B(x_i,\varepsilon)$ cover $\Sigma$; but the positive curvature of these balls are bounded by above, so the positive curvature $\omega^+(\Sigma,d)$ of $\Sigma$ is also bounded by above. The Gauss-Bonnet formula gives us $\omega^+(\Sigma,d) - \omega^-(\Sigma,d) = 2\pi \chi(\Sigma)$ (where $\chi(\Sigma)$ is the Euler characteristic of $\Sigma$), so $\omega^-(\Sigma,d)$ is as well bounded by above, and this shows that there exists some $\Omega>0$ such that $|\omega|(\Sigma,d) \leq \Omega$.

Conversely, if $\diam(\Sigma,d) \leq D$ and $|\omega|(\Sigma,d) \leq \Omega$, then $\Sigma$ is equal to some ball $B(x,D+1)$. Since $\omega^-(\Sigma)$ is bounded, by corollary \ref{corollairevolumeboulesalexandrov}, we know that the area of such a ball is bounded by above, and this shows that there exists a constant $A>0$ such that $\area(\Sigma,d) \leq A$.
\end{proof}

From now on, we fix a closed surface $\Sigma$ and $A,c,\varepsilon,\delta>0$; we then have some positive constants $D$ and $\Omega$ verifying the first part in proposition \ref{propositionequivalencedefinition}.

\subsection{A bound for the harmonic term}

This section is devoted to the proof of the following theorem, which gives a bound for the harmonic term $h$ when we express (locally) any metric $d \in \M_\Sigma(A,c,\varepsilon,\delta)$ as a singular Riemannian metric $g = e^{2V[\omega]+2h}|dz|^2$. This born has to be \emph{uniform}, that is independent of the metric $d \in \M_\Sigma(A,c,\varepsilon,\delta)$. This theorem is very important in the sequel, and relies on conformal geometry of an annulus (see the Appendix).

\begin{theorem}\label{theoremetermeharmoniqueborne}
Let $K \subset D(1/2)$ be a compact set. There exists a constant $M(K) = M(K,\Sigma,A,c,\varepsilon,\delta)$ verifying the following property.

Let $d$ be a metric in $\M_\Sigma(A,c,\varepsilon,\delta)$, and let $H : B(x,\varepsilon) \rightarrow D(1/2)$ be a conformal chart, with $H(x)=0$. As usual, we denote by $h$ the harmonic term for the metric in this chart (see theorem-definition \ref{theoremedefinitionconformalcharts}). We then have
\[
|h(z)| \leq M(K) \mbox{ for every } z \in K.
\]
\end{theorem}

We first give an \emph{upper} bound for $h$. The (explicit) upper bound will be used in the next section:

\begin{proposition} \label{theoremeinitialtermeharmoniquemajore}
Under the hypothesis of theorem \ref{theoremetermeharmoniqueborne} we have, for every $z \in D(1/2)$,
\begin{equation} \label{inegaliteexponentielh}
e^{h(z)} \leq \frac{\varepsilon}{(1/2-|z|)^{1+\Omega/2\pi}} \cdot C(\Omega),
\end{equation}
where $C(\Omega) := \sqrt{1+\Omega/2\pi} \cdot e^{\Omega/4\pi}$.
\end{proposition}

\begin{proof}
Let $a \in D(1/2)$, and set $s = 1/2-|a|>0$: we have $D(a,s) \subset D(1/2)$. Let $u := V[\omega_H] + h$, so that the singular Riemannian metric reads $g = e^{2u}|dz|^2$. By Jensen's inequality we get
\[
\exp\big( \iint_{D(a,s)} 2u(z) \frac{d\lambda(z)}{\pi s^2} \big) \leq \iint_{D(a,s)} e^{2u(z)} \frac{d\lambda(z)}{\pi s^2} \leq \frac{\area(B(x,\varepsilon))}{\pi s^2},
\]
and by corollary \ref{corollairevolumeboulesalexandrov} we have $\area(B(x,\varepsilon)) \leq (\pi + \Omega/2) \cdot \varepsilon^2$, hence
\[
 \frac{1}{\pi s^2} \iint_{D(a,s)} 2u(z) d\lambda(z) \leq \ln\big( (1+\Omega/2\pi) \frac{\varepsilon^2}{s^2} \big).
\]
$h$ is harmonic, hence it verifies the mean-value property:
\[
h(a) = \frac{1}{\pi s^2} \iint_{D(a,s)}  h(z) d\lambda(z),
\]
so we get
\begin{eqnarray}
h(a) & = & \frac{1}{\pi s^2} \iint_{D(a,s)} u(z) d\lambda(z) - \frac{1}{\pi s^2} \iint_{D(a,s)} V[\omega_H](z) d\lambda(z) \\ \label{inegappendiceA1}
  & \leq & \frac{1}{2} \ln\big( (1+\Omega/2\pi) \frac{\varepsilon^2}{s^2} \big) - \frac{1}{\pi s^2} \iint_{D(a,s)} V[\omega_H] (z) d\lambda(z),
\end{eqnarray}
and to conclude we need to find a lower bound for $\iint_{D(a,s)} V[\omega_H] (z) d\lambda(z)$.

We know that $V[\omega_H](z) = V[\omega_H^+](z) - V[\omega_H^-](z) \geq -V[\omega_H^-](z)$ for almost every $z \in D(1/2)$, so
\begin{eqnarray}
\iint_{D(a,s)} V[\omega_H](z) d\lambda(z) & \geq & -\iint_{D(a,s)} \big( \iint_{D(1/2)} \left(\frac{-1}{2\pi} \right) \ln|z-\xi| d\omega_H^-(\xi) \big) d\lambda(z)
\\ \label{inegappendiceA2} 
 & = & -\iint_{D(1/2)} \big( \iint  _{D(a,s)} \left(\frac{-1}{2\pi} \right) \ln|z-\xi| d\lambda(z) \big) d\omega_H^-(\xi).
\end{eqnarray}
And for $\xi \in D(1/2)$, we have
\[
 \iint_{D(a,s)} \left(\frac{-1}{2\pi} \right) \ln|z-\xi| d\lambda(z) = \iint _{D(a-\xi,s)} \left(\frac{-1}{2\pi} \right) \ln|z| d\lambda(z),
\]
and we easily see that this function of $\xi$ is maximum for $\xi=a$ (that is when the disc is centered in 0). So for every $\xi \in D(1/2)$,
\begin{eqnarray*}
\iint_{D(a,s)} \left(\frac{-1}{2\pi} \right) \ln|z-\xi| d\lambda(z) & \leq & \iint _{D(0,s)} \left(\frac{-1}{2\pi} \right) \ln|z| d\lambda(z) \\
 & = & - \int_0^s r \ln r dr  = \frac{s^2}{4} - \frac{s^2 \ln s}{2}.
\end{eqnarray*}
Using inequality (\ref{inegappendiceA2}) we get
\[
 \iint_{D(a,s)} V[\omega_H](z) d\lambda(z) \geq \big( -\frac{s^2}{4} + \frac{s^2 \ln s}{2} \big) \cdot \omega_H^-(D) \geq \big( -\frac{s^2}{4} + \frac{s^2 \ln s}{2} \big) \cdot \Omega
\]
(recall that $\Omega$ satisfies $\omega_H^+(D) + \omega_H^-(D) \leq \Omega$). With the inequality (\ref{inegappendiceA1}) we obtain
\[
 h(a) \leq \frac{1}{2} \ln\big( (1+\Omega/2\pi) \frac{\varepsilon^2}{s^2} \big) +  \frac{\Omega}{4\pi} - \frac{\Omega \ln s}{2\pi},
\]
and this ends the proof.
\end{proof}

We now prove that the image by the conformal chart $H$ of the ball $B(x,\varepsilon/2)$ can not go to close to the boundary of the disc $D(1/2)$. We use the results stated in the appendix, by looking at the annulus $D(1/2)-H(\overline{B}(x,\varepsilon/2))$. On the one hand, by definition, this annulus has a modulus bounded by below; and on the other hand, by Grötzsch's theorem (see theorem \ref{theoremegrotzsch} in the appendix), if $H(\overline{B}(x,\varepsilon/2)$ was arbitrarily close to $\partial D(1/2)$, then the modulus of $D(1/2)-H(\overline{B}(x,\varepsilon/2))$ would be arbitrarily close to zero. We need this result to prove theorem \ref{theoremetermeharmoniqueborne}, but we will also need it later (see theorem \ref{theoremerecouvrement}).
\begin{proposition}\label{propositioninclusionboulerayonmoitie}
Let $r<1/2$ be such that the modulus of the Grötzsch annulus $G(2r)$ verifies $\module(G(2r)) < \varepsilon^2/4A$ (see the appendix). Let $d \in \M_\Sigma(A,c,\varepsilon,\delta)$, and let $H : B(x,\varepsilon) \rightarrow D(1/2)$ be a conformal chart, with $H(x)=0$. Then
\[
H(B(x,\varepsilon/2)) \subset D(r).
\]
\end{proposition}

\begin{proof}
The following subset of $\mathbb{C}$:
\[
U := D(1/2) - H(\overline{B}(x,\varepsilon/2))
\]
is a topological annulus (in the sense given in the appendix, see definition \ref{definitionannulus}). We know that $H$ is an isometry between $(B(x,\varepsilon),d_{|B(x,\varepsilon)})$ and $(D(1/2),d_{\omega_H,h})$. Recall that the modulus of $U$ is
\[
\module(U) = \sup_{\rho} \frac{\inf_{\gamma \in \Gamma} L_\rho (\gamma)^2}{A_\rho (U)},
\]
(see the appendix), where $\Gamma$ is the set of continuous simple curves $\gamma$ in $U$, parametrized by arc-length, joigning $\partial D(1/2)$ and a point in $H(\overline{B}(x,\varepsilon/2))$. We take $\rho := e^{V[\omega_H] + h}$.

For every $\gamma \in \Gamma$, we can not be sure that the $\rho$-length of $\gamma$ is equal to its $d$-length. But we know that, by definition of the length distance $d_{\omega_H,h}$ (let us recall here that $\gamma(0) \notin D(1/2)$) :
\begin{eqnarray*}
L_\rho (\gamma) & \geq & \limsup_{t \rightarrow 0}  d_{\omega_H,h}(\gamma(t),\gamma(1)) \\
 & = & \limsup_{t \rightarrow 0}  d_{B(x,\varepsilon)}(H^{-1}(\gamma(t)),H^{-1}(\gamma(1))) \\
  & \geq & \limsup_{t \rightarrow 0}  d(H^{-1}(\gamma(t)),H^{-1}(\gamma(1)))
\end{eqnarray*}
(the last inequality is a direct consequence of the definition of the induced metric). Consider a sequence $t_k \rightarrow 0$ such that $H^{-1}(\gamma(t_k))$ converges to $y \in \Sigma$ : since $\gamma(0) \notin D(1/2)$, we have $y \notin B(x,\varepsilon)$, and since $H^{-1}(\gamma(1)) \in \overline{B}(x,\varepsilon/2)$ this gives
\[
L_\rho(\gamma) \geq d(y,H^{-1}(\gamma(1))) \geq \varepsilon/2.
\]
Since the $\rho-$area of $U$ is less than or equal to the area of $(\Sigma,d)$ we get
\[
\module(U) \geq \frac{(\varepsilon/2)^2}{A} = \varepsilon^2/4A.
\]
We prove the lemma by contradiction: assume we have $H(\overline{B}(x,\varepsilon/2)) \not\subset D(r)$. Then, after a rotation (such that the complex number of maximum modulus of $H(\overline{B}(x,\varepsilon/2))$, which is greater than or equal to $r$, belongs to the real axis), and after an homothety with scale factor 2, we see that $U$ is conformal to an annulus $U' \subset D(1)$, not containing 0 (since $H(x)=0$) and $2r$. By Grötzsch's theorem (see theorem \ref{theoremegrotzsch} in the appendix), we have $\module(U) = \module(U') \leq \module (G(2r))$, which is impossible with the choice of $r$ we made. We then have $H(\overline{B}(x,\varepsilon/2)) \subset D(r)$ and this ends the proof of the proposition.
\end{proof}

Let $p^* := 4\pi/\delta>1$. To obtain a lower bound for the harmonic term $h$, we need the following lemma, which gives a lower bound for a certain integral involving $h$:

\begin{lemma} \label{lemmeminorationvolume}
Under the hypothesis of theorem \ref{theoremetermeharmoniqueborne}, we have
\[
\left(\frac{\delta^2 \varepsilon^2}{512\pi^2}\right)^{p^*} \leq \iint_{D(r)} e^{2p^* h(z)} d\lambda(z).
\]
\end{lemma}

\begin{proof}
We will use the following proposition, see \cite{Troyanov_concentrationcompacite}:
\begin{proposition}[M. Troyanov] \label{theoremetroyanov}
Let $\nu$ be a non-negative Radon measure defined in $D(1/2)$. Suppose there exists some $p>1$ such that ${\nu(D(1/2))} <2\pi/p$. Then
\[
\iint_{D(1/2)} e^{2pV[\nu](z)} d\lambda(z) \leq \frac{\pi}{1-\frac{p}{2\pi} \nu(D(1/2))}.
\]
\end{proposition}
Let $p>1$ be such that $1/p + 1/p^*=1$. By property 1. in the definition of $\M_\Sigma (A,c,\varepsilon,\delta)$ we have
\[
\omega_H^+(D(1/2)) = \omega^+(B(x,\varepsilon)) \leq 2\pi-\delta = 2\pi(1-2/p^*) = 2\pi(1/p - 1/p^*) 
< 2\pi/p.
\]
Proposition \ref{theoremetroyanov} shows that
\[
\iint_{D(1/2)} e^{2p V[\omega_H^+](z)} d\lambda(z) \leq \frac{\pi}{1-\frac{p}{2\pi} \omega_H^+(D(1/2))},
\]
and with $1-\frac{p}{2\pi} \omega_H^+(D(1/2)) \geq \frac{p}{p^*} \geq \frac{1}{p^*} = \delta/4\pi$ we obtain
\begin{equation} \label{equationapplicationtheoremetroyanov}
\iint_{D(1/2)} e^{2p V[\omega_H^+](z)} d\lambda(z) \leq \frac{4\pi^2}{\delta}.
\end{equation}
By proposition \ref{propositioninclusionboulerayonmoitie} we have $H(B(x,\varepsilon/2)) \subset D(r)$, hence
\[
\iint_{D(r)} e^{2u(z)} d\lambda(z) \geq \area(B(x,\varepsilon/2)).
\]
Corollary \ref{corollairevolumeboulesalexandrov} shows that
\[
\area(B(x,\varepsilon/2)) \geq \big(2\pi - \omega^+(B(x,\varepsilon/2))\big) \cdot (\varepsilon/2)^2/32 \geq \delta \varepsilon^2/128.
\]
Since $u = V[\omega_H] + h$, by Hölder's inequality we get:
\begin{eqnarray*}
\delta \varepsilon^2/128 \leq \area(B(x_,\varepsilon/2)) & \leq & \iint_{D(r)} e^{2u(z)} d\lambda(z) \\
 & \leq & \big( \iint_{D(r)} e^{2pV[\omega_H](z)} d\lambda(z) \big)^{1/p} \big( \iint_{D(r)} e^{2p^* h(z)} d\lambda(z) \big)^{1/{p^*}} \\
 & \leq & \big( \iint_{D(1/2)} e^{2pV[\omega_H](z)} d\lambda(z) \big)^{1/p} \big( \iint_{D(r)} e^{2p^* h(z)} d\lambda(z) \big)^{1/{p^*}},
\end{eqnarray*}
With the inequality $V[\omega_H](z) = V[\omega_H^+](z) - V[\omega_H^-](z) \leq V[\omega_H^+](z)$ valid for almost every $z\in D(1/2)$, and with equation (\ref{equationapplicationtheoremetroyanov}) we get
\begin{eqnarray*}
\delta \varepsilon^2/128 & \leq & \left(\frac{4\pi^2}{\delta}\right)^{1/p} \cdot \big( \iint_{D(r)} e^{2p^* h(z)} d\lambda(z) \big)^{1/{p^*}} \\
 & \leq & \frac{4\pi^2}{\delta} \cdot \big( \iint_{D(r)} e^{2p^* h(z)} d\lambda(z) \big)^{1/{p^*}}
\end{eqnarray*}
(recall that $\frac{4\pi^2}{\delta}>1$) and this ends the proof of the lemma.
\end{proof}

To prove theorem \ref{theoremetermeharmoniqueborne}, we finally need the following Harnack's lemma for non-negative harmonic functions (see \cite{ABR}):
\begin{theorem}[Harnack's lemma] \label{theoremeharnack}
Let $f_m$ be a sequence of non-negative harmonic functions on a connected open set $U \subset \C$. We then have the following alternative: either (1) $f_m \rightarrow +\infty$ locally uniformly on $U$, or (2) there exists a subsequence $m_j$ of $m$ such that $f_{m_j} \rightarrow f$ locally uniformly on $U$, where $f$ is a harmonic function on $U$.
\end{theorem}

We can now finish the proof of theorem \ref{theoremetermeharmoniqueborne}.

\begin{proof}[Proof of theorem \ref{theoremetermeharmoniqueborne}]
Proposition \ref{theoremeinitialtermeharmoniquemajore} ensures that for every compact set $K \subset D(1/2)$, there exists a constant $M'(K)$ such that under the hypothesis of theorem \ref{theoremetermeharmoniqueborne} we have $h(z) \leq M'(K)$ for every $z \in K$ (recall that $M'(K)$ is independent of the metric $d \in \M_\Sigma (A,c,\varepsilon,\delta)$).

We prove theorem \ref{theoremetermeharmoniqueborne} by contradiction. Suppose there exist a compact set $K \subset D(1/2)$ (we may assume $D(r) \subset K$), a sequence $d_m \in \M_\Sigma (A,c,\varepsilon,\delta)$, a sequence of points $x_m \in \Sigma$, and a sequence of conformal charts $H_m : B_m(x_m,\varepsilon) \rightarrow D(1/2)$, with $H_m(x_m)=0$, such that the harmonic term $h_m$ for the metric in this chart verifies
\[
\min_{z \in K} h_m(z) \underset{m \rightarrow \infty} \rightarrow - \infty.
\]
Choose some $z_m \in K$ such that $\min_{z \in K} h_m(z)=h_m(z_m)$. After passing to a subsequence, we may assume $z_m \rightarrow z \in K$. Let $U$ be a connected open set, with $K \subset U \subset \overline{U} \subset D(1/2)$.

Since $h_m(z) \leq M'(\overline{U})$ for every $z \in U$, we can consider the following sequence of non-negative harmonic functions on $U$:
\[
f_m := M'(\overline{U}) - h_m.
\]
Since $f_m(z_m) \rightarrow +\infty$, alternative (2) in theorem \ref{theoremeharnack} can not occur. So we have $f_m \rightarrow +\infty$ locally uniformly on $U$, hence $h_m \rightarrow -\infty$ locally uniformly on $U$. But lemma \ref{lemmeminorationvolume} tells us that
\[
\left(\frac{\delta^2 \varepsilon^2}{512\pi^2}\right)^{p^*} \leq \iint_{D(r)} e^{2p^* h_m(z)} d\lambda(z).
\]
This is a contradiction, since the right-hand side term goes to zero as $m \rightarrow \infty$, by Lebesgue's dominated convergence theorem: we have $e^{2p^* h_m(z)} \rightarrow 0$ for every $z\in D(r) \subset U$, and since $D(r) \subset \overline{U}$, we can dominate $e^{2p^* h_m(z)}$ by the integrable function $e^{2p^* M'(\overline{U})}$.
\end{proof}

\subsection{Conformal images of balls} \label{sectionconstantstocontrol}

This section is devoted to the proof of the next theorem. It is the key step in our article, and also relies on the conformal geometry of an annulus. Roughly, it says the following. Let $H : B(x,\varepsilon) \rightarrow D(1/2)$ be a conformal chart for some metric $d \in \M_\Sigma(A,c,\varepsilon,\delta)$, with $H(x)=0$. Then we have a control on the images of balls of "large" radii $B(x,\varepsilon/4)$ (that is, we have $D(2\alpha) \subset H(B(x,\varepsilon/4))$), and balls of "small" radii $B(x,\kappa\varepsilon)$ (that is, we have $H(B(x,\kappa\varepsilon)) \subset D(\alpha)$), for some positive constants $\alpha$ and $\kappa$ (the picture in the theorem explains the situation). Of course, for any metric $d \in \M_\Sigma(A,c,\varepsilon,\delta)$, such constants $\alpha$ and $\kappa$ do exist, but the hard part of the work is to show that they can be chosen \emph{uniformly}: they do not depend on the metric $d \in \M_\Sigma(A,c,\varepsilon,\delta)$.

\begin{theorem} \label{theoremerecouvrement}
There exists constants $\alpha = \alpha(\Sigma,A,c,\varepsilon,\delta)>0$ and $\kappa = \kappa(\Sigma,A,c,\varepsilon,\delta)>0$ verifying the following property. Let $d \in \M_\Sigma(A,c,\varepsilon,\delta)$, and let $H : B(x,\varepsilon) \rightarrow D(1/2)$ be a conformal chart, with $H(x)=0$. We are in the following situation:

\begin{center}
\begin{tikzpicture}[scale=3]
\draw [dashed] (0,0) circle (0.25) ;
\draw [dashed] (0,0) circle (0.5) ;
\draw [dashed] (0,0) circle (0.9) ;
\draw [double] (0,0) circle (1) ;
\draw [domain = 0:90] [samples = 100] plot (\x : { 0.5*(1/4+1/16*(cos(6.*\x))}) ;
\draw [domain = 90:180] [samples = 100] plot (\x : { 0.5*(1/4+1/16*(cos(6.*\x))}) ;
\draw [domain = 180:270] [samples = 100] plot (\x : { 0.5*(1/4+1/16*(cos(4.*\x))}) ;
\draw [domain = 270:360] [samples = 100] plot (\x : { 0.5*(1/4+1/16*(cos(4.*\x))}) ;
\draw [domain = 0:90] [samples = 100] plot (\x : {0.8+1/20*(cos(4.*\x))}) ;
\draw [domain = 90:180] [samples = 100] plot (\x : {0.8+1/20*(cos(8.*\x))}) ;
\draw [domain = 180:270] [samples = 100] plot (\x : {0.8+1/20*(cos(8.*\x))}) ;
\draw [domain = 270:360] [samples = 100] plot (\x : {0.8+1/20*(cos(4.*\x))}) ;
\draw [domain = 0:90] [samples = 100] plot (\x : {0.6+1/20*(cos(8.*\x))}) ;
\draw [domain = 90:180] [samples = 100] plot (\x : {0.6+1/20*(cos(8.*\x))}) ;
\draw [domain = 180:270] [samples = 100] plot (\x : {0.6+1/20*(cos(4.*\x))}) ;
\draw [domain = 270:360] [samples = 100] plot (\x : {0.6+1/20*(cos(16.*\x))}) ;
\draw (0,1) -- (0,1.1) ;
\draw (0,1.1) node [above] {$D(1/2)=H(B(x,\varepsilon))$};
\draw ({0.9*cos(157)},{0.9*sin(157)}) -- (-1.5,0.5) ;
\draw (-1.5,{0.8*sin(157)+0.2}) node [left] {$D(r)$} ;
\draw (-0.5,0) -- (-1.5,0) ;
\draw (-1.5,0) node [left] {$D(2\alpha)$} ;
\draw ({0.25*cos(203)},{0.25*sin(203)}) -- (-1.5,-0.5) ;
\draw (-1.5,-0.5) node [left] {$D(\alpha)$};
\draw ({1/sqrt(2)*(0.8+1/20*(cos(4*45))},{1/sqrt(2)*(0.8+1/20*(cos(4*45))}) -- (1.5,0.6) ;
\draw (1.5,0.6) node [right] {$H(B(x,\varepsilon/2))$} ;
\draw ({cos(15)*(0.6+1/20*(cos(8*15))},{sin(15)*(0.6+1/20*(cos(8*15))}) -- (1.5,0.2) ;
\draw (1.5,0.2) node [right] {$H(B(x,\varepsilon/4))$};
\draw (1.5,-0.2) -- ({cos(0)*0.5*(1/4+1/16*cos(4*0))},{sin(0)*0.5*(1/4+1/16*cos(4*0))}) ;
\draw (1.5,-0.2) node [right] {$H(B(x,\kappa\varepsilon))$};
\draw (0,0) -- (1.5,-0.6) ;
\draw (1.5,-0.6) node [right] {$0 = H(x)$} ;
\draw (0,0) node {$\bullet$} ;
\end{tikzpicture}
\captionof{figure}{}
\end{center}

(The inclusion $H(B(x,\varepsilon/2)) \subset D(r)$ has already been proved in proposition \ref{propositioninclusionboulerayonmoitie}).

Thus, this is sufficient to prove the following two properties:
\begin{enumerate}
\item If $\gamma$ is the line segment joigning 0 and a point in $D(2\alpha)$, then the length of the curve $H^{-1}(\gamma)$ in $\Sigma$ is smaller than $\varepsilon/4$. This proves that $D(2\alpha) \subset H(B(x,\varepsilon/4))$;
\item we have the inclusion $H(B(x,\kappa\varepsilon)) \subset D(\alpha)$.
\end{enumerate}
\end{theorem}

\begin{remark}
In practice, $\alpha$ may be very small; for the need of the picture, we have chosen $\alpha = 0.125$.
\end{remark}

We first choose $\alpha$ small enough so that property 1. is true. The idea is the following: by theorem \ref{theoremeinitialtermeharmoniquemajore}, we have an (explicit) upper bound for the harmonic term $h$, and proposition \ref{propositionmajorationdistance} gives an upper bound for the length of a line segment, when there is no harmonic term in the expression of the singular metric.

We then prove a convergence theorem for distances (a corollary of the local convergence theorem due to Y. Reshetnyak, theorem \ref{theoremeconvergencedistancereshetnjack}). We need it to prove part 2. of theorem \ref{theoremerecouvrement}, but we will also need it later (see section \ref{sectionlocalproperties}).

With this convergence theorem, we are able to choose $\kappa$ small enough so that the annulus $D(1/2)-H(\overline{B}(x,\kappa\varepsilon))$ has a modulus big enough. Hence by Grötzsch's theorem, $H(\overline{B}(x,\kappa\varepsilon)$ will be far away from the boundary $\partial D(1/2)$, that is we will have $H(B(x,\kappa\varepsilon)) \subset D(\alpha)$.

\subsubsection{Choice of $\alpha$}

Recall that $C(\Omega) = \sqrt{1+\Omega/2\pi} \cdot e^{\Omega/4\pi}$ is the constant which appears in proposition \ref{theoremeinitialtermeharmoniquemajore}. Choose $0<\alpha<1/8$ such that
\begin{equation}\label{equationdefinitionalpha}
\frac{{(2\alpha)}^{\delta/2\pi} \cdot 4\pi}{\delta} \cdot \frac{C(\Omega)}{(1/4)^{1+\Omega/2\pi}} \leq 1/4.
\end{equation}
\begin{proof}[Proof of the first part of theorem \ref{theoremerecouvrement}]
Let $|z|<2\alpha$, and $\gamma(t) := tz$ the line segment $[0,z]$. We want to find an upper bound for the length of $H^{-1}(\gamma)$:
\[
L(H^{-1}(\gamma)) = \int_0^1 e^{V[\omega_H](\gamma(t)) + h(\gamma(t))} |z| dt.
\]
We have the following
\begin{fact} \label{factinegalitealpha}
For every $z' \in D(2\alpha)$ we have
\[
\frac{{(2\alpha)}^{\delta/2\pi} \cdot 4\pi}{\delta} \cdot e^{h(z')} \leq \varepsilon/4.
\]
\end{fact}
\begin{proof}
Proposition \ref{theoremeinitialtermeharmoniquemajore} shows that, for every $z' \in D(1/2)$,
\[
e^{h(z')} \leq \frac{\varepsilon}{(1/2-|z'|)^{1+\Omega/2\pi}} \cdot C(\Omega),
\]
so by multiplying with the inequality (\ref{equationdefinitionalpha}) we get
\[
\frac{{(2\alpha)}^{\delta/2\pi} \cdot 4\pi}{\delta} \cdot \frac{C(\Omega)}{(1/4)^{1+\Omega/2\pi}}  \cdot e^{h(z')} \leq \varepsilon/4 \cdot \frac{C(\Omega)}{(1/2-|z'|)^{1+\Omega/2\pi}}.
\]
For every $z' \in D(2\alpha)$ we have $|z'| \leq 1/4$, hence $(1/2-|z'|)^{1+\Omega/2\pi} \geq (1/4)^{1+\Omega/2\pi}$. After simplification we obtain the inequality announced in fact \ref{factinegalitealpha}.
\end{proof}

The line segment $[0,z]$ is included in $D(2\alpha)$ so by fact \ref{factinegalitealpha} we have
\[
e^{h(\gamma(t))} \leq \frac{\delta}{{(2\alpha)}^{\delta/2\pi} \cdot 4\pi} \cdot \varepsilon/4,
\]
hence
\begin{equation} \label{inegalitefaitchoixkappaalpha}
L (H^{-1}(\gamma)) \leq \frac{\delta}{{(2\alpha)}^{\delta/2\pi} \cdot 4\pi} \cdot \varepsilon/4 \cdot \int_0^1 e^{V[\omega_H](\gamma(t))} |z| dt.
\end{equation}
Moreover, proposition \ref{propositionmajorationdistance} shows that
\begin{eqnarray*}
\int_0^1 e^{V[\omega_H](\gamma(t))} |z| dt & \leq & \frac{2}{1-\omega_H^+(D(1/2))/2\pi} \cdot |z|^{1-\omega_H^+(D(1/2))/2\pi} \\
 & = & \frac{2}{1-\omega^+(B(x,\varepsilon))/2\pi} \cdot |z|^{1-\omega^+(B(x,\varepsilon))/2\pi}.
\end{eqnarray*}
With the inequality $\omega^+(B(x,\varepsilon)) \leq 2\pi-\delta$ we obtain
\[
\int_0^1 e^{V[\omega_H](\gamma(t))} |z| dt < \frac{4\pi}{\delta} \cdot (2\alpha)^{\delta/2\pi},
\]
and with (\ref{inegalitefaitchoixkappaalpha}) we finally obtain $L (H^{-1}(\gamma)) < \varepsilon/4$. This ends the proof of the first part of theorem \ref{theoremerecouvrement}.
\end{proof}

\subsubsection{Convergence of metrics: a corollary of theorem \ref{theoremeconvergencedistancereshetnjack}}

We now prove a corollary of theorem \ref{theoremeconvergencedistancereshetnjack}, which is needed to finish the proof of theorem \ref{theoremetermeharmoniqueborne}. This result will also be a key step at the end of this article (see section \ref{sectionlocalproperties}).

\begin{corollary} \label{corollaireconvergencedistance}
 Let $d_m \in \M_\Sigma (A,c,\varepsilon,\delta)$ be a sequence of metrics, and for every $x_m \in \Sigma$, consider some conformal chart $H_m : B_m(x_m,\varepsilon) \rightarrow D(1/2)$, with $H_m(x_m)=0$. Let $\omega_m$ be the curvature measure of $(\Sigma,d_m)$, and let $\omega_{H_m} := (H_m)_\# \omega_m$ be the measure and $h_m$ be the harmonic map such that $H_m$ is an isometry between $(B_m(x_m,\varepsilon),{d_m}_{|B_m(x_m,\varepsilon)})$ and $(D(1/2),d_{\omega_{H_m},h_m})$. Then after passing to a subsequence, the following is true.

There is a constant $C>0$ and a measure $\widetilde{\omega}$, with support in $\overline{D}(1/2)$, such that
\[
 d_{\omega_{H_m},h_m} \mbox{ converges to } C \cdot \overline{d}_{\widetilde{\omega},0}, \mbox{ locally uniformly on $D(2\alpha)$}
\]
(that is, if $z_m \rightarrow z \in D(2\alpha)$ and $z'_m \rightarrow z' \in D(2\alpha)$, then $d_{\omega_{H_m},h_m}(z_m,z'_m) \rightarrow C \cdot \overline{d}_{\widetilde{\omega},0}(z,z')$).
\end{corollary}

For the proof, we need to apply theorem \ref{theoremeconvergencedistancereshetnjack}, which is a convergence theorem for distances, when there is no harmonic term in the metric. Hence we need to get rid of $h_m$: to do so, we express $h_m$ as the potential of some measure, with support on a circle.

\begin{proof}
 We know that we can express an harmonic map in terms of its normal derivatives along a circle: for $z \in D(r)$ we have
\[
 h_m(z) = h_m(0) - \frac{1}{\pi} \int_{\partial D(r)} \ln|z-\xi| \cdot \frac{\partial h_m}{\partial \nu} (\xi) |d\xi|,
\]
where $\frac{\partial h_m}{\partial \nu}$ is the radial derivative of $h_m$. Hence for $z \in D(r)$, we can write $h_m$ as
\begin{equation} \label{equation1corollaireconvergencedistance}
h_m(z) = h_m(0) + V[\mu_m](z),
\end{equation}
where $\mu_m$ is the following measure with support in $\partial D(r)$ : $\mu_m := \frac{1}{2} \frac{\partial h_m}{\partial \nu} |d\xi|$. Let 
\[
\widetilde{\omega}_{m} := \omega_{H_m} + \mu_m.
\]
We have the following fact, which needs some justifications, since representation (\ref{equation1corollaireconvergencedistance}) is only valid for $z \in D(r)$:
\begin{fact} \label{factcorollaireconvergencedistance}
For $u,u' \in D(2\alpha)$ we have
\[
 d_{\omega_{H_m},h_m}(u,u') = e^{h_m(0)} \cdot \overline{d}_{\widetilde{\omega}_{m},0}(u,u').
\]
\end{fact}

\begin{proof}
By definition,
\begin{equation} \label{equation2corollaireconvergencedistance}
 d_{\omega_{H_m},h_m}(u,u') = \inf_\gamma \int_0^l e^{V[\omega_{H_m}](\gamma(t)) + h_m(\gamma(t))} dt, 
\end{equation}
where the infimum is taken over all simple continuous curves $\gamma : [0,l] \rightarrow D(1/2)$, parametrized by arc length, with $\gamma(0)=u$ and $\gamma(l)=u'$. Let $y_m = (H_m)^{-1}(u)$ and $y'_m = (H_m)^{-1}(u')$. Since property 1. of theorem \ref{theoremerecouvrement} has already been proved, we have $(H_m)^{-1} (D(2\alpha)) \subset B_m(x_m,\varepsilon/4)$, hence $d_m(x_m,y_m) < \varepsilon/4$ and $d_m(x_m,y'_m) < \varepsilon/4$.

Now, assume that $\gamma : [0,l] \rightarrow D(1/2)$ is a continuous simple curve between $u$ and $u'$, parametrized by arc length, which is not included in $D(r)$. Let $\widetilde{\gamma} := (H_m)^{-1}(\gamma)$: this is a curve between $y_m$ and $y'_m$, and since $H_m(B_m(x_m,\varepsilon/2)) \subset D(r)$ (this is property \ref{propositioninclusionboulerayonmoitie}), $\widetilde{\gamma}$ is not included in $B_m(x_m,\varepsilon/2)$. But $\widetilde{\gamma}$ is a curve joigning two points in $B_m(x_m,\varepsilon/4)$, and has to leave $B_m(x_m,\varepsilon/2)$, so its length is greater than or equal to $2 \cdot \varepsilon/4 = \varepsilon/2$:
\begin{center}
\begin{tikzpicture}[scale=1.1]
\draw (0,0) node {$\bullet$};
\draw (0,0) node [left] {$x_m$};
\draw (-1.2,0.5) node {$B_m(x_m,\varepsilon/4)$};
\draw (-3.5,0) node {$B_m(x_m,\varepsilon/2)$};
\draw (0,0) ellipse (2.5cm and 1.5cm);
\draw (0,0) ellipse (4.5cm and 2cm);
\draw [domain = 0:60] [dashed] [samples = 100] plot (\x : {1+3*sin(3.*\x)}) ;
\draw (1,0) node [below] {$y_m$};
\draw ({cos(60)*(1+3*sin(3*60))},{sin(60)*(1+3*sin(3*60))}) node [below] {$y'_m$};
\draw ({cos(20)*(1+3*sin(3*20))-0.2},{sin(20)*(1+3*sin(3*20))-0.2}) node [left] {$\widetilde{\gamma}$};
\end{tikzpicture}
\captionof{figure}{}
\end{center}
Since the distance between $y_m$ and $y'_m$ is less than $\varepsilon/2$, this shows that in the formula (\ref{equation2corollaireconvergencedistance}) we can only consider curves $\gamma$ included in $D(r)$. For such curves $\gamma$ we can use the representation (\ref{equation1corollaireconvergencedistance}), so we get
\[
  d_{\omega_{H_m},h_m}(u,u') = e^{h_m(0)} \cdot \inf_\gamma \int_0^l e^{V[\widetilde{\omega}_{m}](\gamma(t))} dt, 
\]
where the infimum is taken over all continuous simple curves $\gamma : [0,l] \rightarrow D(r)$, parametrized by arc length, with $\gamma(0)=u$ and $\gamma(l)=u'$. For the same reason as before, this is equal to the infimum of the same quantity, over all the curves $\gamma : [0,l] \rightarrow \overline{D}(1/2)$, and this is exactly $e^{h_m(0)} \cdot \overline{d}_{\widetilde{\omega}_{m},0}(u,u')$ (see the equation (\ref{equationdistancetechnique}) in section \ref{sectionconformalcharts} for the definition of $\overline{d}_{\widetilde{\omega}_{m},0}$). This ends the proof of fact \ref{factcorollaireconvergencedistance}.
\end{proof}

By theorem \ref{theoremetermeharmoniqueborne}, the sequence $(h_m(0))_{m\in \N}$ is bounded, so after passing to a subsequence, we may assume that $e^{h_m(0)} \rightarrow C>0$. Moreover, for (bounded) harmonic maps, Cauchy's formula gives a bound for the derivatives (at some point $x$) of the map, in terms of a bound for the modulus of the map (on some ball centered in $x$). Since the harmonic maps $h_m$ are bounded on every compact subset of $D(1/2)$, this shows that $\frac{\partial h_m}{\partial \nu}$ is bounded on $\partial D(r)$ by a quantity which does not depend on $m$, hence $\mu_m^+(D(1/2))$ and $\mu_m^-(D(1/2))$ are bounded, so after passing to a subsequence we may assume that $\mu_m^+ \rightarrow \mu^+$ and $\mu_m^- \rightarrow \mu^-$ weakly. Since the supports of $\mu_m^+$ and $\mu_m^-$ are included in $\partial D(r)$, the supports of $\mu^+$ and $\mu^-$ are also included in $\partial D(r)$. Let $\mu:=\mu^+ - \mu^-$.

We may also assume that $\omega_{H_m}^+ \rightarrow \omega^+$, and $\omega_{H_m}^- \rightarrow \omega^-$ weakly. We set $\omega:=\omega^+ - \omega^-$. Since $\widetilde{\omega}_m=\omega_{H_m} + \mu_m$, we have $\widetilde{\omega}_m \rightarrow \widetilde{\omega}:= \omega+\mu$ weakly.

Since $\omega_{H_m}^+(D(1/2)) = \omega_m^+(B_m(x_m,\varepsilon))\leq 2\pi - \delta$, we have $\omega^+(\{z\})<2\pi$ for every $z \in D(2\alpha)$, and $\mu^+$ has its support in $\partial D(r)$, hence $\mu^+(\{z\})=0$ for every $z \in D(2\alpha)$. We have obtained $\widetilde{\omega}^+(\{z\})<2\pi$ for every $z \in D(2\alpha)$, we can then apply theorem \ref{theoremeconvergencedistancereshetnjack}: if $z_m \rightarrow z \in D(2\alpha)$ and $z'_m \rightarrow z' \in D(2\alpha)$, then
\[
 \overline{d}_{\widetilde{\omega}_{m},0}(z_m,z'_m) \rightarrow \overline{d}_{\widetilde{\omega},0}(z,z'),
\]
and fact \ref{factcorollaireconvergencedistance} gives
\[
 d_{\omega_{H_m},h_m}(z_m,z'_m) = e^{h_m(0)} \cdot \overline{d}_{\widetilde{\omega}_{m},0}(z_m,z'_m) \rightarrow C \cdot \overline{d}_{\widetilde{\omega},0}(z,z').
\]
\end{proof}

\subsubsection{Choice of $\kappa$}

We first choose $\kappa$ small enough so that the annulus $D(1/2)-H(\overline{B}(x,\kappa\varepsilon))$ has a modulus big enough (this is lemma \ref{lemmechoixkappa}); we can then use Grötzsch theorem to finish the proof of theorem \ref{theoremerecouvrement} (see below). Recall that $\module (G(2\alpha))$ is the modulus of the Grötzsch annulus $G(2\alpha)$.
\begin{lemma} \label{lemmechoixkappa}
There exists a constant $\kappa= \kappa(\Sigma,A,c,\varepsilon,\delta)>0$ verifying the following property. Under the hypothesis of theorem \ref{theoremerecouvrement}, the topological annulus (in the sense given in the appendix, see definition \ref{definitionannulus}) $D(1/2)-H(\overline{B}(x,\kappa\varepsilon))$ has a modulus greater than $\module (G(2\alpha))$.
\end{lemma}

\begin{remark}
This property is obvious in the Euclidean plane: an annulus $A(R_1,R_2)$ of boundary two concentric circles of radii $R_1<R_2$ has modulus $\frac{1}{2\pi} \ln(R_2/R_1)$, so this quantity goes to infinity when $R_1$ goes to zero.
\end{remark}

\begin{proof}
We prove the lemma by contradiction. Suppose there exists a sequence $d_m \in \M_\Sigma(A,c,\varepsilon,\delta)$, a sequence of points $x_m \in \Sigma$, a sequence of harmonic charts $H_m : B_m(x_m,\varepsilon) \rightarrow D(1/2)$, with $H_m(x_m)=0$, such that
\[
\module\big(D(1/2)-H_m(\overline{B}_m(x_m,\varepsilon/m))\big) \leq \module(G(2\alpha)).
\]
Let $\iota>0$ be such that $\frac{1}{2\pi}\ln(1/2\iota)>\module (G(2\alpha))$ (we may assume $\iota<2\alpha$). We have the following
\begin{fact} \label{factinclusionnoninclusion}
We have $H_m(B_m(x_m,\varepsilon/m)) \not\subset D(\iota)$.
\end{fact}
\begin{proof}
Suppose $H_m(B_m(x_m,\varepsilon/m)) \subset D(\iota)$. Then we have
\[
D(1/2) - \overline{D}(\iota) \subset D(1/2)-H_m(\overline{B}_m(x,\varepsilon/m)),
\]
hence
\[
\module\big(D(1/2)-H_m(\overline{B}_m(x,\varepsilon/m))\big) \geq \module\big(D(1/2) - \overline{D}(\iota)\big) = \frac{1}{2\pi} \ln(1/2\iota) > \module (G(2\alpha)),
\]
and this is a contradiction.
\end{proof}

Consider the singular Riemannian metric $g_m = e^{2V[\omega_m] + 2h_m}|dz|^2$, such that $H_m$ is an isometry between $(B_m(x_m,\varepsilon), {d_m}_{|B_m(x_m,\varepsilon)})$ and $(D(1/2),d_{\omega_m,h_m})$.

By fact \ref{factinclusionnoninclusion}, there exists complex numbers $z_m$, with $|z_m|\geq \iota$, and $d_{\omega_m,h_m}(0,z_m) \leq \varepsilon/m$. By considering the intersection point between a geodesic from $0$ to $z_m$ and $\partial D(\iota)$, we may even assume that $|z_m|=\iota$ and $d_{\omega_m,h_m}(0,z_m) \leq \varepsilon/m$. By compactness, after passing to a subsequence we may assume $z_m \rightarrow z \neq 0$. Since $|z|=\iota<2\alpha$, by corollary \ref{corollaireconvergencedistance}, we know that there exists a constant $C>0$ and a Radon measure $\widetilde{\omega}$, with support in $\overline{D}(1/2)$, such that, after passing to a subsequence, we have
\[
 d_{\omega_m,h_m}(0,z_m) \underset{m \rightarrow \infty} \longrightarrow C \cdot \overline{d}_{\widetilde{\omega},0}(0,z) \neq 0,
\]
and this is absurd since $d_{\omega_m,h_m}(0,z_m) \leq \varepsilon/m \rightarrow 0$.
\end{proof}

\begin{proof}[Proof of the second part of theorem \ref{theoremerecouvrement}]
We will prove the inclusion
\[
H(\overline{B}(x,\kappa\varepsilon)) \subset D(\alpha)
\]
by contradiction : suppose we have
\[
H(\overline{B}(x,\kappa\varepsilon)) \not\subset D(\alpha).
\]
After a rotation, we may assume that $H(\overline{B}(x,\kappa \varepsilon))$ does not contain the point $\alpha \in \R \subset \C$.
Then, after an homothety of scale factor 2, we see that the annulus $D(1/2)-H(\overline{B}(x,\kappa\varepsilon))$ is conformal to an annulus $U \subset D(1)$, not containing 0 and $2\alpha$. By Grötzsch's theorem (see theorem \ref{theoremegrotzsch} in the appendix), we know that the modulus of this annulus is
\[
\module(U) = \module\big(D(1/2)-H(\overline{B}(x,\kappa\varepsilon))\big) \leq \module (G(2\alpha)),
\]
and this is a contradiction by lemma \ref{lemmechoixkappa}. This ends the proof of theorem \ref{theoremerecouvrement}.
\end{proof}

\section{Proof of the Main theorem}

We can now start the proof of the Main theorem. From now on, we consider a sequence of metrics $d_m \in \M_\Sigma(A,c,\varepsilon,\delta)$, that is
\begin{enumerate}
\item for every $x\in \Sigma$ we have $\omega_m^+ (B_m(x,\varepsilon)) \leq 2\pi -\delta$;
\item $\cont({\Sigma, d_m}) \geq c$;
\item $\area(\Sigma,d_m) \leq A$.
\end{enumerate}
Recall that we always assume $\varepsilon<c$. By proposition \ref{propositionequivalencedefinition}, we also have some constants $D>0$ and $\Omega>0$ such that
\[
\diam(\Sigma,d_m) \leq D \mbox{ and } |\omega|(\Sigma,d_m) \leq \Omega,
\]
and we have some constants $\alpha>0$ and $\kappa>0$ such that theorem \ref{theoremerecouvrement} is true.

We will often consider subsequences of the original sequence $(d_m)$; we will never change the name of the sequence, and we will assume that the sequence has the desired properties from the beginning.

\subsection*{Sketch of the proof}
\addcontentsline{toc}{subsection}{Sketch of the proof}
The proof is an adaptation of the proof of Cheeger-Gromov's compactness theorem, presented in \cite{HH}. We give an outline of the proof here: to understand it in its globality, we have simplified many of the arguments. See the proof below for precise statements.

\begin{enumerate}

\item We cover $\Sigma$ by open sets $B_m(x_i^m,\varepsilon)$, for $i \in \{1,...,N\}$ (by volume arguments, the number $N$ is independent of $m$). We can then define conformal charts $H_i^m : B_m(x_i^m,\varepsilon) \rightarrow D$ with $H_i^m(x_i^m)=0$.

\item We extend the charts $H_i^m$ to the whole surface $\Sigma$ by defining maps $\overline{H_i^m} : \Sigma \rightarrow D$, which are equal to $H_i^m$ near $x_i^m$.  We then embed $\Sigma$ into an Euclidean space $\R^q$:
\[
 \Psi^m(x) :\simeq (\overline{H_1^m}(x), ... , \overline{H_N^m}(x)) \in \R^q.
\]
The embedded surface $\Sigma^m$ is locally a graph over some subset $D' \subset D$: for example with $i=1$, we have, for $z \in D'$,
\begin{equation}   \label{equationsketchoftheproof}
 \Psi^m((H_1^m)^{-1}(z)) \simeq (z, \Theta_1^m(z)) 
\end{equation}
where $\Theta_1^m(z)= (\overline{H_2^m} \circ (H_1^m)^{-1} (z),..., \overline{H_N^m} \circ (H_1^m)^{-1} (z))$.

\item The maps $\overline{H_j^m} \circ (H_i^m)^{-1}$ are either zero, or looks like $H_j^m \circ (H_i^m)^{-1}$ when this last expression makes sense. Moreover, since the charts $H_i^m$ are conformal, the transition maps $H_j^m \circ (H_i^m)^{-1}$ are conformal maps between open sets of $\C$: since they are bounded, by Montel's theorem they will converge uniformly (up to a subsequence). We can pass to the limit in the representation of $\Sigma^m$ as a union of graphs (see the equation (\ref{equationsketchoftheproof})) , and define a subset $\Sigma^\infty \subset \R^q$. We prove that $\Sigma^\infty$ is an embedded surface.

\item For $m$ large enough, the embedded surfaces $\Sigma^m$ are in a tubular neighborhood of $\Sigma^\infty$. We can then project $\Sigma^m$ along the normals onto $\Sigma^\infty$ and define a map $\Pi^m : \Sigma^m \rightarrow \Sigma^\infty$. Since $\Sigma^m$ converge to $\Sigma^\infty$ (in the sense given above), we prove that $\Pi^m$ is actually a diffeomorphism.

\item We have diffeomorphisms $\Sigma \overset{\sim} \longrightarrow \Sigma^m \subset \R^q \overset{\sim} \longrightarrow \Sigma^\infty \subset \R^q$. We transport the initial metric $d_m$ to a metric $\widetilde{d_m}$ on $\Sigma^\infty$, so that $(\Sigma,d_m)$ is isometric to $(\Sigma^\infty,\widetilde{d_m})$. We finally show that the metric $\widetilde{d_m}$ converge, by using the local convergence theorem due to Reshetnyak (theorem \ref{theoremeconvergencedistancereshetnjack}).
\end{enumerate}

\subsection{Covering of $\Sigma$ and notations} \label{partieconstructioncartes}

Let $m \in \N$. Consider a maximal number $N(m)$ of disjoint balls $B_m(x_i^m, \kappa\varepsilon/4)$ in $\Sigma$. By corollary \ref{corollairevolumeboulesalexandrov}, we know that the area of these balls are bounded from below. Since the area of $(\Sigma,d_m)$ is bounded from above, this shows that the integer $N(m)$ is bounded: after passing to a subsequence, we can assume that $N(m)=N$ is constant. Moreover, by an elementary doubling property, we have
\begin{equation} \label{equationrecouvrementsigma}
\Sigma = \bigcup_{i=1}^N B_m(x_i^m, \kappa\varepsilon/2).
\end{equation}
By proposition \ref{propositioncontboulesouvertes} and theorem-definition \ref{theoremedefinitionconformalcharts} we can consider conformal charts
\[
H_i^m : B_m(x_i^m,\varepsilon) \rightarrow D(1/2),
\]
with $H_i^m(x_i^m)=0$. In the sequel, we set $\omega_i^m := (H_i^m)_\# \omega$, and $h_i^m$ is the harmonic function on $D(1/2)$ such that the singular Riemannian metric writes
\[
g_i^m := e^{2V[\omega_i^m](z) + 2h_i^m(z)} |dz|^2.
\]
Property 2. in theorem \ref{theoremerecouvrement} shows that we have
\[
H_i^m(B_m(x_i^m,\kappa\varepsilon)) \subset D(\alpha),
\]
and with the relation (\ref{equationrecouvrementsigma}) we obtain the following
\begin{fact} \label{propositionrecouvrementsigma}
We have the covering
\[
\Sigma = \bigcup_{i=1}^N (H_i^m)^{-1}(D(\alpha)).
\]
\end{fact}

\subsection{Embedding in an Euclidean space} \label{partieplongementespaceeuclidien}

We use a cut-off map $\varphi:\C \rightarrow [0,1]$ to extend the charts $H_i^m$ to the whole surface $\Sigma$; then, by an analog to Whitney's embedding theorem, we embed $\Sigma$ into an Euclidean space, and we show that this set is locally a graph.

Let $\varphi : \mathbb{C} \rightarrow [0,1]$ be a smooth function, with value 1 on $D(5\alpha/3)$ and 0 outside $D(2\alpha)$ (see the picture in theorem \ref{theoremerecouvrement}). For $1 \leq i \leq N$, we define on $\Sigma$ the following smooth maps:
\begin{itemize}
\item $\varphi_i^m := \varphi \circ H_i^m : \Sigma \rightarrow [0,1]$. By theorem \ref{theoremerecouvrement}, this map has value 1 on $B_m(x_i^m,\kappa\varepsilon)$, and 0 outside $B_m(x_i^m,\varepsilon/4)$;
\item $\overline{H_i^m} := \varphi_i^m H_i^m: \Sigma \rightarrow D(1/2)$. This map extends $H_i^m$. It is equal to $H_i^m$ on $B_m(x_i^m,\kappa\varepsilon)$, and is 0 outside $B_m(x_i^m,\varepsilon/4)$.
\end{itemize}

Let us now describe Withney's embedding. Let $q := 2N + N$. We define
\[
\begin{array}{llll}
\Psi^m : & \Sigma & \longrightarrow & \R^q  \\
 & x & \longmapsto & \Psi^m(x) = (\overline{H_1^m}(x),...,\overline{H_N^m}(x), \varphi_1^m(x), ... ,\varphi_N^m (x)).
\end{array}
\]
This is an easy verification to show that $\Psi^m$ is a smooth embedding, from $\Sigma$ into $\R^q$: the $2N$ first coordinates ensure the immersion property, and with the $N$ last coordinates we obtain injectivity. We denote by
\[
\Sigma^m := \Psi^m(\Sigma)
\]
the submanifold of $\R^q$ we have obtained.

Since $\varphi=1$ on $D(5\alpha/3)$, we remark that these submanifolds are locally graphs, parametrized by $D(5\alpha/3)$. Indeed, for $1\leq i \leq N$, the open sets $(H_i^m)^{-1}(D(5\alpha/3))$ cover $\Sigma$ (this is fact \ref{propositionrecouvrementsigma}), and for $z \in D(5\alpha/3)$ we have
\[
\varphi_i^m((H_i^m)^{-1}(z))=\varphi(z)=1 \mbox{ and } \overline{H_i^m}((H_i^m)^{-1}(z))=z,
\]
hence
\[
\Phi_i^m(z) := \Psi^m((H_i^m)^{-1}(z)) =
\]
\begin{equation} \label{eqcartesgraphees}
\bigg(\overline{H_1^m}((H_i^m)^{-1}(z)),...,z,...,\overline{H_N^m}((H_i^m)^{-1}(z)), \varphi_1^m((H_i^m)^{-1}(z)) ,...,1,...,\varphi_{N}^m((H_i^m)^{-1}(z))\bigg)
\end{equation}
(this is a graph since the $i$-th coordinate is $z$). $\Sigma^m$ is then the union of $N$ pieces of graphs:
\[
\Sigma^m = \bigcup_{i=1}^N \Phi_i^m(D(5\alpha/3)).
\]
If $x\in \Phi_i^m(D(5\alpha/3))$, we say that $x$ is in the graph number $i$.

\subsection{Convergence of the embedded surfaces $\Sigma^m$ to an embedded surface $\Sigma^{\infty}$}

We want to show the convergence of the maps defining $\Sigma^m$ as graphs, that is the convergence of the maps $\overline{H_j^m} \circ (H_i^m)^{-1}$ and $\varphi_j^m \circ (H_i^m)^{-1}$. We first show that, on some good open sets $V$, these maps are either zero on $V$, for every $m \in \N$, or $H_j^m \circ (H_i^m)^{-1}$ is well defined on $V$ (that is $(H_i^m)^{-1}(V) \subset B_m(x_j^m,\varepsilon)$), for every $m \in \N$. In the first case, the sequence of maps $\overline{H_j^m} \circ (H_i^m)^{-1}$ converges trivially; and in the second case, Montel's theorem allows us to conclude that this sequence of bounded conformal maps converges locally uniformly on $V$. By passing to the limit, we can define a subset $\Sigma^\infty \subset \R^q$. We prove that this set is an embedded surface.

\subsubsection{A preliminary study of the maps $\overline{H_j^m} \circ (H_i^m)^{-1}$ and $\varphi_j^m \circ (H_i^m)^{-1}$} \label{partieexpressionfonctionstransition}

Let $z \in D(1/2)$. If the expression $H_j^m \circ (H_i^m)^{-1}(z)$ makes sense (that is, if $(H_i^m)^{-1}(z) \in B_m(x_j^m,\varepsilon)$), then we have
\[
\left\lbrace
\begin{array}{rll}
\overline{H_j^m} \circ (H_i^m)^{-1}(z) &=& \varphi({H_j^m} \circ (H_i^m)^{-1}(z)) \cdot {H_j^m} \circ (H_i^m)^{-1}(z)\\
\varphi_j^m \circ (H_i^m)^{-1}(z)& =& \varphi({H_j^m} \circ (H_i^m)^{-1}(z)),
\end{array}
\right.
\]
otherwise (that is, if $(H_i^m)^{-1}(z) \notin B_m(x_j^m,\varepsilon)$) we have
\[
\overline{H_j^m} \circ (H_i^m)^{-1}(z)=0 \mbox{ and } \varphi_j^m \circ (H_i^m)^{-1}(z)=0.
\]
In some sense, we want to show that this dichotomy is valid \emph{uniformly in} $m \in \N$. After passing to a subsequence, the following proposition is true:

\begin{proposition} \label{propositionalternative}
There exists a finite covering with open sets
\[
 D(2\alpha) = \bigcup_{t \in T} V_t
\]
such that the following property is true (after passing to a subsequence). We fix $i,j \in \{1,...,N\}$ and $t \in T$. Then at least one of the following two properties is true:
\begin{itemize}
 \item (A) for every $m \in \N$, $V_t \subset H_i^m(B_m(x_i^m,\varepsilon) \cap B_m(x_j^m,\varepsilon))$.

Then $H_j^m \circ (H_i^m)^{-1}$ is well defined on $V_t$.
 \item (B) for every $m \in \N, \varphi_j^m \circ (H_i^m)^{-1}=0$ on $V_t$.

Then for every $m \in \N$, $\overline{H_j^m} \circ (H_i^m)^{-1} = 0$ $V_t$.
\end{itemize}
\end{proposition}

\begin{remark}
An open set $V_t$ may of course verify the two properties. Moreover, proposition (B) is verified when there are no transition maps between $B_m(x_i^m,\varepsilon)$ and $B_m(x_j^m,\varepsilon)$, that is when $B_m(x_i^m,\varepsilon) \cap B_m(x_j^m,\varepsilon) = \emptyset$.
\end{remark}

This proposition is a direct consequence of the following lemma. We set $\eta>0$ the constant verifying
\begin{equation} \label{equationeta}
e^{M(\overline{D}(2\alpha))} \cdot \frac{4\pi}{\delta} \cdot \eta^{\delta/2\pi} = 3\varepsilon/4
\end{equation}
(recall that $M(\overline{D}(2\alpha))$ is the constant which appears in theorem \ref{theoremetermeharmoniqueborne} for the compact set $K =\overline{D}(2\alpha)$).
\begin{lemma} \label{lemmealternative}
Let $m \in \N$, $i,j \in \{1,...,N\}$ and $z_0,z \in D(2\alpha)$. Suppose $\varphi_j^m \circ (H_i^m)^{-1} (z_0) \neq 0$. Then
\[
|z-z_0| \leq \eta \implies (H_i^m)^{-1}(z) \in B_m(x_j^m,\varepsilon) \mbox{ (hence } H_j^m \circ (H_i^m)^{-1}(z) \mbox{ exists)}.
\]
\end{lemma}
\begin{proof}
Since $\varphi_j^m \circ (H_i^m)^{-1} (z_0) = \varphi(H_j^m \circ (H_i^m)^{-1}(z_0)) \neq 0$, we have $H_j^m \circ (H_i^m)^{-1} (z_0) \in D(2\alpha)$, so by theorem \ref{theoremerecouvrement} we have $H_j^m \circ (H_i^m)^{-1}(z_0) \in H_j^m(B_m(x_j^m,\varepsilon/4))$, hence
\begin{equation} \label{inegalitetemporairedistance}
d_m(x_j^m,(H_i^m)^{-1}(z_0)) < \varepsilon/4.
\end{equation}
Now we will use the same arguments as in the proof of the first part of theorem \ref{theoremerecouvrement}. Let $\gamma(t) := (1-t)z_0 + tz$ be the line segment between $z_0$ and $z$. We have
\[
L_m((H_i^m)^{-1}(\gamma)) = \int_0^1 e^{V[\omega_i^m](\gamma(t)) + h_i^m(\gamma(t))} |z-z_0| dt.
\]
Since the line segment $\gamma$ is included in $D(2\alpha)$, we have $e^{h_i^m(\gamma(t))} \leq e^{M(\overline{D}(2\alpha))}$. And using proposition \ref{propositionmajorationdistance}, we have
\begin{eqnarray*}
\int_0^1 e^{V[\omega_i^m](\gamma(t))} |z-z_0| dt & \leq & \frac{2}{1-(\omega_i^m)^+(D(1/2))/2\pi} \cdot |z-z_0|^{1-(\omega_i^m)^+(D(1/2))/2\pi} \\
 & \leq & \frac{4\pi}{\delta} \cdot |z-z_0|^{\delta/2\pi}.
\end{eqnarray*}
Thus we obtain
\[
L_m((H_i^m)^{-1}(\gamma)) \leq e^{M(\overline{D}(2\alpha))} \cdot \frac{4\pi}{\delta} \cdot |z-z_0|^{\delta/2\pi} \leq e^{M(\overline{D}(2\alpha))} \cdot \frac{4\pi}{\delta} \cdot \eta^{\delta/2\pi} = 3\varepsilon/4
\]
(we have chosen $\eta$ so that the last equality is true). Since $(H_i^m)^{-1}(\gamma)$ is a continuous curve in $\Sigma$ joigning $(H_i^m)^{-1}(z_0)$ and $(H_i^m)^{-1}(z)$, we have $d_m((H_i^m)^{-1}(z_0),(H_i^m)^{-1}(z)) \leq 3\varepsilon/4$,
and with the inequality (\ref{inegalitetemporairedistance}) we obtain $d_m(x_j^m,(H_i^m)^{-1}(z)) < 3\varepsilon/4+\varepsilon/4 = \varepsilon$, and this ends the proof.
\end{proof}

\begin{proof}[Proof of proposition \ref{propositionalternative}]
The following fact is a direct consequence of lemma \ref{lemmealternative}:
\begin{fact}
Let $m \in \mathbb{N}$, $i,j \in \{1,...,N\}$, and $V \subset D(2\alpha)$ be an open set with diameter less than $\eta$. Then at least one of the following two properties is true:
\begin{itemize}
 \item (A') we have $V \subset H_i^m(B_m(x_i^m,\varepsilon) \cap B_m(x_j^m,\varepsilon))$,
 \item (B') we have $\varphi_j^m \circ (H_i^m)^{-1}=0$ on $V$.
\end{itemize}
\end{fact}
\begin{proof}
Assume (B') is not true. Then there exists some $z_0 \in V$ with $\varphi_j^m \circ (H_i^m)^{-1}(z_0) \neq 0$. And every $z \in V$ satisfies $|z-z_0| \leq \eta$, so $(H_i^m)^{-1}(z) \in B_m(x_j^m,\varepsilon)$ by lemma \ref{lemmealternative}: this shows that property (A') is true.
\end{proof}

Now, cover $D(2\alpha)$ by a finite number of open sets $V_t$, for $t \in T$, with diameter less than $\eta$. For every $i,j \in \{1,...,N\}$ and every $t \in T$, by the preceding fact, there exists an infinite number of integer $m$ verifying the same proposition, (A') or (B'). Hence there exists a subsequence $m'$ of $m$ such that this property ((A') or (B')) is verified for every $m'$. Taking a finite number of successive extractions, when $i,j \in \{1,...,N\}$ and $t\in T$, we obtain proposition \ref{propositionalternative}.
\end{proof}

\subsubsection{Convergence of the transition maps} \label{partieconvergencefonctionstransition}

Let $i,j \in \{1,..., N\}$. We want to show the convergences of the sequences of maps $\overline{H_j^m} \circ (H_i^m)^{-1}$ and $\varphi_j^m \circ (H_i^m)^{-1}$ on $D(2\alpha)$.

On an open set $V_t$ verifying property (B) in proposition \ref{propositionalternative}, we have, for every $m \in \N$, $\overline{H_j^m} \circ (H_i^m)^{-1}=0$ and $\varphi_j^m \circ (H_i^m)^{-1}=0$ on $V_t$, hence the sequences converge trivially.

Consider some open set $V_t$ such that property (A) in proposition \ref{propositionalternative} is satisfied. Then $H_j^m \circ (H_i^m)^{-1}$ is well defined on $V_t$. The maps $H_i^m$ and $H_j^m$ are conformal charts, so $H_j^m \circ (H_i^m)^{-1}$ is a conformal map between open subsets of $\C$: this classical property for surfaces with smooth Riemannian metrics extends to the class of surfaces with B.I.C. (this is theorem 7.3.1 in \cite{Reshetnyak_livre}). $(H_j^m \circ (H_i^m)^{-1})_{m\in\N}$ is then a sequence of uniformly bounded holomorphics (or anti-holomorphics) maps on $V_t$: by Montel's theorem, we know that after passing to a subsequence, $H_j^m \circ (H_i^m)^{-1}$ converge locally uniformly (as well as the derivatives) to some holomorphic (or anti-holomorphic) map on $V_t$. 

Let $A_{ji} := $ the union of the open sets $V_t$ such that property (A) in proposition \ref{propositionalternative} is satisfied, and $B_{ji} := $ the union of the open sets $V_t$ such that property (B) is satisfied. We have $D(2\alpha) = A_{ji} \cup B_{ji}$.

After considering successive subsequences, we can define a smooth map $H_{ji}$ on $A_{ji}$ by $H_{ji}(z) := \lim_{m\rightarrow \infty} H_j^m \circ (H_i^m)^{-1} (z)$.

For every $i,j \in \{1,...,N\}$ we have the following properties (after passing to a subsequence):
\begin{itemize}
\item there exists a smooth map $\varphi_{ji}$ on $D(2\alpha)$ such that
\[
\varphi_j^m \circ (H_i^m)^{-1} \underset{m \rightarrow \infty} \longrightarrow \varphi_{ji} \mbox{ locally uniformly (as well as the derivatives) on } D(2\alpha):
\]
$\varphi_{ji}$ is defined by $\varphi_{ji} = \varphi \circ H_{ji}$ on $A_{ji}$, and $\varphi_{ji}=0$ on $B_{ji}$.
\item There exists a smooth map $\overline{H_{ji}}$ on $D(2\alpha)$ such that
\[
\overline{H_j^m} \circ (H_i^m)^{-1} \underset{m \rightarrow \infty} \longrightarrow \overline{H_{ji}} \mbox{ locally uniformly (as well as the derivatives) on } D(2\alpha):
\]
$\overline{H_{ji}}$ is defined by $\overline{H_{ji}} = \varphi_{ji} H_{ji}$ on $A_{ji}$, and $\overline{H_{ji}}= 0$ on $B_{ji}$.
\end{itemize}

\subsubsection{Construction of the limit embedded surface $\Sigma^\infty$} \label{partieconstructionsurfacelimite}

Let $m$ tend to infinity in relation (\ref{eqcartesgraphees}): for $z \in D(5\alpha/3)$, set
\begin{equation} \label{equationdefinitionphiiinfini}
\Phi_i^\infty (z) := \bigg(\overline{H_{1i}}(z),...,z,...,\overline{H_{Ni}}(z), \varphi_{1i}(z),...,1,...,\varphi_{Ni}(z)\bigg).
\end{equation}
We also define the following subset of $\R^q$:
\[
 \Sigma^\infty := \bigcup_{i=1}^N \Phi_i^\infty (D(5\alpha/3)).
\]
If $x \in \Sigma^\infty$ is in the open set $\Phi_i^\infty (D(5\alpha/3))$, we will say that $x$ is in the graph number $i$. Since $\Sigma^m$ is covered by the sets $(H_i^m)^{-1}(D(\alpha))$, for $1 \leq i \leq N$, the following proposition is a straightforward verification:
\begin{proposition} \label{propositionrecouvrementsigmainfini}
We have
\[
 \Sigma^\infty = \bigcup_{i=1}^N \Phi_i^\infty (\overline{D}(\alpha)),
\]
hence
\[
 \Sigma^\infty = \bigcup_{i=1}^N \Phi_i^\infty (D(4\alpha/3)).
\]
\end{proposition}
\begin{proof}
Let $x \in \Sigma^\infty$: by definition, there exists some points $x_m \in \Sigma^m$ with $x_m \rightarrow x$. Every $x_m$ belongs to some open set $\Phi_{i(m)}^m(D(\alpha))$, for some $i(m) \in \{1,...,N\}$: after passing to a subsequence, we may assume this $i(m)$ is constant. For simplicity, assume $i(m)=1$. Thus there exists a sequence of complex numbers $z_m \in D(\alpha)$ with
\[
x_m = \Phi_1^m(z_m) =  \bigg(z_m,...,\overline{H_N^m}((H_1^m)^{-1}(z_m)),1,...,\varphi_{N}^m((H_1^m)^{-1}(z_m))\bigg).
\]
By compactness, after passing to a subsequence, we may assume $z_m \rightarrow z \in \overline{D}(\alpha)$, and by uniform convergence of all the maps which appear in the last equality, we get $x = \Phi_1^\infty (z)$, and this ends the proof.
\end{proof}

We easily see that such a "limit" of embedded submanifolds may not be a submanifold:
\begin{center}
\begin{tikzpicture}[scale=0.5]
\draw [domain = 0:4.5] [samples = 500] plot (\x, {0.3+3*sqrt(2*\x) *exp(-2*\x)}) ;
\draw [domain = 0:4.5] [samples = 500] plot (9-\x, {0.3+3*sqrt(2*\x) *exp(-2*\x)}) ;
\draw (0,0.3)--(0,-0.3) ;
\draw (9,0.3)--(9,-0.3) ;
\draw [domain = 0:4.5] [samples = 500] plot (\x, {-0.3-3*sqrt(2*\x) *exp(-2*\x)}) ;
\draw [domain = 0:4.5] [samples = 500] plot (9-\x, {-0.3-3*sqrt(2*\x) *exp(-2*\x)}) ;
\draw (12,0) node [scale=2] {$\longrightarrow$} ;
\draw [domain = 15:19.5] [samples = 500] plot (\x, {3*sqrt(2*(\x-15)) *exp(-2*(\x-15))}) ;
\draw [domain = 15:19.5] [samples = 500] plot (39-\x, {3*sqrt(2*(\x-15)) *exp(-2*(\x-15))}) ;
\draw [domain = 15:19.5] [samples = 500] plot (\x, {-3*sqrt(2*(\x-15)) *exp(-2*(\x-15))}) ;
\draw [domain = 15:19.5] [samples = 500] plot (39-\x, {-3*sqrt(2*(\x-15)) *exp(-2*(\x-15))}) ;
\end{tikzpicture}
\captionof{figure}{}
\end{center}

But in this case, we have the following

\begin{proposition} \label{propositionsousvariétélisse}
$\Sigma^\infty$ is a (possibly disconnected) smooth embedded compact surface in $\R^q$.
\end{proposition}

This is a straightforward consequence of the following technical lemma, which will also be used later: points of $\Sigma^m$ (or points of $\Sigma^\infty$) which are close to points in the graph number $i$, are also in the graph number $i$.

\begin{lemma} \label{lemmetechniquecartesgraphees}
The following properties are true for every $i \in \{1,...,N\}$:
\begin{itemize}
\item[1.] Let $m \in \N$, and $x_0 \in \Phi_i^m (D(4\alpha/3)) \subset \Sigma^m$. For every $x \in \Sigma^m$,
\[
 ||x-x_0|| < \alpha/6 \Longrightarrow x \in \Phi_i^m(D(5\alpha/3)).
\]
\item[2.] Let $x_0 \in \Phi_i^\infty (D(4\alpha/3)) \subset \Sigma^\infty$. For every $x \in \Sigma^\infty$,
\[
 ||x-x_0|| < \alpha/6 \Longrightarrow x \in \Phi_i^\infty(D(5\alpha/3)).
\]
\item[2'.] Let $x_0 \in \Phi_i^\infty (D(\alpha)) \subset \Sigma^\infty$. For every $x \in \Sigma^\infty$,
\[
 ||x-x_0|| < \alpha/6 \Longrightarrow x \in \Phi_i^\infty(D(7\alpha/6)).
\]
\end{itemize}
($||.||$ is the Euclidean norm of $\R^q$.)
\end{lemma}

\begin{proof}[Proof of proposition \ref{propositionsousvariétélisse}]
Compactness follows from proposition \ref{propositionrecouvrementsigmainfini}. Now, let $x_0 \in \Sigma^\infty$. By proposition \ref{propositionrecouvrementsigmainfini}, there exists $i \in \{1,...,N\}$ such that $x_0 \in \Phi_i^\infty(D(4\alpha/3))$. Second part of lemma \ref{lemmetechniquecartesgraphees} shows that
\[
B_{\euc}(x_0,\alpha/6) \cap \Sigma^\infty = B_{\euc} (x_0,\alpha/6) \cap \Phi_i^\infty(D(5\alpha/3))
\]
($B_{\euc}(x_0,\alpha/6)$ is the Euclidean ball with center $x_0$ and radius $\alpha/6$).
Since $\Phi_i^\infty(z)$ is a graph of a map (see the equality \ref{equationdefinitionphiiinfini}), this shows that $\Sigma^\infty$ is a submanifold of $\R^q$.
\end{proof}

\begin{proof}[Proof of lemma \ref{lemmetechniquecartesgraphees}]
We do the computations in the case $i=1$.

\noindent\underline{Proof of 1).} There exists some $z_0 \in D(4\alpha/3))$ such that
\[
x_0=\Phi_1^m(z_0)=\bigg(z_0,...,\overline{H_N^m}((H_1^m)^{-1}(z_0)),1,...,\varphi_{N}^m((H_1^m)^{-1}(z_0))\bigg),
\]
and we consider some $x \in \Sigma^m$ with $||x-x_0|| < \alpha/6$. There exists an integer $i \in \{1,...,N\}$ and $z \in D(4\alpha/3)$ such that
\[
x=\bigg(\overline{H_1^m}((H_i^m)^{-1}(z)),...,z,...,\overline{H_N^m}((H_i^m)^{-1}(z)), \varphi_1^m((H_i^m)^{-1}(z)) ,...,1,...,\varphi_{N}^m((H_i^m)^{-1}(z))\bigg).
\]
Set $z' = \overline{H_1^m}((H_i^m)^{-1}(z))$: we want to show $|z'|<5\alpha/3$ and $x = \Phi_1^m(z')$.

Since $|z'-z_0| \leq ||x-x_0|| < \alpha/6$, we have $|z'| < 4\alpha/3 + \alpha/6= 3\alpha/2$. For the same reason, $|\varphi_1^m((H_i^m)^{-1}(z))-1| \leq ||x-x_0|| < \alpha/6 < 1/10$, so $\varphi_1^m((H_i^m)^{-1}(z)) > 9/10$.

Since $\varphi_1^m((H_i^m)^{-1}(z)) \neq 0$, we know that $H_1^m((H_i^m)^{-1}(z))$ exists, so we have $z'=\overline{H_1^m}((H_i^m)^{-1}(z))=\varphi_1^m((H_i^m)^{-1}(z))\cdot H_1^m((H_i^m)^{-1}(z))$. Since $\varphi_1^m((H_i^m)^{-1}(z)) \geq 9/10$ we have
\[
 |z'| \geq \frac{9}{10} \cdot |H_1^m((H_i^m)^{-1}(z))|,
\]
hence
\[
  |H_1^m((H_i^m)^{-1}(z))| \leq \frac{10}{9} \cdot |z'| <  \frac{10}{9} \cdot \frac{3\alpha}{2} = \frac{5\alpha}{3}.
\]
Since $\varphi =1$ on $D(5\alpha/3)$, we get
\[
 \varphi_1^m((H_i^m)^{-1}(z)) = \varphi(H_1^m((H_i^m)^{-1}(z)))) =1,
\]
and we finally obtain $z' = H_1^m((H_i^m)^{-1}(z))$. We already have $|z'| < 5\alpha/3$. To show the equality $x = \Phi_1^m(z')$, we need to show
\[
x=\bigg(z',...,\overline{H_N^m}((H_1^m)^{-1}(z')),1,...,\varphi_{N}^m((H_1^m)^{-1}(z'))\bigg),
\]
so we need to prove the following equalities, for $j \geq 2$:
\[
\overline{H_j^m}((H_1^m)^{-1}(z')) = \overline{H_j^m}((H_i^m)^{-1}(z))
\]
and
\[
\varphi_{j}^m((H_1^m)^{-1}(z') = \varphi_{j}^m((H_i^m)^{-1}(z)),
\]
and these are directs consequences of the equality $(H_1^m)^{-1}(z') = (H_i^m)^{-1}(z)$.

\noindent\underline{Proof of 2).} (The proof of 2'. is perfectly analogous). The proof looks like the proof of 1., only the end will change.

There exists some $z_0 \in D(4\alpha/3))$ such that
\[
x_0 = \Phi_1^\infty(z_0) = \bigg(z_0,\overline{H_{21}}(z_0),...,\overline{H_{N1}}(z_0), 1,\varphi_{21}(z_0)...,\varphi_{N1}(z_0)\bigg),
\]
and some $i \in \{1,...,N\}$ and $z \in D(4\alpha/3)$ with
\[
x = \bigg(\overline{H_{1i}}(z),...,z,...,\overline{H_{Ni}}(z), \varphi_{1i}(z),...,1,...,\varphi_{Ni}(z)\bigg).
\]
We set $z' = \overline{H_{1i}}(z)$, and we want to show that $|z'|<5\alpha/3$, and $x = \Phi_1^\infty(z')$.

For the same reasons than in the proof of 1., we have $|z'|<3\alpha/2$, and $\varphi_{1i}(z)>9/10$: $H_{1i}(z)$ exists (that is $z$ is in some open set $V_t$ verifying property (A) in proposition \ref{propositionalternative}), and we have
\[
 |H_{1i}(z)| < 5\alpha/3.
\]
We get $\varphi_{1i}(z) = 1$, and finally $z' = \varphi_{1i}(z)  \cdot H_{1i}(z) = H_{1i}(z)$.
We have to show the equality:
\[
x = \Phi_1^\infty(z') = \bigg(z',\overline{H_{21}}(z'),...,\overline{H_{N1}}(z'), 1,\varphi_{21}(z')...,\varphi_{N1}(z')\bigg).
\]
So we need to show the following equalities, for $j \geq 2$:
\[
 \overline{H_{j1}}(z') = \overline{H_{ji}}(z)
\]
and
\[
 \varphi_{j1}(z') = \varphi_{ji}(z).
\]
The first equality writes
\[
 \lim_{m \rightarrow \infty} \overline{H_j^m} \circ (H_1^m)^{-1} \bigg( \lim_{m'\rightarrow \infty} H_1^{m'} \circ (H_i^{m'})^{-1} (z) \bigg) = \lim_{m \rightarrow \infty} \overline{H_j^m} \circ (H_i^m)^{-1} (z),
\]
and this equality is true because all the convergences are uniform. The second equality writes
\[
 \lim_{m \rightarrow \infty} \varphi_j^m \circ (H_1^m)^{-1} \bigg( \lim_{m'\rightarrow \infty} H_1^{m'} \circ (H_i^{m'})^{-1} (z) \bigg) = \lim_{m \rightarrow \infty} \varphi_j^m \circ (H_i^m)^{-1} (z),
\]
and is true for the same reason.
\end{proof}

\subsection{Construction of a diffeomorphism $\Pi^m : \Sigma^m \rightarrow \Sigma^\infty$}

For $m$ large enough, $\Sigma^m$ is in a tubular neighborhood of $\Sigma^\infty$: hence we can define a projection $\Sigma^m \rightarrow \Sigma^\infty$. Since $\Sigma^m$ converges to $\Sigma^\infty$ (in the sense given above), this projection is actually a diffeomorphism.

\subsubsection{Construction of a projection $\Pi^m: \Sigma^m \rightarrow \Sigma^\infty$}

$\Sigma^\infty$ is a smooth compact embedded surface in $\R^q$, possibly disconnected, with only a finite number of connected components. We can thus consider the normal projection onto $\Sigma^\infty$: there exists $\tau>0$ (we may assume $\tau<\alpha/12$), a tubular neighborhood
\[
\mathcal{V} = \{x \in \R^q ~|~ d_{\euc}(x,\Sigma^\infty) < \tau \}
\]
and a smooth projection $\Pi : \mathcal{V} \rightarrow \Sigma^\infty$ verifying the following property (see \cite{Bredon}): if $x \in \mathcal{V}$, then $\Pi(x)$ is the closest point of $\Sigma^\infty$. For every $x \in \mathcal{V}$ we then have
\[
x - \Pi(x) \in (T_{\pi(x)} \Sigma^\infty)^\bot
\]
(see the picture below), where we denote by $T_{\pi(x)} \Sigma^\infty \subset \R^q$ is the tangent space of $\Sigma^\infty$ at the point $\pi(x)$, and $(T_{\pi(x)} \Sigma^\infty)^\bot$ its orthogonal in $\R^q$.

Thanks to section \ref{partieconvergencefonctionstransition}, we know that after passing to a subsequence the following is true:
\begin{fact} \label{propositiondifferenceuniformephii}
For every $m \in \N, i \in \{1,...,N\}$ and $z \in D(5\alpha/3)$ we have
\[
||\Phi_i^m(z) - \Phi_i^\infty(z)|| < \tau.
\]
\end{fact}

Since every $x \in \Sigma^m$ can be written $x = \Phi_i^m(z)$ for some $i \in \{1,..,N\}$ and some $z \in D(\alpha)$, we have $d_{\euc}(x,\Sigma^\infty) <\tau$, so $\Sigma^m \subset \mathcal{V}$. We can thus consider the following restriction:
\[
\Pi^m := \Pi_{|\Sigma^m} : \Sigma^m \rightarrow \Sigma^\infty.
\]

\begin{center}
\begin{tikzpicture}
\draw [domain=0.3:1][samples=200] plot (4*\x,{1.2*sqrt(1-\x*\x)});
\draw [domain=0.7:1.3][samples=200] plot (4*\x,{2/\x});
\draw (4*0.8,{1.2*sqrt(1-0.64)}) node  {$\bullet$};
\draw (4*0.8,{1.2*sqrt(1-0.64)}) node [below left] {$\Pi^m(x)$};
\draw (4*0.93,2/0.93) node {$\bullet$};
\draw (4*0.93,2/0.93) node [above right] {$x$};
\draw [dashed] (4*0.8,{1.2*sqrt(1-0.64)}) -- (4*0.93,2/0.93) ;
\draw (2.4,3) node {$\Sigma^m$};
\draw (1,1.4) node {$\Sigma^\infty$};
\end{tikzpicture}
\captionof{figure}{}
\end{center}

\subsubsection{$\Pi^m: \Sigma^m \rightarrow \Sigma^\infty$ is a diffeomorphism}

We want to show the following
\begin{proposition} \label{propositiondiffeomorphisme}
After passing to a subsequence, for every $m \in \N$, $\Pi^m : \Sigma^m \rightarrow \Sigma^\infty$ is a $\mathcal{C}^\infty$ diffeomorphism.
\end{proposition}

For technical reasons, we first show that for every $i \in \{1,...,N\}$, points in $\Sigma^m$ in the graph number $i$ are sent to points in $\Sigma^\infty$ in the graph number $i$, and conversely:
\begin{proposition} \label{propositioninclusionpim}
The following inclusions are true for every $m \in \N$ and every $i \in \{1,...,N\}$:
\[
1. ~ (\Pi^m)^{-1}(\Phi_i^\infty(D(4\alpha/3))) \subset \Phi_i^m(D(5\alpha/3)),
\]
\[
2. ~ \Pi^m(\Phi_i^m(D(4\alpha/3))) \subset \Phi_i^\infty(D(5\alpha/3)),
\]
and
\[
2'. ~ \Pi^m(\Phi_i^m(D(\alpha))) \subset \Phi_i^\infty(D(7\alpha/6)).
\]
\end{proposition}
\begin{proof}
This is a straightforward consequence of lemma \ref{lemmetechniquecartesgraphees} and fact \ref{propositiondifferenceuniformephii}.

\noindent Proof of 1). Let $x \in (\Pi^m)^{-1}(\Phi_i^\infty(D(4\alpha/3)))$. There exists $z \in D(4\alpha/3)$ such that $\Pi^m(x) = \Phi_i^\infty(z)$. By proposition \ref{propositiondifferenceuniformephii} we have 
\[
 ||\Pi^m(x)-\Phi_i^m(z)|| = ||\Phi_i^\infty(z)-\Phi_i^m(z)|| <\tau<\alpha/12,
\]
and we also have $||\Pi^m(x)-x||<\tau<\alpha/12$ (by definition of the normal projection $\Pi$). Hence $||x - \Phi_i^m(z)||<\alpha/6$, and the identity 1. in lemma \ref{lemmetechniquecartesgraphees} shows that $x \in \Phi_i^m(D(5\alpha/3))$.

\noindent Proof of 2). (The proof of 2') is perfectly analogous, using 2') in lemma \ref{lemmetechniquecartesgraphees} instead of 2)) Let $z \in D(4\alpha/3)$ and $x = \Phi_i^m(z)$. By proposition \ref{propositiondifferenceuniformephii} we have
\[
||x - \Phi_i^\infty(z)|| = ||\Phi_i^m(z) - \Phi_i^\infty(z)|| < \tau < \alpha/12,
\]
and we also have $||x - \Pi^m(x)|| < \tau < \alpha/12$. Hence $||\Pi^m(x) - \Phi_i^\infty(z)|| < \alpha/6$, and we can use the identity 2. in lemma \ref{lemmetechniquecartesgraphees} to show that $\Pi^m(x) \in \Phi_i^\infty(D(5\alpha/3))$.
\end{proof}

We can now prove the following
\begin{proposition}
$\Sigma^\infty$ is path-connected.
\end{proposition}

\begin{proof}
Let $x = \Phi_i^\infty(z)$ and $y = \Phi_j^\infty(z')$ be two points in $\Sigma^\infty$, with $z,z' \in D(4\alpha/3)$. For $m=1$, let $\gamma$ be a continuous path in $\Sigma^m$ joigning $\Phi_i^m(z)$ and $\Phi_j^m(z')$ (recall that $\Sigma^m$ is connected). Then, $\Pi^m \circ \gamma$ is a continuous path in $\Sigma^\infty$ joigning $\Pi^m(\Phi_i^m(z))$ and $\Pi^m(\Phi_j^m(z'))$. We have
\[
\Pi^m(\Phi_i^m(z)) \in \Pi^m(\Phi_i^m(D(4\alpha/3))) \subset \Phi_i^\infty(D(5\alpha/3)),
\]
and since $x \in \Phi_i^\infty(D(5\alpha/3))$ and $\Phi_i^\infty(D(5\alpha/3))$ is path-connected, we can join $\Pi^m(\Phi_i^m(z))$ and $x$ by a continuous path. For the same reason we can also join $\Pi^m(\Phi_j^m(z'))$ and $y$ by a continuous path, thus we can join $x$ and $y$ by a continuous path.
\end{proof}
To prove proposition \ref{propositiondiffeomorphisme}, we only need to prove the following
\begin{lemma} \label{lemmeimmersioninjective}
After passing to a subsequence, for every $m \in \N$, $\Pi^m$ is an injective immersion.
\end{lemma}

\begin{proof}[Proof of proposition \ref{propositiondiffeomorphisme}]
$\Pi^m$ is a diffeomorphism onto its image, which is an open and closed subset of $\Sigma^\infty$, thus is $\Sigma^\infty$ itself by connectedness.
\end{proof}

\begin{proof}[Proof of lemma \ref{lemmeimmersioninjective}]
There are two distinct steps. We prove both steps by contradiction: roughly speaking, since $\Sigma^m$ converges to $\Sigma^\infty$ (in the sense given above), the tangent spaces have to converge as well, and this will give a contradiction.

\noindent
\emph{First step:} we show that after passing to a subsequence, $\Pi^m : \Sigma^m \rightarrow \Sigma^\infty$ is an immersion.

Suppose this is not true. Then $\Pi^m$ is an immersion only for a finite number of $m \in \N$: there exists $M_0 \in \N$ such that $\Pi^m$ is not an immersion for $m \geq M_0$. We assume that $m \geq M_0$.

There exists a sequence $x_m \in \Sigma^m$ verifying $\ker(D\Pi^m (x_m)) \neq \{0\}$. By compactness, we may assume $x_m \rightarrow x \in \Sigma^\infty$; we also have $\Pi^m(x_m) = \Pi(x_m) \rightarrow \Pi(x)=x$ (recall that $\Pi$ is the normal projection onto $\Sigma^\infty$, and $\Pi^m$ its restriction to $\Sigma^m$).

Moreover, $x \in \Phi_i^\infty(D(4\alpha/3)$ for some $i \in \{1,...,N\}$: for simplicity, we may assume $i=1$. Let $z \in D(4\alpha/3)$ such that $x = \Phi_1^\infty(z)$. For $m$ large enough, $\Pi^m(x_m) \in \Phi_1^\infty(D(4\alpha/3))$, so by proposition \ref{propositioninclusionpim} we get
\[
x_m \in (\Pi^m)^{-1}(\Phi_1^\infty(D(4\alpha/3))) \subset \Phi_1^m(D(5\alpha/3)).
\]
We then have sequences $z_m$ and $z'_m$ in $D(5\alpha/3)$ such that $x_m = \Phi_1^m(z_m)$ and $\Pi^m(x_m) = \Phi_1^\infty(z'_m)$. For simplicity, we write $x_m$ and $\Pi^m(x_m)$ under the following form:
\[
x_m = (z_m, \Theta^m(z_m)) \mbox{ and } \Pi^m(x_m) = (z'_m, \Theta^\infty(z'_m)),
\]
where $\Theta^m$ and $\Theta^\infty$ are smooth maps, and $\Theta^m \rightarrow \Theta^\infty$ uniformly (and all the derivatives) on every compact set of $D(2\alpha)$ (see section \ref{partieconvergencefonctionstransition}). We have
\[
\ker(D\Pi^m (x_m)) = T_{x_m} \Sigma^m \cap (T_{\Pi^m(x_m)} \Sigma^\infty)^\bot \neq \{0\},
\]
so we can consider a unit vector $u_m$ in this vector space. We know a basis of $T_{x_m} \Sigma^m$, so there exists real numbers $a_m$ and $b_m$ such that
\[
u_m = a_m (1,0,\partial_x \Theta^m(z_m)) + b_m (0,1,\partial_y \Theta^m (z_m)).
\]
Since $u_m$ is a unit vector, we have $|a_m|\leq 1$ and $|b_m|\leq 1$. We can consider the following vector in $T_{\Pi^m(x_m)} \Sigma^\infty$:
\[
v_m = a_m (1,0,\partial_x \Theta^\infty (z'_m)) + b_m (0,1,\partial_y \Theta^\infty(z'_m)).
\]
Since $u_m$ and $v_m$ are orthogonal, we have $1 = ||u_m||^2 \leq ||u_m - v_m||^2$, so $1 \leq ||u_m - v_m||$ and
\begin{eqnarray*}
 1 & \leq & |a_m| \cdot  ||\partial_x \Theta^m(z_m) - \partial_x\Theta^\infty(z'_m)|| + |b_m| \cdot ||\partial_y \Theta^m (z_m) - \partial_y\Theta^\infty(z'_m)|| \\
  & \leq & ||\partial_x \Theta^m(z_m) - \partial_x \Theta^\infty (z'_m)|| + ||\partial_y \Theta^m(z_m) - \partial_y \Theta^\infty(z'_m)||.
\end{eqnarray*}
This is a contradiction: when $m$ goes to infinity, $x_m$ and $\Pi^m(x_m)$ converge to $x$, so $z_m$ and $z'_m$ converge to $z$, and we have uniform convergence of the derivatives of $\Theta^m$ to the derivatives of $\Theta^\infty$, which shows that the right-hand side term of the inequality goes to zero.

\noindent
\emph{Second step:} we show that after passing to a subsequence, $\Pi^m : \Sigma^m \rightarrow \Sigma^\infty$ is injective.

Suppose this is not true. Then $\Pi^m$ is injective only for a finite number of $m \in \N$: there exists $M_0 \in \N$ such that $\Pi^m$ is not injective for every $m \geq M_0$. We assume that $m \geq M_0$.

There exists sequences $x_m, x'_m \in \Sigma^m$, with $x_m \neq x'_m$, such that $\Pi^m(x_m)= \Pi^m(x'_m)$, so we have $x_m - x'_m \in (T_{\Pi^m(x_m)} \Sigma^\infty)^\bot$:
\begin{center}
\begin{tikzpicture}
\draw [domain = 90:270] plot(\x:1.5);
\draw (0,1.5) -- (1,1.5);
\draw (0,-1.5) -- (1,-1.5);
\draw (-2,0) -- (1,0);
\draw (1,1.5) node [right] {$\Sigma^m$}; 
\draw (-2,0) node [left] {$\Sigma^\infty$};
\draw ({1.5*cos(130)},{1.5*sin(130)}) node {$\bullet$};
\draw ({1.5*cos(130)},{1.5*sin(130)}) node [above left] {$x_m$};
\draw ({1.5*cos(230)},{1.5*sin(230)}) node {$\bullet$};
\draw ({1.5*cos(230)},{1.5*sin(230)}) node [below left] {$x'_m$};
\draw [dashed] ({1.5*cos(230)},{1.5*sin(230)}) -- ({1.5*cos(130)},{1.5*sin(130)});
\draw ({1.5*cos(230)},0) node {$\bullet$};
\draw ({1.5*cos(230)},0) node [above right] {$\Pi^m(x_m)=\Pi^m(x'_m)$};
\end{tikzpicture}
\captionof{figure}{}
\end{center}
We can also suppose that $x_m$ and $x'_m$ converge, and these sequences have the same limit $x \in \Sigma^\infty$, since $\lim x_m= \lim \Pi^m(x_m)$ and  $\lim x'_m= \lim \Pi^m(x'_m)$. We know that there exists some $i \in \{1,...,N\}$ such that $x \in \Phi_i^\infty(D(4\alpha/3))$; for simplicity, we may assume $i=1$. There exists $z \in D(4\alpha/3)$ such that $x = \Phi_1^\infty(z)$. If $m$ is large enough, $x_m$ and $x'_m$ are also in $\Phi_1^\infty(D(4\alpha/3))$, so there exists $z_m$ and $z'_m$ such that
\[
x_m = (z_m, \Theta^m(z_m)) \mbox{ and } x'_m = (z'_m, \Theta^m(z'_m))
\]
(with the notations as above); we then have
\[
x_m - x'_m = (z_m - z'_m, \Theta^m(z_m) - \Theta^m(z'_m)).
\]
If $m$ is large enough, $\Pi^m(x_m) = \Pi^m(x_m')$ is also in $\Phi_1^\infty(D(4\alpha/3))$, so we can also write
\[
\Pi^m(x_m) = \Pi^m(x'_m) = (u_m, \Theta^\infty(u_m)),
\]
for some $u_m \in D(4\alpha/3)$.

Write $z_m = a_m + i b_m$ and $z'_m = a'_m + i b'_m$ for $a_m, b_m \in \R$ and consider the following vector in $T_{\Pi^m(x_m)} \Sigma^\infty$:
\[
\frac{a_m - a'_m}{|z_m - z'_m|^2} \cdot \big( 1,0, \partial_x \Theta^\infty(u_m) \big) + \frac{b_m - b'_m}{|z_m - z'_m|^2} \cdot \big(0,1, \partial_y \Theta^\infty(u_m) \big)
\]
($x_m \neq x'_m$ implies $z_m \neq z'_m$).
By taking the scalar product with $x_m - x'_m \in (T_{\Pi^m(x_m)} \Sigma^\infty)^\bot$ we obtain
\[
0 = 1 + <\frac{\Theta^m(z_m) - \Theta^m(z'_m)}{|z_m - z'_m|}, D\Theta^\infty(u_m) \left( \frac{z_m-z'_m}{|z_m - z'_m|} \right) >
\]
(we denote by $<~,~>$ the Euclidean scalar product in $\R^{q-2}$). We want to use the mean-value theorem, hence we need to consider real-valued maps.
We can write the components of $\Theta^m$ and $\Theta^\infty$ as
\[
\Theta^m = (\Theta^{m,1},..., \Theta^{m,q-2}) \mbox{ and } \Theta^\infty = (\Theta^{\infty,1},..., \Theta^{\infty,q-2})
\]
with functions $\Theta^{m,j}, \Theta^{\infty,j} : D(5\alpha/3) \rightarrow \R$. We can then write
\[
0 = 1+\sum_{j=1}^{q-2} \left(\frac{\Theta^{m,j}(z_m) - \Theta^{m,j}(z'_m)}{|z_m - z'_m|}\right) \cdot D\Theta^{\infty,j}(u_m) \left( \frac{z_m-z'_m}{|z_m - z'_m|} \right).
\]
Since the $\Theta^{m,j}$ are functions with values in $\R$, by the mean-value theorem, we know that for every $j \in \{1,...,q-2\}$, there exists some $\zeta_m^j \in [z_m,z'_m]$ such that $\Theta^{m,j}(z_m) - \Theta^{m,j}(z'_m) = D\Theta^{m,j}(\zeta_m^j) \cdot (z_m - z'_m)$: we then have
\[
0 = 1+\sum_{j=1}^{q-2} D\Theta^{m,j}(\zeta_m^j) \cdot \left( \frac{z_m-z'_m}{|z_m - z'_m|} \right) \cdot D\Theta^{\infty,j}(u_m) \left( \frac{z_m-z'_m}{|z_m - z'_m|} \right).
\]
By compactness we can suppose that $\frac{z_m-z'_m}{|z_m-z'_m|} \rightarrow v \in \mathbb{S}^1$, and since $\Theta^m$ (and its derivatives) converge to $\Theta^\infty$, we get $0 = 1 + ||D\Theta^\infty(z) \cdot (v) ||^2$
and this is a contradiction.
\end{proof}

\subsection{End of the proof of the Main theorem} \label{sectionconclusion}

We have constructed the following diffeomorphisms:
\[
\Sigma ~ \underset{\Psi^m} {\overset{\sim} \longrightarrow} ~ \Sigma^m \subset \R^q ~ \underset{\Pi^m}{\overset{\sim }\longrightarrow} ~ \Sigma^\infty \subset \R^q.
\]
Recall that $\Psi^m$ is obtained by Whitney's embedding, and $\Pi^m$ is the restriction to $\Sigma^m$ of the normal projection onto $\Sigma^\infty$. We can consider the following metric on $\Sigma^\infty$:
\[
\widetilde{d_m} := ((\Pi^m \circ \Psi^m)^{-1})^* d_m,
\]
that is
\[
\widetilde{d_m}(x,y) = d_m((\Pi^m \circ \Psi^m)^{-1}(x),(\Pi^m \circ \Psi^m)^{-1}(y)),
\]
so that $(\Sigma,d_m)$ is isometric to $(\Sigma^\infty,\widetilde{d_m})$. To finish the proof of the Main theorem, we need to show that $\widetilde{d_m}$ converges uniformly to some metric with B.I.C. $\widetilde{d}$ on $\Sigma^\infty$.

In section \ref{sectionlocalproperties}, we prove that $(\widetilde{d_m}(x,y))$ converges, if $x$ and $y$ are in the same graph $\Phi_i^\infty(D(4\alpha/3))$. Then, in section \ref{sectionconstructionetconclusion}, we prove that $(\widetilde{d_m}(x,y))$ converges for every $x$ and $y$ in $\Sigma$; we can define $\widetilde{d}(x,y):=\lim_{m\rightarrow \infty} \widetilde{d_m}(x,y)$. To finish the proof of the main theorem, we show that $\widetilde{d_m}$ converges \emph{uniformly} to $\widetilde{d}$, and $\widetilde{d}$ is a metric with B.I.C.

\subsubsection{Local properties}\label{sectionlocalproperties}
By proposition \ref{propositioninclusionpim}, for every $m \in \N$ and every $i \in \{1,...,N\}$ we have 
\[
(\Pi^m)^{-1}(\Phi_i^\infty(D(4\alpha/3))) \subset \Phi_i^m(D(5\alpha/3)),
\]
so we can consider a map $f_i^m : D(4\alpha/3) \rightarrow D(5\alpha/3)$ such that the following diagram commutes:
\[
\begin{array}{ccl}
D(5\alpha/3) & \overset{\sim} \longleftarrow & \Phi_i^m(D(5\alpha/3)) \\
f_i^m \bigg\uparrow &  & (\Pi^m)^{-1} \bigg\uparrow \\
D(4\alpha/3) & \overset{\sim} \longrightarrow & \Phi_i^\infty(D(4\alpha/3))
\end{array}
\]
that is $f_i^m= (\Phi_i^m)^{-1} \circ (\Pi^m)^{-1} \circ \Phi_i^\infty$.

\begin{proposition} \label{propositionconvergenceverslinclusion}
For every $i \in \{1,...,N\}$, $f_i^m : D(4\alpha/3) \rightarrow D(5\alpha/3)$ converges uniformly to the inclusion $D(4\alpha/3) \hookrightarrow D(5\alpha/3)$.
\end{proposition}

\begin{proof}
For every $z \in D(4\alpha/3)$, let $z' = f_i^m(z) \in D(5\alpha/3)$. Since $z$ (resp. $z'$) is the $i$-th component of $\Phi_i^\infty(z)$ (resp. $\Phi_i^m(z')$), we have
\[
 |z-z'| \leq ||\Phi_i^\infty(z) - \Phi_i^m(z')|| = ||\Pi^m (\Phi_i^m (z')) - \Phi_i^m(z')|| \leq ||\Phi_i^\infty(z') - \Phi_i^m(z')||:
\]
the last inequality comes from the fact that $||\Pi^m (\Phi_i^m (z')) - \Phi_i^m(z')||$ is the distance between $\Phi_i^m(z')$ and the embedded surface $\Sigma^\infty$, and we have $\Phi_i^\infty(z') \in \Sigma^\infty$. We then have
\[
|z-f_i^m(z)| \leq \sup_{u \in D(5\alpha/3)} ||\Phi_i^\infty(u) - \Phi_i^m(u)||,
\]
and we know that the right-hand side goes to zero as $m$ goes to infinity.
\end{proof}

We know that for every $m \in \N$ and every $i \in \{1,...,N\}$ we have an isometry
\begin{equation} \label{isometriedistancelocale}
(B_m(x_i^m,\varepsilon), {d_m}_{|B_m(x_i^m,\varepsilon)}) ~ \underset{H_i^m}{\overset{\sim} \longrightarrow} ~ 
(D(1/2),d_{\omega_i^m,h_i^m}).
\end{equation}
Moreover, for every $i \in \{1,...,N\}$, by the important corollary \ref{corollaireconvergencedistance}, there exists a measure $\widetilde{\omega}_i$, with support in $\overline{D}(1/2)$, and a constant $C_i>0$ such that, after passing to a subsequence, $d_{\omega_i^m,h_i^m}$ converge locally uniformly on $D(2\alpha)$ to the metric $d^i :=  C_i \cdot \overline{d}_{\widetilde{\omega}_i,0}$ (when $m$ goes to infinity).

Now, if we consider the diagram at the beginning of section \ref{sectionconclusion} we can consider the following metric on $\Phi_i^\infty(D(4\alpha/3))$:
\[
\widetilde{d^i} := ((\Phi_i^\infty)^{-1})^* d^i.
\]
\begin{proposition} \label{propositionconvergencelocaledistance}
Let $i \in \{1,...,N\}$ and $(x_m),(y_m)$ be two sequences of points in $\Sigma^\infty$ such that $x_m \rightarrow x \in \Phi_i^\infty(D(4\alpha/3))$ and $y_m \rightarrow y \in \Phi_i^\infty(D(4\alpha/3))$. Then
\[
\widetilde{d_m}(x_m,y_m) \underset{m \rightarrow \infty}\longrightarrow \widetilde{d^i}(x,y).
\]
\end{proposition}

\begin{proof}
Suppose $m$ is large enough so that $x_m, y_m \in \Phi_i^\infty(D(4\alpha/3))$. By corollary \ref{corollaireconvergencedistance}, $d_{\omega_i^m,h_i^m}$ converge locally uniformly to $d^i$ on $D(2\alpha)$, and $f_i^m : D(4\alpha/3) \rightarrow D(5\alpha/3)$ converges uniformly to the inclusion $D(4\alpha/3) \hookrightarrow D(5\alpha/3)$ (this is proposition \ref{propositionconvergenceverslinclusion}), so
\begin{eqnarray*}
\widetilde{d^i}(x,y) & = & d^i((\Phi_i^\infty)^{-1} (x),(\Phi_i^\infty)^{-1} (y) \\
 & = & \lim_{m\rightarrow \infty} d_{\omega_i^m,h_i^m}(f_i^m \circ (\Phi_i^\infty)^{-1} (x_m),f_i^m \circ (\Phi_i^\infty)^{-1} (y_m)).
\end{eqnarray*}
And by the isometry (\ref{isometriedistancelocale}), we have, for every $z,z' \in D(5\alpha/3)$,
\[
 d_{\omega_i^m,h_i^m} (z,z') = {d_m}_{|B_m(x_i^m,\varepsilon)} ((H_i^m)^{-1}(z),(H_i^m)^{-1}(z')) = {d_m}((H_i^m)^{-1}(z),(H_i^m)^{-1}(z')):
\]
indeed, by theorem \ref{theoremerecouvrement}, we have $(H_i^m)^{-1}(z),(H_i^m)^{-1}(z') \in B_m(x_i^m,\varepsilon/4)$, so a curve which (almost) minimizes the distance $d_m((H_i^m)^{-1}(z),(H_i^m)^{-1}(z'))$ has to stay inside $B_m(x_i^m,\varepsilon)$ (we have $d_m((H_i^m)^{-1}(z),(H_i^m)^{-1}(z')) \leq \varepsilon/2$, and a curve which joins $(H_i^m)^{-1}(z)$ and $(H_i^m)^{-1}(z')$, and which is not contained in $B_m(x_i^m,\varepsilon)$ has a length $\geq 2 \cdot (\varepsilon - \varepsilon/4)>\varepsilon/2$). Hence we get
\[
\widetilde{d^i}(x,y) = \lim_{m\rightarrow \infty} d_m ((H_i^m)^{-1} \circ f_i^m \circ (\Phi_i^\infty)^{-1} (x_m),(H_i^m)^{-1} \circ f_i^m \circ (\Phi_i^\infty)^{-1} (y_m)).
\]

By definition of $\Phi_i^m$, we have $(H_i^m)^{-1} = (\Psi^m)^{-1} \circ \Phi_i^m$, and with the equality $f_i^m \circ (\Phi_i^\infty)^{-1} = (\Phi_i^m)^{-1} \circ (\Pi^m)^{-1}$ (see the commutative diagram at the beginning of this section) we obtain
\begin{eqnarray*}
\widetilde{d^i}(x,y)& = & \lim_{m\rightarrow \infty} d_m((\Psi^m)^{-1} \circ (\Pi^m)^{-1} (x_m),(\Psi^m)^{-1} \circ (\Pi^m)^{-1} (y_m)) \\
 & = & \lim_{m\rightarrow \infty} \widetilde{d_m}(x_m,y_m),
\end{eqnarray*}
and this ends the proof.
\end{proof}

\subsubsection{Construction of the limit metric $\widetilde{d}$ and conclusion} \label{sectionconstructionetconclusion}

We already know that $\lim_{m \rightarrow \infty} \widetilde{d_m}(x,y)$ exists if $x$ and $y$ are in the same graph (that is, if there exists some $i \in \{1,...,N\}$ with $x,y \in \Phi_i^\infty(D(4\alpha/3))$); see proposition \ref{propositionconvergencelocaledistance}. To prove that this limit exists for every $x,y \in \Sigma$, the idea is to consider $\widetilde{d_m}-$geodesics between $x$ and $y$, $\gamma_m : [0,1] \rightarrow \Sigma$, and to cut the segment $[0,1]$ into subintervals for which we can apply proposition \ref{propositionconvergencelocaledistance}.

\begin{proposition} \label{propositionexistencelimite}
Let $x,y \in \Sigma^\infty$: the limit $\lim_{m\rightarrow \infty} \widetilde{d_m}(x,y)$ exists, and we set
\[
\widetilde{d}(x,y) := \lim_{m\rightarrow \infty} \widetilde{d_m}(x,y) \in [0,+\infty).
\]
\end{proposition}

\begin{proof}
Let $x,y \in \Sigma^\infty$. Consider some subsequence $(m_j)$ of $(m)$ such that 
\[
\widetilde{d_{m_j}}(x,y) \underset{j \rightarrow \infty} \longrightarrow \liminf_{m \rightarrow \infty} \widetilde{d_m}(x,y).
\]
Let $\gamma_{m_j} : [0,1] \rightarrow \Sigma^\infty$ be minimizing geodesics for the metric $\widetilde{d_{m_j}}$, between $x$ and $y$ (with $\gamma_{m_j}(0)=x$ and $\gamma_{m_j}(1)=y$), and parametrized with constant speed: we have, for $t,t' \in [0,1]$,
\[
\widetilde{d_{m_j}}(\gamma_{m_j}(t),\gamma_{m_j}(t')) = \widetilde{d_{m_j}}(x,y) \cdot |t-t'|.
\]
After taking a subsequence of $m_j$ (as usual, we do not change the name of the sequence), the following fact is true:
\begin{fact} \label{faitconvergencefinal}
Let $n \in \N$ be an integer such that $D/n \leq \kappa \varepsilon/2$. For every $k \in \{0,...,n-1\}$, there exists $i(k) \in \{1,...,N\}$ such that for every $j \in \N$,
\[
\gamma_{m_j}(k/n) \mbox{ and } \gamma_{m_j}((k+1)/n) \mbox{ belong to } \Phi_{i(k)}^\infty(D(7\alpha/6)).
\]
\end{fact}
\begin{proof}
By the covering (\ref{equationrecouvrementsigma}), we know that for every $k \in \{0,...,n-1\}$ and for every $j \in \N$, there exists an integer $i(k,j) \in \{1,...,N\}$ such that
\begin{equation} \label{appartenancelemmetechniquefinal}
(\Pi^{m_j} \circ \Psi^{m_j})^{-1}(\gamma_{m_j}(k/n)) \in B_{m_j}(x_{i(k,j)}^{m_j},\kappa \varepsilon/2).
\end{equation}
Since $i(k,j)$ belongs to a finite set, after taking a subsequence of $(m_j)$, we may assume that for every $k \in \{0,...,n-1\}$, $i(k,j)$ does not depend on $j$: we can write $i(k,j) = i(k)$. To finish the proof, we will show that
\[
\gamma_{m_j}([k/n,(k+1)/n]) \subset \Phi_{i(k)}^\infty(D(7\alpha/6)).
\]
Let $t \in [k/n,(k+1)/n]$: we have
\[
\widetilde{d_{m_j}}(\gamma_{m_j}(t),\gamma_{m_j}(k/n)) = d_{m_j}(x,y) \cdot |t-k/n|,
\]
so
\[
d_{m_j}((\Pi^{m_j} \circ \Psi^{m_j})^{-1}(\gamma_{m_j}(t)),(\Pi^{m_j} \circ \Psi^{m_j})^{-1}(\gamma_{m_j}(k/n))) \leq D \cdot 1/n \leq \kappa \varepsilon/2,
\]
and with (\ref{appartenancelemmetechniquefinal}) this shows that
\[
(\Pi^{m_j} \circ \Psi^{m_j})^{-1}(\gamma_{m_j}(t)) \in B_{m_j}(x_{i(k)}^{m_j},\kappa \varepsilon).
\]
So
\[
(\Pi^{m_j})^{-1}(\gamma_{m_j}(t)) \in \Psi^{m_j}(B_{m_j}(x_{i(k)}^{m_j},\kappa \varepsilon)) = \Phi_{i(k)}^{m_j} \circ H_{i(k)}^{m_j} (B_{m_j}(x_{i(k)}^{m_j},\kappa \varepsilon)) \subset \Phi_{i(k)}^{m_j}(D(\alpha))
\]
(the last inclusion comes from theorem \ref{theoremerecouvrement}). With the identity 2'. in proposition \ref{propositioninclusionpim}, we obtain
\[
\gamma_{m_j}(t) \in \Pi^{m_j} (\Phi_{i(k)}^{m_j}(D(\alpha))) \subset \Phi_{i(k)}^\infty(D(7\alpha/6)),
\]
and this ends the proof of fact \ref{faitconvergencefinal}.
\end{proof}
After passing to a subsequence of $(m_j)$, we may also assume that for every $k\in \{0,...,n\}$,
\[
\gamma_{m_j}(k/n) \underset{j\rightarrow \infty} \longrightarrow \alpha_{k/n} \in \overline{\Phi_{i(k)}^\infty(D(7\alpha/6))} \subset \Phi_{i(k)}^\infty (D(4\alpha/3)),
\]
where we have $\alpha_0=x$ and $\alpha_1=y$. For every $k \in \{0,...,n-1\}$ we have $\alpha_{k/n},\alpha_{(k+1)/n} \in \Phi_{i(k)}^\infty(D(4\alpha/3))$, and proposition \ref{propositionconvergencelocaledistance} gives
\[
\widetilde{d_{m_j}} (\gamma_{m_j}(k/n), \gamma_{m_j}((k+1)/n))  \underset{j \rightarrow \infty} \longrightarrow \widetilde{d^{i(k)}} (\alpha_{k/n},\alpha_{(k+1)/n}).
\]
Since the curves $\gamma_{m_j}$ are minimizing geodesics, we have
\[
\widetilde{d_{m_j}}(x,y) = \sum_{k=0}^{n-1} \widetilde{d_{m_j}} (\gamma_{m_j}(k/n), \gamma_{m_j}((k+1)/n)),
\]
hence when $j$ goes to infinity we obtain
\begin{eqnarray*}
\liminf_{m \rightarrow \infty} \widetilde{d_m}(x,y) & = & \sum_{k=0}^{n-1} \widetilde{d^{i(k)}} (\alpha_{k/n},\alpha_{(k+1)/n}) \\
 & = & \sum_{k=0}^{n-1} \limsup_{m\rightarrow \infty} \widetilde{d_m}(\alpha_{k/n},\alpha_{(k+1)/n}) \\
 & \geq & \limsup_{m\rightarrow \infty} \sum_{k=0}^{n-1} \widetilde{d_m}(\alpha_{k/n},\alpha_{(k+1)/n}) \\
 & \geq & \limsup_{m\rightarrow \infty} \widetilde{d_m}(x,y).
\end{eqnarray*}
Hence $\lim_{m\rightarrow \infty} \widetilde{d_m}(x,y)$ exists in $[0,+\infty]$, and this limit is finite since $\widetilde{d_m}(x,y) \leq D$. This ends the proof of proposition \ref{propositionexistencelimite}.
\end{proof}

We know prove the \emph{uniform} convergence of $(\widetilde{d_m})$ to $\widetilde{d}$:

\begin{corollary} \label{corollaireconvergenceuniformedistance}
Let $(x_m)$ and $(y_m)$ be two sequences in $\Sigma^\infty$, such that $x_m \rightarrow x \in \Sigma^\infty$ and $y_m \rightarrow y \in \Sigma^\infty$. Then
\[
\widetilde{d_m}(x_m,y_m) \underset{m \rightarrow \infty} \longrightarrow \widetilde{d}(x,y).
\]
\end{corollary}
\begin{proof}
We have
\begin{eqnarray*}
|\widetilde{d_m}(x_m,y_m) - \widetilde{d}(x,y)| & \leq & |\widetilde{d_m}(x_m,y_m) - \widetilde{d_m}(x,y)| + |\widetilde{d_m}(x,y) - \widetilde{d}(x,y)| \\
 & \leq & \widetilde{d_m}(x_m,x) + \widetilde{d_m}(y_m,y) + |\widetilde{d_m}(x,y) - \widetilde{d}(x,y)|.
\end{eqnarray*}
By definition, $|\widetilde{d_m}(x,y) - \widetilde{d}(x,y)|$ goes to zero. And if $i \in \{1,...,N\}$ is such that $x \in \Phi_i^\infty(D(4\alpha/3))$, then proposition \ref{propositionconvergencelocaledistance} shows that $\widetilde{d_m}(x_m,x) \rightarrow \widetilde{d}^i(x,x)=0$. For the same reason $\widetilde{d_m}(y_m,y) \rightarrow 0$, and this ends the proof.
\end{proof}

To finish the proof of the Main theorem, we need to show that $\widetilde{d}$ is a metric with B.I.C. on $\Sigma^\infty$.

\begin{proposition}
$\widetilde{d}$ is a distance on $\Sigma^\infty$.
\end{proposition}
\begin{proof}
By definition of $\widetilde{d}(x,y)=\lim_{m \rightarrow \infty}\widetilde{d_m}(x,y)$, symmetry and triangular inequality are clear. Now, consider some $x,y\in\Sigma^\infty$ with $\widetilde{d}(x,y)=0$. Then $\widetilde{d_m}(x,y) \rightarrow 0$. For every $m \in \N$, there exists an integer $i(m) \in \{1,...,N\}$ such that
\[
(\Pi^m \circ \Psi^m)^{-1}(x) \in B_m(x_{i(m)}^{m},\kappa \varepsilon/2);
\]
since $i(m)$ belong to a finite set, there exists a subsequence $m_j$ of $(m)$ and an integer $i\in \{1,...,N\}$ such that $i({m_j})=i$ for all $j \in \N$:
\[
(\Pi^{m_j} \circ \Psi^{m_j})^{-1}(x) \in B_{m_j}(x_i^{m_j},\kappa \varepsilon/2).
\]
If $j$ is large enough so that $\widetilde{d_{m_j}}(x,y) \leq \kappa \varepsilon/2$, we have
\[
d_{m_j}((\Pi^{m_j} \circ \Psi^{m_j})^{-1}(x),(\Pi^{m_j} \circ \Psi^{m_j})^{-1}(y)) \leq \kappa \varepsilon/2,
\]
hence
\[
(\Pi^{m_j} \circ \Psi^{m_j})^{-1}(x) \mbox{ and } (\Pi^{m_j} \circ \Psi^{m_j})^{-1}(y) \mbox{ belong to } B_{m_j}(x_i^{m_j},\kappa \varepsilon).
\]
Then as in the end of the proof of fact \ref{faitconvergencefinal} we have
\[
(\Pi^{m_j})^{-1}(x) \mbox{ and } (\Pi^{m_j})^{-1}(y) \mbox{ belong to } \Psi^{m_j}(B_{m_j}(x_i^{m_j},\kappa \varepsilon)),
\]
and we have
\[
\Psi^{m_j}(B_{m_j}(x_i^{m_j},\kappa \varepsilon)) = \Phi_i^{m_j} (H_i^{m_j} (B_{m_j}(x_i^{m_j},\kappa \varepsilon))) \subset \Phi_i^{m_j}(D(\alpha))
\]
(the last inclusion comes from theorem \ref{theoremerecouvrement}). Hence we obtain
\[
x \mbox{ and } y \mbox{ belong to } \Pi^{m_j} (\Phi_i^{m_j}(D(\alpha))) \subset \Phi_i^\infty(D(7\alpha/6))
\]
(the last inclusion comes from the identity 2'. in proposition \ref{propositioninclusionpim}). Since $x,y \in \Phi_i^\infty(D(4\alpha/3))$, we can apply proposition \ref{propositionconvergencelocaledistance} to obtain $\widetilde{d_{m_j}}(x,y) \rightarrow \widetilde{d}^i(x,y)$, so $\widetilde{d^i}(x,y)=0$ and $x=y$.
\end{proof}

The fact that $\widetilde{d}$ is an intrinsic distance comes from the following lemma, which has its own interest:
\begin{proposition}
Let $x,y \in \Sigma^\infty$ and $\gamma_m : [0,1] \rightarrow \Sigma^\infty$ be minimizing geodesics for the metric $\widetilde{d_m}$, between $x$ and $y$ (with $\gamma_m(0)=x$ and $\gamma_m(1)=y$), and parametrized with constant speed. Then there exists a subsequence $(m_j)$ of $(m)$ such that $\gamma_{m_j}$ converges uniformly to a continuous curve $\gamma : [0,1] \rightarrow \Sigma^\infty$, and for the metric $\widetilde{d}$, $\gamma$ is a minimizing geodesic between $x$ and $y$.
\end{proposition}
\begin{proof}
We adapt the proof of Arzela-Ascoli's lemma to our setting. We have
\[
\widetilde{d_m}(\gamma_m(t),\gamma_m(t'))= \widetilde{d_m}(x,y) \cdot |t-t'| \leq D \cdot |t-t'|.
\]
Now let $\{t_k, k \in \N\}$ be a dense subset of $[0,1]$: by a diagonal argument, we can construct a subsequence $(m_j)$ of $(m)$ such that for every $k \in \N$, $\gamma_{m_j}(t_k) \rightarrow_{j\rightarrow \infty} \alpha_k \in \Sigma^\infty$. Let $\gamma : \{t_k, k \in \N\} \rightarrow \Sigma^\infty$ be the map defined by $\gamma(t_k) := \alpha_k$. We have
\[
\widetilde{d_{m_j}}(\gamma_{m_j}(t_{k_1}),\gamma_{m_j}(t_{k_2})) \leq D \cdot |t_{k_1}-t_{k_2}|,
\]
and when $j$ goes to infinity, with corollary \ref{corollaireconvergenceuniformedistance} we get
\[
\widetilde{d}(\gamma(t_{k_1}),\gamma(t_{k_2})) \leq D \cdot |t_{k_1} - t_{k_2}|.
\]
$\gamma$ is then Lipschitz on $\{t_k, k \in \N\}$, which is dense in $[0,1]$, so there exists a (unique) Lipschitz extension $\gamma : [0,1] \rightarrow \Sigma^\infty$. Then, when $j$ goes to infinity,
\[
\gamma_{m_j} \mbox{ converges uniformly on } [0,1] \mbox{ to  } \gamma.
\]
Indeed, let $(u_{m_j})$ be a sequence in $[0,1]$ such that $u_{m_j} \rightarrow u \in [0,1]$: we want to show that $\gamma_{m_j}(u_{m_j}) \rightarrow \gamma(u)$ when $j$ goes to infinity. Let $\varepsilon>0$, and suppose $j$ is large enough so that $\widetilde{d} \leq \widetilde{d_{m_j}} + \varepsilon$ on $\Sigma^\infty \times \Sigma^\infty$. Take some $k \in \N$ such that $|t_k - u| \leq \varepsilon/D$. Then we have
\begin{eqnarray*}
\widetilde{d}(\gamma_{m_j}(u_{m_j}),\gamma(u)) & \leq & \widetilde{d}(\gamma_{m_j}(u_{m_j}),\gamma_{m_j}(t_k)) + \widetilde{d}(\gamma_{m_j}(t_k),\gamma(t_k)) + \widetilde{d}(\gamma(t_k),\gamma(u)) \\
 & \leq & \widetilde{d_{m_j}}(\gamma_{m_j}(u_{m_j}),\gamma_{m_j}(t_k)) + \varepsilon + \widetilde{d}(\gamma_{m_j}(t_k),\gamma(t_k)) + \widetilde{d}(\gamma(t_k),\gamma(u)) \\
 & \leq & D \cdot |u_{m_j}-t_k| + \varepsilon + \widetilde{d}(\gamma_{m_j}(t_k),\gamma(t_k)) + D \cdot |t_k-u| \\
 & \leq & D \cdot |u_{m_j}-t_k| + \varepsilon + \widetilde{d}(\gamma_{m_j}(t_k),\gamma(t_k)) + \varepsilon,
\end{eqnarray*}
and the right-hand side is $\leq 3\varepsilon$ if $j$ is large enough, so $\gamma_{m_j}(u_{m_j})$ converges to $\gamma(u)$ as $j$ goes to infinity.

Finally, the metric $\widetilde{d}$, $\gamma$ is a minimizing geodesic between $x$ and $y$: for every subdivision $0=\lambda_1 \leq \lambda_2 \leq ... \leq \lambda_p=1$ of $[0,1]$, since $\gamma_{m_j}$ is a minimizing geodesic we have
\[
\widetilde{d_{m_j}}(x,y) = \sum_{k=0}^{p-1} \widetilde{d_{m_j}}(\gamma_{m_j}(\lambda_k), \gamma_{m_j}(\lambda_{k+1})).
\]
When $j$ goes to infinity we get
\[
\widetilde{d}(x,y) = \sum_{k=0}^{p-1} \widetilde{d}(\gamma(\lambda_k), \gamma(\lambda_{k+1})),
\]
and this proves the claim, taking the supremum over all subdivisions $0=\lambda_1 \leq \lambda_2 \leq ... \leq \lambda_p=1$ of $[0,1]$: $\widetilde{d}(x,y)$ is equal to the $\widetilde{d}-$length of the curve $\gamma$.
\end{proof}

The next proposition finishes the proof of the Main theorem.

\begin{proposition}
 $\widetilde{d}$ is a metric with B.I.C. on $\Sigma^\infty$.
\end{proposition}

\begin{proof}
$\widetilde{d}$ is an intrinsic metric, and is compatible with the topology of $\Sigma^\infty$: for every $\varepsilon>0$, if $m$ is large enough, for any $x \in \Sigma$ we have
\[
B_{\widetilde{d}}(x,\varepsilon/2) \subset B_{\widetilde{d_m}}(x,\varepsilon) \subset B_{\widetilde{d}}(x,2\varepsilon)
\]
(with obvious notations), and the metric $\widetilde{d_m}$ is compatible with the topology of $\Sigma^\infty$, so $\widetilde{d}$ is also compatible with the topology of $\Sigma^\infty$.

And for every $\varepsilon>0$, consider some $m \in \N$ such that $||\widetilde{d_m} - \widetilde{d}||_{\infty} \leq \varepsilon$. Since $\widetilde{d_m}$ is a metric with B.I.C., there exists some smooth Riemannian metric $g$ on $\Sigma^\infty$ with $||\widetilde{d_m}-d_g|| \leq \varepsilon$, and with $\int_{\Sigma^\infty} |\K_g| d\mathcal{A}_g \leq \Omega+1$. We then have $||\widetilde{d}-d_g|| \leq 2\varepsilon$: $\widetilde{d}$ can be uniformly approximate by Riemannian metrics, with $\int_\Sigma |\K_g| d\mathcal{A}_g$ bounded, hence $\widetilde{d}$ is a metric with B.I.C. (see definition \ref{definitionsurfacesBICbyapproximation}).
\end{proof}

\section*{Appendix: conformal geometry of an annulus}

The material presented here is standard; we recall it to fix the notations we use in this article. A classical reference is the book of L. Ahlfors $\cite{Ahlfors}$.

\begin{definition} \label{definitionannulus}
 A \emph{topological annulus} is a bounded open set of the plane $U$, such that $\C-U$ has only one bounded component, and this component is not reduced to a point.
\end{definition}

Every topological annulus $U$ is conformal to a regular annulus
\[
A(R_1,R_2) = \{ z \in \mathbb{C} | ~ R_1 < |z| < R_2 \},
\]
for some $0< R_1 < R_2<\infty$, and the ratio $R_2/R_1$ is uniquely determined by $U$ (see \cite{Ahlfors2}).

\subsection*{Modulus of a topological annulus $U$}

Let $U$ be a topological annulus. Let $\Gamma$ be the set of continuous simple curves $\gamma : [0,l] \rightarrow U$, parametrized by arc-length (that is for every $t_1 \leq t_2$, the Euclidean length of $\gamma_{|[t_1,t_2]}$ is $t_2-t_1$), joigning the bounded and the unbounded components of $\C-U$: that is, $\gamma(0)$ (resp., $\gamma(1)$) belongs to the bounded (resp., unbounded) component of $\C-U$, and $\gamma(t) \in U$ for $t \in (0,1)$. If $\rho : U \rightarrow [0,+\infty]$ is a mesurable map, we define the $\rho-$length of $\gamma$ by
\[
L_\rho (\gamma) := \int_\gamma \rho |dz|= \int_0^l \rho(\gamma(t)) dt,
\]
and the $\rho-$area of $U$ by
\[
A_\rho(U) := \iint_U \rho^2 d\lambda.
\]
These are the length of $\gamma$ (resp., the area of $U$) for the (singular) Riemannian metric $g=\rho^2 |dz|^2$. We define the modulus of $U$ as follows:
\[
\module(U) := \sup_{\rho} \frac{\inf_{\gamma \in \Gamma} L_\rho (\gamma)^2}{A_\rho (U)},
\]
where the supremum is taken over all mesurable maps $\rho$ with $0 < A_\rho(U) < +\infty$.

The modulus of an annulus is a conformal invariant, and it measures the "thickness" of the annulus: if $U$ and $U'$ are two annuli with $U \subset U'$, then $\module(U) \leq \module(U')$. For example for a regular annulus $U=A(R_1,R_2)$ we have $\module(U) = \frac{1}{2\pi} \ln(R_2/R_1)$. 

\subsection*{The Grötzsch annulus}

Let $0<r<1$. The \emph{Grötzsch annulus} is the following topological annulus:
\[
G(r) := D(1) - [0,r].
\]
\begin{center}
\begin{tikzpicture}[scale=2]
\draw (0,0) circle (1) ;
\draw (0,0) -- (0.6,0) ;
\draw (0,0) node {$\bullet$};
\draw (0,0) node [below] {0};
\draw (0.6,0) node {$\bullet$};
\draw (0.6,0) node [below] {$r$};
\draw (0,0.5) node {$G(r)$};
\end{tikzpicture}
\captionof{figure}{}
\end{center}

It is also defined as the set of complex numbers $z$ with $|z|>1$ and $z \notin [1/r,+\infty)$; we can pass from one definition to the other by the conformal map $z \mapsto 1/z$.

The map $r \in (0,1) \mapsto \module (G(r))$ is decreasing, and we have $\lim_{r \rightarrow 0^+} \module (G(r)) = +\infty$ and $\lim_{r \rightarrow 1^-} \module (G(r)) = 0$. We will use the following theorem:
\begin{theorem}[Grötzsch] \label{theoremegrotzsch}
Let $U \subset D(1)$ be a topological annulus, not containing 0 and $r$. Then
\[
\module (U) \leq \module (G(r)).
\]
\end{theorem}
If the annulus $U$ does not intersect the whole line segment $[0,r]$, then $U \subset G(r)$ and the theorem is useless. But in the following case, the inequality $\module (U) \leq \module (G(r))$ is not obvious:

\begin{center}
\begin{tikzpicture}[scale=2]
\draw (0,0.5) node {$U$};
\draw (0,0) circle (1) ;
\draw (0,0) node {$\bullet$};
\draw (0,0) node [below] {0};
\draw (0.6,0) node {$\bullet$};
\draw (0.6,0) node [below] {$r$};
\draw[-] (-0.15,0) to[out=90,in=90] (0.1,0.2);
\draw[-] (0.1,0.2) to[out=270,in=180] (0.2,-0.2);
\draw[-] (0.2,-0.2) to[out=0,in=180] (0.3,-0.1);
\draw[-] (0.3,-0.1) to[out=0,in=270] (0.5,-0.15);
\draw[-] (0.5,-0.15) to[out=90,in=180] (0.5,0.3);
\draw[-] (0.5,0.3) to[out=0,in=90] (0.8,-0.1);
\draw[-] (0.8,-0.1) to[out=270,in=0] (0.5,-0.4);
\draw[-] (0.5,-0.4) to[out=180,in=270] (-0.2,-0.2);
\draw[-] (-0.2,-0.2) to[out=90,in=270] (-0.15,0);
\end{tikzpicture}
\captionof{figure}{}
\end{center}


\end{document}